\documentclass[12pt,
]{article}

\usepackage[utf8]{inputenc}
\usepackage[T1]{fontenc}
\usepackage[english]{babel}
\usepackage{amsmath,amssymb,amsthm}
\usepackage{mathrsfs,dsfont} %fonts
\usepackage{mathtools, ifthen, xparse, tikz-cd}
\usepackage[shortlabels]{enumitem}
\usepackage{csquotes}
\MakeOuterQuote{"}
\usepackage[colorlinks]{hyperref}
\usepackage[capitalize]{cleveref}
\usepackage{comment}
\usepackage{multicol}
\usepackage[misc]{ifsym}
\usepackage{setspace}
\usepackage[margin=2.5cm]{geometry}

\def\RR{{\mathbb R}}

\def\d{{\mathrm{d}}}

\DeclareMathOperator{\tr}{tr}

\DeclareMathOperator{\Id}{I}

\newcommand{\hide}[1]{}

\newtheorem{thm}{Theorem}
\numberwithin{thm}{section}
\newtheorem*{thm*}{Theorem}
\newtheorem{lem}[thm]{Lemma}
\newtheorem{cor}[thm]{Corollary}
\newtheorem{prop}[thm]{Proposition}

\newtheorem*{prob*}{Problem}
\newtheorem{defin}[thm]{Definition}
\newtheorem{rem}[thm]{Remark}

\theoremstyle{definition}
\newtheorem{assum}{Assumption}
\newtheorem{ex}[thm]{Example}
\theoremstyle{remark}
\newtheorem*{ex*}{Example}
\newtheorem*{rem*}{Remark}
\newtheorem{question}{Question}

\title{Quantum Harmonic Analysis on locally compact abelian groups}
\author{Robert Fulsche and Niklas Galke}
\date{\emph{Dedicated to Karlheinz Gr\"{o}chenig on the occasion of his 65th birthday.}}

\begin{document}

\maketitle

\begin{abstract}
    We extend the notions of \emph{quantum harmonic analysis}, as introduced in R.~Werner's paper from 1984 (\emph{J.~Math.~Phys.~25(5)}), to abelian phase spaces, by which we mean a locally compact abelian group endowed with a Heisenberg multiplier. In this way, we obtain a joint harmonic analysis of functions and operators for each such phase space. Throughout, we spend significant extra effort to include also phase spaces which are not second countable. We obtain most results from Werner's paper for these general phase spaces, up to Wiener's approximation theorem for operators. In addition, we extend certain of those results (most notably Wiener's approximation theorem) to operators acting on certain coorbit spaces affiliated with the phase space.


\end{abstract}

{\setstretch{.75}
\tableofcontents}

\section{Introduction}
The term \emph{quantum harmonic analysis} (QHA), without any further explanation, is too extensive to refer to any particular mathematical theory. In principle, this could refer to any piece of mathematics relating ideas from harmonic analysis with the description of quantum mechanics or quantum field theory. There has been an enormous amount of work on these subjects, some of which may seem very much unrelated. So in this perspective the title of the present work appears quite generic. To be more precise, we will deal with quantum harmonic analysis in the sense of the paper \cite{Werner1984} by R.\ F.\ Werner. While we see this particular paper as the beginning of quantum harmonic analysis in the strict sense we consider here, this theory is of course connected to many of the contributions to the mathematics of quantum mechanics that were worked out before. The interested reader will notice that many well-known ideas can also be found in the present work. For example, aspects of the works of Stone, von Neumann \cite{vonNeumann1931} and Mackey \cite{Mackey1958} on the CCR relations, the work of Pool on the Weyl transform \cite{Pool1966}, papers of Kastler, Loupias and Miracle-Sole \cite{Kastler1965, Loupias_MiracleSole1966, Loupias_MiracleSole1967} as well as Daubechies \cite{Daubechies1980} have obvious influence on the contents of the present work. 

Let us - very briefly - summarize the ideas from \cite{Werner1984}. Based on the class of unitary operators $W_z$, acting on $L^2(\mathbb R^n)$ and satisfying the canonical commutation relations $W_z W_w = e^{i\sigma(z, w)/2}W_{z+w}$, where $z, w$ are points from the phase space $\Xi = \mathbb R^{2n}$ and $\sigma$ is the standard symplectic form on $\Xi$, Werner introduced a notion of \emph{convolution of functions and operators}, which extends the usual convolution of functions, and turns $L^1(\mathbb R^{2n}) \oplus \mathcal T^1(L^2(\mathbb R^n))$ into a commutative Banach algebra, which properly contains $L^1(\mathbb R^{2n})$ as a subalgebra. Here, $\mathcal T^1(L^2(\mathbb R^n))$ is the ideal of trace class operators in the underlying Hilbert space. As a tool to study the harmonic analysis on this algebra, he used the Fourier-Weyl transform (also called Fourier-Wigner transform), which plays the same role for operators as the symplectic Fourier transform for functions. Werner then obtained results for operators analogous to the classical Riemann-Lebesgue lemma, the convolution theorem, Plancherel's theorem, and the Hausdorff-Young inequality. Besides this, he also obtained a version of Wiener's approximation theorem for operators, resulting in his \emph{correspondence theory}. While these ideas, having been published in 1984, were ignored by mathematicians for many years, the recent years showed that these ideas are indeed very useful from a mathematical point of view, leading to many important applications in time-frequency analysis and operator theory \cite{Bauer_etAl2023, Berge_Berge_Fulsche,fulsche2020correspondence, Fulsche_Rodriguez2023, Fulsche_Hagger2023, luef_skrettingland2018, Luef_Skrettingland2018a, Luef_Skrettingland2019a, Luef_Skrettingland2021}. All of the previously cited works have in common that they solely work with the phase space $\Xi = \mathbb R^{2n}$.
The operators $W_z$ satisfying the canonical commutation relations associated to this phase space act as the position and momentum shifts of $n$-mode continuous variable quantum systems and as such play a role in several areas within quantum physics.
A much used tool for studying these systems is a quantum state's \emph{Wigner function} -- that is the symplectic Fourier transformation applied to the Fourier-Weyl transform of a density operator.
The inverse, provided it is well-defined, then takes a Wigner function to a density operator.
Correspondence theory deals with a similar but different question, closely related to the \emph{Weyl quantization}:
Even though density operators are positive semi-definite, the Wigner function is not necessarily positive and thus not a probability distribution (despite being normalized).
\emph{Correspondence rules} on the other hand are exactly the normal quantum channels -- hence, by definition, positivity-preserving -- between $L^\infty(\Xi)$ and $\mathcal L(L^2(\RR^{n}))$ that are covariant with respect to the action of $W_z$.
They are given by convolutions with a density operator while convolutions with a probability density give rise to covariant classical and quantum channels.
The product formula for the Fourier(-Weyl) transform then establishes a certain connection between correspondence rules and Wigner functions.
Replacing the phase space by a more general group gives rise to systems with different physical meaning that still let us apply above ideas. Obviously, considering such ideas for periodic systems, i.e., with respect to the circle group, have practical relevance. On the other hand, from the point of view of quantum information theory the relevant cases of this are those of qudits with the \emph{generalized Pauli matrices} satisfying the commutation relations of two copies of a finite cyclic group.
Works from quantum information theory and quantum computation with connection to above concepts -- in particular Wigner functions -- also for $\Xi\neq \RR^{2n}$ are, e.g., \cite{Kiukas_Lahti_Schultz_Werner2012, Gibbons_Hoffman_Wootters_2004, Gross_2008, Raussendorf_Browne_Delfosse_Okay_BermejoVega_2017, Fawzi_Oufkir_Salzmann2024}.

While there have been some abstract works on other phase spaces (see \cite{Berge_Berge_Luef2022, Berge_Berge_Luef_Skrettingland2022} for some results on the affine group and the paper by Halvdansson \cite{Halvdansson2022} for some recent results on the operator convolutions on arbitrary locally compact groups) and some other works which clearly are related to the phase space concepts of QHA on more general lca groups (see, e.g., \cite{Feichtinger_Kozek1998}), there is still a lot of work to be done in the direction of more general phase spaces. The present work is a contribution in this direction: We will present the theory for the case where $\Xi$ is a locally compact abelian group (which we will usually abbreviate as \emph{lca group}). We already now want to mention that we do not need to assume second countability of the phase space, an assumption which is otherwise frequently encountered in related works.

In contrast to classical harmonic analysis, where only an (abelian) locally compact group is given as the structure to be studied, quantum harmonic analysis needs the structure of a \emph{phase space}. In our setting, this will mean that the lca group $\Xi$ is given together with a \emph{Heisenberg multiplier}, which is a cocycle on $\Xi$ with certain properties, most importantly that it induces a self-duality, i.e.\ a group isomorphism of $\Xi$ with $\widehat{\Xi}$, the Pontryagin dual. Affiliated with each such phase space is a unique projective unitary representation, having the Heisenberg multiplier as its cocycle. Starting from this data, we will start our investigations of quantum harmonic analysis. Certainly, many more results can be studied within this rather general framework, but in the present work we restricted ourselves towards discussing the theory up to Wiener's approximation theorem for operators.  

The prototypical example of such a phase space is certainly $\Xi = G \times \widehat{G}$, where $G$ is any locally compact abelian group and $\widehat{G}$ the dual group. The canonical Heisenberg multiplier is $m((x, \xi), (y, \eta)) = \langle x, \eta\rangle$. There are many multipliers similar to this one, and we will give a list of examples. While this class is certainly our leading example and the main motivation for the authors to study the present structures, the class of phase spaces to which our theory applies is strictly larger than this (this was shown in \cite{Prasad_Shapiro_Vemuri2010}).  

Let us give a rough description of the contents of the present work. In Section \ref{sec:repr}, we will collect all the results needed on projective unitary representations for our work. No new results are proven here, and no significant proofs are demonstrated, but we give references to proofs of all important results we need. 

Section \ref{sec:Gelfand_Pettis} is a short reminder on some facts from integration theory. The purpose of this section is two-fold: On the one hand, we want to describe the essentials on vector-valued integration that we will need during the rest of the paper. On the other hand, we will address certain issues that arise for Haar measures on groups that are not second countable. 

In Section \ref{sec:conv}, we really start with the heart of the matter. We introduce the concepts of function-operator and operator-operator convolutions and discuss their main properties. This part of the theory is rather similar to the special case $\Xi = \mathbb R^{2n}$ discussed by Werner in \cite{Werner1984}, so we are sometimes rather brief about proofs whenever nothing strange happens. All the results turn out as expected. 

Section \ref{sec:positivity} contains a rather detailed discussion on positivity results for the convolutions, as well as the theorem on positive correspondence rules, which was one of the main motivations for Werner's initial work on the matter. 

It is Section \ref{sec:fouriertrafo}, where we bring the Fourier transform of operators into play, where some interesting effects connected to the group structure of the phase space occur. Indeed, some of the desired results, most prominently Bochner's theorem for operators, is not obtained in full generality, but only for a certain class of phase spaces. Nevertheless, Wiener's approximation theorem for operators, which is one of the author's main motivation to study quantum harmonic analysis, is obtained in full generality.

In Section \ref{sec:coorbit}, we will describe how to extend certain parts of the theory to operators acting on coorbit spaces. Most prominently, we will give a version of Wiener's approximation theorem in this generalized setting. This approach has already been proven useful in the first-named author's investigation of Toeplitz operators on the Fock space \cite{fulsche2020correspondence, Fulsche2022}, and we expect more interesting applications in operator theory of this approach.

Finally, in Section \ref{sec:discussion} we end with mentioning some open problems related to the present work.

\section{Projective representations}\label{sec:repr}

It turns out convenient to allow for more general functions than the $\sigma$ given above. For this, we will assume in the following that $\Xi$ is any locally compact abelian group. We will deal with certain projective unitary representations of $\Xi$. Recall that a projective unitary representation of $\Xi$ is a map $\rho: \Xi \mapsto \mathcal U(\mathcal H)$, where $\mathcal H$ is some (complex) Hilbert space and $\mathcal U(\mathcal H)$ the unitary group over $\mathcal H$, satisfying $\rho(x)\rho(y) = m(x,y)\rho(x+y)$. Further, the coefficients $\langle \rho(x)\varphi, \psi\rangle$, $\varphi, \psi \in \mathcal H$, are assumed to be measurable functions of $x$. The function $m$ takes values $m(x,y) \in S^1 = \{ z \in \mathbb C: ~|z| = 1\}$. In our convention, the inner product $\langle \cdot, \cdot\rangle$ on $\mathcal H$ is linear in the first and anti-linear in the second argument.  We will usually write $U_x = \rho(x)$. Therefore, it is:
\begin{align*}
    U_x U_y = m(x,y) U_{x+y}, \quad x,y \in \Xi.
\end{align*}
The function $m: \Xi \times  \Xi \to \mathbb T$ is called the \emph{multiplier} (or \emph{cocycle}) of the projective representation. It satisfies
\begin{align*}
m(x+y,z) m(x,y) = m(x,y+z)m(y,z), \quad x,y,z \in \Xi.
\end{align*}
We will always assume that $m(x,0) = m(0,x) = 1$ for every $x \in \Xi$, or equivalently $U_0 = I$. Projective representations are usually considered under certain measurability conditions on the coefficients $x \mapsto \langle U_x \varphi, \psi\rangle$ for $\varphi, \psi \in \mathcal H$. We will need to work with a projective representation which is continuous in strong (or equivalently: weak) operator topology, which imply the above usual measurabilty conditions. It is not hard to verify that, under this condition, the multiplier $m$ is separately continuous, which we will always assume from now on. Note that it is $U_x U_{-x} = m(x,-x) \Id$. Since $U_x$ is unitary, this yields $U_x^\ast = \overline{m(x,-x)}U_{-x}$.

There will be one more critical assumption on $m$, namely that it is a \emph{Heisenberg multiplier} \cite{Digernes_Varadarajan2004}, a notion that we are going to recall now. The map
\begin{align}
    \sigma(x,y) = \frac{m(x,y)}{m(y,x)}, \quad x, y \in \Xi
\end{align}
satisfies
\begin{align*}
    \sigma(x,y+z) = \sigma(x,y)\sigma(x,z), \quad \sigma(x,y) = \frac{1}{\sigma(y,x)}, \quad \sigma(x,x) = 1.
\end{align*}
This can be summarized as saying that $\sigma$ is an \emph{alternating bicharacter}. Using the continuity properties of $m$, the map
\begin{align*}
    \Xi \ni x \mapsto \sigma(x, (\cdot)) \in \widehat{\Xi}
\end{align*}
is a continuous morphism from $\Xi$ to $\widehat{\Xi}$, the Pontryagin dual of $\Xi$. We say that $m$ is a \emph{Heisenberg multiplier} if this map is a topological isomorphism, and in this case $\sigma$ is called the \emph{symplectic form}. Hence, $\Xi$ is self-dual and the self-duality is induced from $\sigma$. Note that $\sigma$ arises naturally from the representation $\rho$, as
\begin{align}
U_x U_y = \sigma(x,y)U_y U_x.
\end{align}
The importance of Heisenberg multipliers stems from the fact that they are immediately linked to a rather general form of the Mackey-Stone-von Neumann theorem:
\begin{thm}\label{Mackey-Stone-vNeumann}
Let $\Xi$ be a locally compact abelian group and $m$ a separately continuous Heisenberg multiplier of $\Xi$. Then, there exists an irreducible projective representation with multiplier $m$ which is unique up to unitary equivalence.
\end{thm}
A proof can be found in \cite[Theorem 3.3]{Baggett_Kleppner1973}. Note that a Heisenberg multiplier is totally skew and type I in the notion of Baggett and Kleppner, cf.~\cite[Theorem 3.2]{Baggett_Kleppner1973}.

Another property of the representation we will need is that of \emph{square integrability}. Recall that an irreducible projective representation $(\mathcal H, \rho)$ is called square integrable if there are $0 \neq \varphi, \psi \in \mathcal H$ such that $x \mapsto \langle \rho(x)\varphi, \psi\rangle \in L^2(\Xi)$. On a more general level, the irreducible $m$-representations which are square integrable are precisely those which appear as subrepresentations of the $m$-regular representation, i.e., the representation on $L^2(\Xi)$ defined as as $\lambda(x)f(y) = m(x,y)f(y+x)$ (a proof for ordinary unitary representations can be found in \cite[Theorem VII.1.5]{Gaal1973}, which can easily be adapted to projective representations). Square integrability is indeed a crucial property for everything that follows, so we will always assume it. Indeed, it is not entirely clear if the unique $m$-representation is automatically square integrable, so we impose this as an assumption. Given the assumptions of Theorem \ref{Mackey-Stone-vNeumann}, it is not hard to see that square integrability of the unique irreducible $m$-representation is equivalent to the following statement: The $m$-regular representation contains an irreducible subrepresentation. It seems to be a non-trivial problem to decide if this statement is true, which on this level seemingly has not been answered in the literature, and we leave this as an open problem. But note that, if this property is satisfied, another important property of the representation is satisfied: From separate continuity of the multiplier $m$ one can easily deduce that the $m$-regular representation acts strongly continuous (for $\varphi \in \mathcal L^2(\Xi)$ the function $\Xi \ni x \mapsto \lambda(x)\varphi$ is continuous). Hence, the same is true for its irreducible subrepresentation, i.e., the representation we obtain from Theorem \ref{Mackey-Stone-vNeumann} is continuous in strong operator topology.

The following is a classical result by Godement \cite{Godement1947}\footnote{Initially, this result has been formulated only for unitary representations, i.e., $m(x,y) = 1$. The simple proof by D.~Shucker \cite{Shucker1983} works verbatim also for projective unitary representations.}. Godement's result works for any unimodular group, but we shall only formulate it in our abelian case. Note that notions such as ``$\d x$'', ``$\d \xi$'', ``$\d g$'' always refer to integration with respect to a suitably normalized Haar measure.
\begin{thm}\label{thm:moyal}
    Let $(\mathcal H, \rho)$ be an irreducible projective unitary representation of the lca group $\Xi$ which is square integrable. Then, for every $\varphi, \psi \in \mathcal H$ it is
    \begin{align*}
        \int_\Xi |\langle \rho(x) \varphi, \psi\rangle|^2~dx < \infty.
    \end{align*}
    Further, there exists some constant $\lambda > 0$ such that for every $\varphi_1, \varphi_2, \psi_1, \psi_2 \in \mathcal H$ it is:
    \begin{align*}
        \int_\Xi \langle \rho(x) \varphi_1, \psi_1\rangle \overline{\langle \rho(x) \varphi_2, \psi_2\rangle}~dx = \lambda \langle \varphi_1, \varphi_2\rangle \overline{\langle \psi_1, \psi_2\rangle}.
    \end{align*}
    \end{thm}
Of course, we can normalize the Haar measure on $\Xi$ such that we obtain $\lambda = 1$ in the above formula, which we always do in the following. 

We now have discussed most of the facts we need from representation theory. We will fix the prerequisites on $\Xi$ and $m$ described so far as a first assumption, which will always be in charge in the following.
\begin{assum}\label{assum:1}
Let $\Xi$ be a locally compact abelian group and $m: \Xi \times \Xi \to S^1$ a separately continuous Heisenberg multiplier on it. $(\mathcal H, \rho)$ and $U_x = \rho(x)$ denotes (an element of the equivalence class of) the unique irreducible, projective $m$-representation, which is also assumed to be square integrable. Then, the representation is also continuous in strong operator topology.
\end{assum}
We now provide a list of examples, which all satisfy the above assumptions. 
\begin{ex}\label{main_example}
Our main motivation and leading example to study the presented theory is the case of $\Xi = G \times \widehat{G}$ for a locally compact abelian group $G$ with Pontryagin dual $\widehat{G}$. Together with the multiplier $m((x,\xi), (y, \eta)) = \overline{\langle x, \eta\rangle}$, we obtain a projective representation on $L^2(G)$ given by
\begin{align*}
    U_{(x,\xi)}f(y) = \langle y, \xi\rangle f(y-x), \quad (x, \xi) \in G \times \widehat{G}.
\end{align*}
\end{ex}
Note that, by the results of \cite{Prasad_Shapiro_Vemuri2010}, there are examples of phase spaces $\Xi$, endowed with Heisenberg multipliers, which are not of the above product type $\Xi = G \times \widehat{G}$. While these examples are covered by our theory, we shall not explicitly consider them as examples.
\begin{ex}\label{mod:main_example}
If $m$ is a separately continuous Heisenberg multiplier on $\Xi$ and $a: \Xi \to S^1$ is continuous with $a(e_\Xi) = 1$, then 
\begin{align*}
    m_a(x,y) = \frac{a(x)a(y)}{a(x+y)}m(x,y).
\end{align*}
is also a separately continuous Heisenberg multiplier (indeed, the resulting symplectic forms are the same). If $U$ is the irreducible projective unitary representation of $\Xi$ on $\mathcal H$ with multiplier $m$, then $U_x^a = a(x) U_x$ is the irreducible projective representation with multiplier $m_a$. For being in accordance with Assumption \ref{assum:2}, which we will explain later, we will always assume that $a(x) = a(-x)$ for every $x \in \Xi$ whenever an example of this kind is considered.
\end{ex}
\begin{ex}
When $U$ is the projective representation of $\Xi = G \times \widehat{G}$ with multiplier $m((x, \xi), (y, \eta)) = \langle x, \eta\rangle$ and $\mathcal F: L^2(G) \to L^2(\widehat{G})$ is the Fourier transform, then $\widehat{U}_{(\xi, x)} = \mathcal F^{-1} U_{(x, \xi)} \mathcal F$ is a strongly continuous projective representation of $\widehat{X} = \widehat{G} \times G$ with multiplier $\widehat{m}((\xi, x),(\eta, y)) = \langle x, \eta\rangle$. If $a: \Xi \to \mathbb T$ is continuous, then $\mathcal F^{-1} U_{(x, \xi)}^a \mathcal F = \widehat{U}_{(\xi, x)}^{\widehat{a}}$ with $\widehat{a}(\xi, x) = a(x, \xi)$. We call $\widehat{U}_{(\xi, x)}^{\widehat{a}}$ the representation dual to $U_{(x, \xi)}^a$. Since $\widehat{\Xi} \cong \Xi$, both representations give rise to the same quantum harmonic analysis (up to unitary equivalence). But nevertheless, there are some interesting outcomes of this concept of a dual representation, relating operator theory on $L^2(G)$, $L^2(\widehat{G})$ and quantum harmonic analysis, that we plan to present in future work.
\end{ex}
\begin{ex}
We say that a locally compact abelian group $H$ is 2-regular when the map $H \ni g \mapsto g+g=2g$ is a topological isomorphism of $H$. In this case, denote by $\frac{g}{2}$ the image of $g$ under the inverse map. When $\Xi$ is 2-regular and the multiplier $m$ satisfies any of the three identities
\begin{align*}
    m(x,\frac{x}{2}) m(x,-x) m(-x,-\frac{x}{2}) &= 1,\\
    m(\frac{x}{4}, \frac{x}{4})m(x,-x) m(-\frac{x}{4}, -\frac{x}{4}) &= 1,\\
    m(\frac{x}{2}, x) m(x,-x) m(-\frac{x}{2}, -x) &= 1,
\end{align*}
where $x \in \Xi$ arbitrary and $x/4 = (x/2)/2$, then setting
\begin{align*}
    a(x) = m(x, \frac{x}{2}), \quad a(x) = m(\frac{x}{4}, \frac{x}{4}), \quad a(x) = m(\frac{x}{2}, x),
\end{align*}
according to which of the above identities is satisfied, gives a continuous $a: \Xi \to S^1$ and yields the representation $U_x^a$ satisfying
\begin{align*}
    (U_x^a)^\ast = U_{-x}^a, \quad x \in \Xi.
\end{align*}
When we assume that $\Xi$ is 2-regular, we will always assume that one of the above identities is satisfied, so we can pass to the representation $U_x^a$ satisfying this identity.
\end{ex}
\begin{ex}\label{ex:2regularweyl}
The lca group $G$ is 2-regular if and only if $\widehat{G}$ is 2-regular, which follows immediately from the functorial properties of Pontryagin duality. If this is the case, $\Xi = G \times \widehat{G}$ is 2-regular as well. Then, the multiplier $m((x, \xi), (y, \eta)) = \langle x, \eta\rangle$ satisfies the identities from the previous example. All definitions of the function $a(x,\xi)$ agree in this case: $a(x,\xi) = \langle x, \frac{\xi}{2}\rangle$. The projective unitary representation on $L^2(G)$ is given by
\begin{align*}
    U_{(x,\xi)}^af(y) = a(x,\xi) \langle y, \xi\rangle f(y-x) = \langle \frac{x}{2}+y, \xi \rangle f(y-x).
\end{align*}
This representation is now the representation through Weyl operators, satisfying $(U_{(x, \xi)}^a)^\ast = U_{(-x,-\xi)}$. We want to note that in this particular example, $m_a$ is a bicharacter of $\Xi$.
\end{ex}
\begin{ex}
    Let $G = \mathbb R^d$ with $d \geq 2$ and $\Xi \cong \mathbb R^d \times \mathbb R^d$. Let $B$ be a constant magnetic field on $\mathbb R^d$, i.e.\ a 2-form $B = \sum_{j < k} B_{jk} dx_j \wedge dx_k$, where each $B_{jk}$ is constant on $\mathbb R^d$. Let $A$ be a potential of $B$, i.e.\ $A = \sum_{k=1}^d A_k dx_k$ with $dA = B$, which boils down to $\partial_j A_k - \partial_k A_j = B_{jk}$. Letting
    \begin{align*}
        m((x, \xi), (y, \eta)) = e^{i\frac{\sigma((x, \xi), (y, \eta))}{2} - i\Gamma^B([0, -x, -x-y])},
    \end{align*}
    where $\sigma$ is the standard symplectic form $\sigma((x, \xi), (y, \eta)) = y\xi - x\eta$ and $\Gamma^B([0, 0-x, 0-x-y])$ is the flux of the magnetic field through the triangle $\Delta(0, -x,-x-y)$ with oriented corners $0, -x,-x-y$:
    \begin{align*}
        \Gamma^B([0, -x,-x-y]) = \int_{\Delta(0,-x,-x-y)}B.
    \end{align*}
    Then, $m$ is a Heisenberg multiplier and the CCR relations are realized by the magnetic Weyl operators
    \begin{align*}
        W_{(x, \xi)}^Af(t) = e^{-i\frac{x\xi}{2} + i t\xi - i\Gamma^A([t, t-x]} f(t-x).
    \end{align*}
    Here, $\Gamma^A([t,t-x])$ is the magnetic circulation through the line segment $[t,t-x]$, i.e.\ 
    \begin{align*}
        \Gamma^A([t,t-x]) = \int_{[t,t-x]} A.
    \end{align*}
    The magnetic Weyl operators are the building blocks of the magnetic Weyl calculus. Indeed, many of the results we obtain are closely related to properties of this calculus, we refer to \cite{Mantoiu_Purice2004, Nittis_Lein_Seri} for details. We only want to note that, upon picking magnetic fields which are not constant, one can still consider the magnetic Weyl operators. They are certainly still unitary operators which satisfy a form of the CCR relation. The difficulty now arises from the fact that the identity
    \begin{align*}
        W_{(x, \xi)}^A W_{(y, \eta)}^A = m((x, \xi), (y, \eta)) W_{(x+y, \xi + \eta)}^A,
    \end{align*}
    which is formally still satisfied, yields a cocycle $m$ which takes values in multiplication operators with unimodular symbols. Therefore, $m$ no longer commutes with the Weyl operators. This certainly leads to an interesting structure, which nevertheless is beyond the scope of this work.
\end{ex}

We will want to work with a unitary operator $R \in \mathcal U(\mathcal H)$ such that:
\begin{align}\label{eq:defR}
    U_x R = RU_{-x}, \quad x \in \Xi.
\end{align}
Note that, when such a $R$ is given, we have
\begin{align*}
m(x,y) U_{x+y} &= U_{x}U_y= RU_{-x}U_{-y}R^\ast \\
&= m(-x,-y)RU_{-(x+y)}R^\ast = m(-x,-y)U_{x+y}, \quad x,y \in \Xi,
\end{align*}
i.e.~$m(x,y) = m(-x,-y)$ for all $x,y \in \Xi$. We cannot come up with a single practically relevant situation where this identity is not satisfied. We add this as a minor second assumption:
\begin{assum}\label{assum:2}
The multiplier $m$ satisfies $m(x,y) = m(-x,-y)$ for all $x, y \in \Xi$.
\end{assum}
Indeed, when the above assumption is satisfied, one easily obtains such an operator $R$ that we were looking for: Considering $V_x = U_{-x}$, we have:
\begin{align*}
    V_x V_y = m(x,y) V_{x+y},
\end{align*}
i.e.~$V$ gives a new irreducible projective representation of $\Xi$ on $\mathcal H$. By Theorem \ref{Mackey-Stone-vNeumann}, there exists $R \in \mathcal U(\mathcal H)$ such that $U_x R = RV_x = RU_{-x}$ for all $x \in \Xi$, which is what we wanted. 
\begin{ex}
When $\Xi = G \times \widehat{G}$ with the multiplier $m((x, \xi),(y,\eta)) = \overline{\langle x, \eta\rangle}$, then $Rf(x) = f(-x), f \in L^2(G)$, provides the operator $R$. 
\end{ex}
\begin{ex}
When $R$ is the appropriate operator for the representation $U_x$ and $a: \Xi \to \mathbb T$ is continuous, then $R$ is still the appropriate operator for the representation $U_x^a$, provided that $a(x) = a(-x)$.
\end{ex}
Note that it is
\begin{align*}
    RRU_x = RU_{-x}R = U_x RR, \quad x \in \Xi.
\end{align*}
Therefore, by Schur's lemma, $RR = cI$ for some $c \in \mathbb C$. Since $RR$ is unitary, $|c| = 1$. Replacing $R$ by $R' = \frac{1}{c^{1/2}}R$, with $c^{1/2}$ any square root of $c$, $R'$ becomes self-adjoint and still satisfies the desired identity.

For some parts of the theory, we will need another assumption in charge, namely that the projective representation considered is \emph{integrable}. We recall the following theorem, the proof of which is analogous to \cite[Prop.~14.5.1]{Dixmier_1977}:
\begin{thm}
    Let $(\mathcal H, \rho)$ be an irreducible projective unitary representation of the locally compact abelian group $\Xi$. Then, the following are equivalent:
    \begin{enumerate}[(i)]
        \item There exists some $0 \neq \varphi \in \mathcal H$ such that $x \mapsto \langle \rho(x) \varphi, \varphi\rangle  \in L^1(\Xi)$.
        \item There exists a dense subspace $\mathcal H_1 \subset \mathcal H$ such that $x \mapsto \langle \rho(x)\varphi, \psi\rangle \in L^1(\Xi)$ for all $\varphi, \psi \in \mathcal H_1$.
    \end{enumerate}
\end{thm}
The projective unitary representation $(\rho, \mathcal H)$ is called \emph{integrable} if it satisfies the above equivalent conditions. We will usually write that $U_x = \rho(x)$ is integrable.
\begin{ex}
The representations from Examples \ref{main_example} are integrable. In particular, if we pick a Schwartz-Bruhat function $\varphi \in \mathcal S(G) \subset \mathcal H$, then saying that $(x,\xi) \mapsto \langle U_x\varphi, \varphi\rangle \in L^1(\Xi)$ is the same as saying that $\varphi$ belongs to the Feichtinger algebra $\mathcal S_0(G)$. This is known to be true, see e.g.~\cite[Theorem 9]{Feichtinger1981}.
\end{ex}
\begin{ex}
When $U_x$ is integrable and $a: \Xi \to \mathbb T$ is continuous, then $U_x^a$ is also integrable.
\end{ex}
We indeed do not know if there are any representations within our scope which are not integrable, mostly because phase spaces $\Xi$ which are not of the form $G \times \widehat{G}$ are somewhat conceptually complicated. We will specifically mention when integrability of the representation is needed, as large parts of the theory work without this assumption.

Let us end this part on projective representations by mentioning the following result, which is an adaptation of \cite[Lemma 18]{Dammeier_Werner2023}. While we will not make use of it in this work, we have the feeling that this is the right place to set up the result for future reference. For this result, we will of course assume that all our assumptions are in charge. In principle, this result is an extension of the statement of Schur's lemma.
\begin{prop}
    Let $A \in \mathcal L(\mathcal H)$ such that for every $x \in \Xi$ there exists $c_x \in \mathbb C$ such that $U_x A U_x^\ast = c_x A$. Then, we have $A = bU_z$ for some $b \in \mathbb C$ and $z \in \Xi$.
\end{prop}
\begin{proof}
    There is of course no loss of generality in assuming that $A \neq 0$. We first note that
    \begin{align*}
        c_{x+y}A = U_{x+y}AU_{x+y}^\ast = U_x U_y AU_y^\ast U_x^\ast = c_x c_y A 
    \end{align*}
    such that $c_{x+y} = c_x c_y$. Further, we have
    \begin{align*}
        \| A\| = \|U_x A U_x^\ast\| = \| c_x A\| = |c_x| \| A\|
    \end{align*}
    so that we obtain $|c_x| = 1$ for every $x \in \Xi$. Let $\varphi, \psi \in \mathcal H$ with $\langle A\varphi, \psi\rangle \neq 0$. Then,
    \begin{align*}
       \Xi \ni x \mapsto  c_x \langle A\varphi, \psi\rangle = \langle U_x A U_x^\ast \varphi, \psi\rangle
    \end{align*}
    is continuous such that $\Xi \ni x \mapsto c_x$ has to be continuous. Therefore, we have seen that $[x \mapsto c_x] \in \widehat{\Xi}$. Hence, there exists $z \in \Xi$ with $c_x = \sigma(z, x)$. We now let $B = A U_{-z}^\ast$. Then, 
    \begin{align*}
        U_x B U_x^\ast = U_x A U_{-z}^\ast U_x^\ast = \overline{\sigma(x, -z)}U_x A U_x^\ast U_{-z}^\ast = \frac{\sigma(z, x)}{\sigma(x, -z)} A U_{-z}^\ast.
    \end{align*}
    Since
    \begin{align*}
        \frac{\sigma(z, x)}{\sigma(x, -z)} = \sigma(z, x) \sigma(x, z) = \frac{m(z, x)}{m(x, z)} \cdot \frac{m(x, z)}{m(z,x)} = 1,
    \end{align*}
    we obtain $U_x B U_x^\ast = B$ for every $x \in \Xi$. By Schur's lemma, we arrive at $B = b \Id$ for some $b \in \mathbb C$ and hence
    \begin{align*}
        A = BU_{-z} = b U_{-z}.
    \end{align*}
    This finishes the proof.
\end{proof}

\section{Some facts from integration theory}\label{sec:Gelfand_Pettis}

Large parts of the present paper will make use of certain Banach space-valued integrals. To set the foundation for those occasions, we will now present some details on the two versions of vector-valued integrals that we will make use of: The Bochner integral and the Gelfand integral.  The classical reference for these integrals would be \cite[Section II]{Diestel_Uhl1977}. There, the standing assumption of the measure space being finite is made, which is nevertheless not crucial. For the reader's convenience (and to convince that finiteness of the measure space is indeed not crucial),  we will give some proofs for the properties of these integral.

Let $(\Omega, \mathcal A, \mu)$ be a measure space and let $X$ be a Banach space. A function $f: \Omega \to X$ is called a \emph{simple function} if there are $x_1, \dots, x_n\in X$ and $A_1, \dots, A_n \in \mathcal A$ such that $f = \sum_{j=1}^n x_j \mathbf 1_{A_j}$. $f$ is called $\mu$-measurable if there exists a sequence $(f_n)_{n\in \mathbb N}$ of $X$-valued simple functions such that $\| f_n(t) - f(t)\|_X \to 0$ for $\mu$-almost every $t \in \Omega$. Setting things up in this way, this yields a straightforward generalization of Lebesgue measurability in the vector-valued case, which in turn leads to a vector-valued integral called the \emph{Bochner integral}. We will come to this later, after we have discussed an alternative approach.

Instead of working with the above notion of measurability, one could consider functions which satisfy a weaker property: $f: \Omega \to X$ is said to be \emph{weakly measurable} if for every $\varphi \in X^\ast$, the function $\varphi \circ f$ is measurable. Now, Pettis' measurability theorem \cite[Theorem II.1.2]{Diestel_Uhl1977} states that $f$ is $\mu$-measurable if and only if it is weakly measurable \emph{and} it is $\mu$-essentially separable valued. The latter means that there exists a set $S \in \mathcal A$ with $\mu(S) = 0$ such that $f(\Omega \setminus S)$ is a subset of $X$ which is separable in norm. Weak measurability is of course simple to check. When the space $X$ itself is already separable, the condition of being $\mu$-essentially separable valued is of course trivially satisfied. If this is not the case, this technical condition is in practice sometimes hard to verify. 

The integral based on this notion of weak measurability are the Dunford- and the Pettis integral. The downside of the Dunford integral is that it, in general, takes values in $X''$ (the bidual of $X$). Proving that such an integral takes values in $X$ instead (i.e., saying that the integrand is Pettis integrable), is a nontrivial task, see, e.g., \cite{Talagrand1984}. We will not pursue this road here. 

Fortunately, there is a convenient variation of this integral when the Banach space $X$, in which the integrand takes values, is already a dual Banach space: $X = Y'$. This will mostly be useful when we consider integrals taking values in $\mathcal L(\mathcal H)$ (which is the dual space of $\mathcal T^1(\mathcal H)$, the trace class). In Section \ref{sec:coorbit}, we will replace the Hilbert space by a class of reflexive Banach spaces, where the above duality is still valid: For a reflexive Banach space $Z$, $\mathcal L(Z)$ is the dual of $\mathcal T^1(Z)$ (the nuclear operators). We will come back to this in Section \ref{sec:coorbit}.

Therefore, let now $X = Y'$, $(\Omega, \mathcal A, \mu)$ a measure space and $f: \Omega \to X$. $f$ is called weak$^\ast$ measurable if for every $y \in Y$, $\Omega \ni t \mapsto \langle f(t), y\rangle$ is measurable. Here, $\langle f(t), y\rangle$ denotes the application of the linear functional $f(t)$ to $y$. A weak$^\ast$ measurable function $f$ is called weak$^\ast$ integrable if for every $y \in Y$, $t \mapsto \langle f(t), y\rangle$ is integrable. Here's a fact about weak$^\ast$ (or Gelfand) integrals:
\begin{thm}
    Let $(\Omega, \mathcal A, \mu)$ be a measure space, $X = Y'$ a dual Banach space and $f: \Omega \to X$ weak$^\ast$ integrable. Then, there exists a unique $g \in X$ such that for every $y \in Y$:
    \begin{align*}
        \int \langle f, y\rangle ~d\mu = \langle g, y\rangle.
    \end{align*}
    With this $g$ we set $\int f ~d\mu := g$ and call this the \emph{Gelfand} or \emph{weak$^\ast$} integral of $f$.
\end{thm}
\begin{proof}
    The proof is essentially as in \cite[Lemma II.3.1]{Diestel_Uhl1977} and is given here for the reader's convenience.

    We consider the linear operator $T: Y \to L^1(\mu)$ given by $Ty = \langle f, y\rangle$. Since $f$ is weak$^\ast$ integrable, $T$ is well-defined on all of $Y$. Below, we will show that $T$ is a closed operator, which we already assume now. Then, the closed graph theorem shows that $T$ is a bounded linear operator from $Y$ to $L^1(\mu)$. Hence, the map
    \begin{align*}
        Y \ni y \mapsto \int \langle f, y\rangle~d\mu
    \end{align*}
    is a linear functional on $Y$ with
    \begin{align*}
        \left | \int \langle f, y\rangle ~d\mu  \right | &= \left | \int T(y) ~d\mu  \right |\\
        &\leq \int |Ty|~d\mu = \| Ty\|_{L^1(\mu)}\\
        &\leq \| T\| \|y\|,
    \end{align*}
    i.e., the functional is bounded. When denoting this functional by $g$, we obtain the element of $X$ that we were looking for. Since a bounded linear functional on a Banach space is uniquely determined by its values, $g$ is of course also unique.

    Hence, we are left with proving that the operator $T$ above is closed. Therefore, let $y_n$ be a sequence in $Y$ with $y_n \to y \in Y$ and $Ty_n \to h$ in $L^1(\mu)$. Since $Ty_n$ converges to $h$ in $L^1(\mu)$, there exists a subsequence $Ty_{n_k}$ which converges $\mu$-almost everywhere to $h$. On the other hand, $Ty_n = \langle f, y_n\rangle$ converges everywhere to $\langle f, y\rangle$, as $f(t)$ is a bounded linear functional for every $t \in \Omega$. Hence, also the subsequence $Ty_{n_k}$ converges to $\langle f, y\rangle = Ty$ everywhere. We obtain that $Ty$ and $h$ agree $\mu$-almost everywhere, i.e., they give rise to the same element of $L^1(\mu)$. 
\end{proof}
Clearly, the map $f \mapsto \int f~d\mu$ is linear. Here are some more properties that we will need. We do not formulate them in the largest possible generality, but already in the setting that we need.
\begin{thm}\label{thm:Gelfand}
    Let $G$ be an lca group, $X = Y'$ a dual Banach space, $f \in L^1(G)$ and $h: G \to X$ bounded and weak$^\ast$ continuous. Then, $f \cdot h: G \to X$ is Gelfand integrable and satisfies:
    \begin{enumerate}
        \item $\| \int_G f(g) h(g)~dg\|_X \leq \| f\|_{L^1}\| h\|_\infty$.
        \item If $A \in \mathcal L(Y)$ (and hence $A' \in \mathcal L(X)$), then  the following holds true:
        \begin{align*}
        A' \int_G f(g) h(g)~dg = \int_G f(g) A' h(g)~dg.
        \end{align*}
    \end{enumerate}
\end{thm}
\begin{proof}
    Clearly, $f\cdot h$ is weak$^\ast$ measurable, as $\langle f \cdot h, y\rangle = f \cdot \langle h, y\rangle$ for $y \in Y$, which is the product of a continuous and a measurable function, hence measurable. Further,
    \begin{align*}
        \int_G |f(g) \langle h(g), y\rangle|~dg \leq \int_G |f(g)|\| h(g)\| \| y\|~dg \leq \| f\|_{L^1(G)} \| h\|_\infty \| y\| < \infty,
    \end{align*}
    such that $f \cdot g$ is weak$^\ast$ integrable. The above estimate also shows the first statement. For the second statement, note that for $y \in Y$:
    \begin{align*}
        \langle A' \int_G f(g) h(g)~dg, y\rangle &= \langle \int_G f(g) h(g)~dg, Ay\rangle\\
        &= \int_G f(g) \langle h(g), Ay\rangle~dg\\
        &= \int_G f(g) \langle A' h(g), y\rangle~dg\\
        &= \langle \int_G f(g) A' h(g)~dg, y\rangle.
    \end{align*}
    This finishes the proof.
\end{proof}
The previous theorem, besides being often useful during this paper, also shows the biggest drawback of the Gelfand integral: You cannot pull arbitrary linear functionals into the integral. This is only possible for functionals which are given as duals of a predual map: $\varphi: \mathbb C \to Y$. In particular, when forming integrals in the trace class (which is itself a dual Banach space), it is not possible to pull the trace functional into the integral, at least without further explanation (as the trace functional does not possess a predual). Therefore, we will in the following also make use of the Bochner integral. 

As already stated earlier, a function $\Omega \to X$ is $\mu$-measurable if and only if it can be written as the (a.e.-)limit of simple functions. A $\mu$-measurable function $f$ is called Bochner-integrable if $\| f\|_X$ is integrable. In that case, $f$ can be written as the limit of integrable simple functions (meaning for $\sum_{j=1}^n x_j \mathbf 1_{A_j}$ that $A_j$ has finite measure whenever $x_j \neq 0$), and $\int f~d\mu$ is defined as the limit of $\lim_{n \to \infty} \int f_n~d\mu$, where $(f_n)_{n \in \mathbb N}$ is any sequence of integrable simple functions converging to $f$. Some simple facts on Bochner integrability are listed below:
\begin{thm}
    Let $(\Omega, \mathcal A, \mu)$ be a measure space and $f, g: \Omega \to X$ be Bochner integrable. Then, the following facts hold true:
    \begin{enumerate}
        \item $\int \alpha f + g~d\mu = \alpha \int f~d\mu + \int g ~d\mu$ for $\alpha \in \mathbb C$;
        \item $\| \int f~d\mu\| \leq \int \| f\|~d\mu$;
        \item If $Z$ is another Banach space and $A\in \mathcal L(X, Z)$, then $Af$ is Bochner integrable with $A\int f~d\mu = \int Af~d\mu$.
        \item When $X$ is a dual Banach space, then $f$ is also Gelfand integrable and both the Gelfand and the Bochner integral of $f$ agree.
    \end{enumerate}
\end{thm}
Of course, the second-to-last property yields exactly what we want regarding linear functionals such as the trace. 

We will only consider the following special class of functions for which we will form Bochner integrals:
\begin{thm}\label{thm:Bochner}
    Let $G$ be an lca group, $X$ a Banach space, $f \in L^1(G)$ and $h: G \to X$ bounded and uniformly continuous. Then, $f \cdot h$ is Bochner integrable.
\end{thm}
Before we prove this, we will mention the following fact that we will make repeated use of throughout the paper.
\begin{lem}{{\cite[Corollary 1.3.6]{Deitmar_Echterhoff2014}}}\label{lemma:Deitmar}
   Let $G$ be a locally compact abelian group and let $f \in L^p(G)$, $1 \leq p < \infty$. Then, $f$ is supported in some $\sigma$-compact open subgroup of $G$.
\end{lem}
\begin{proof}[Proof of Theorem \ref{thm:Bochner}]
    By the previous lemma, we may assume that $f \cdot h$ is supported on a $\sigma$-finite open subgroup. Clearly, 
    \begin{align*} 
    \| f \cdot h\| \leq |f| \| h\|_\infty,
    \end{align*} 
    hence we only need to prove that $f \cdot h$ is $\mu$-measurable (with $\mu$ the Haar measure) to establish Bochner integrability. By measure-theoretic induction, it suffices to prove the statement for $f$ being a simple function: $f = \mathbf 1_A$ with $A$ in the Borel-$\sigma$-algebra of $G$. By $\sigma$-compactness of the subgroup in which $f$ is supported, it again suffices to prove the statement for $A \subset G$ being relatively compact. 

    Let $\varepsilon > 0$. Then, there exists an open neighborhood $O \subset G$ of $0 \in G$ such that, for all $y \in G$: $\| h(x+y) - h(x)\| \leq \varepsilon$, $x \in O$. Let $(O + x_k)_{k = 1, \dots, N}$ be an open cover of $\overline{A}$ and set $O_k = (O + x_k) \setminus (\cup_{j=1}^{k-1} O + x_j)$. Then, 
    \begin{align*}
        k(x) := \sum_{k=1}^N h(x_k) \cdot \mathbf 1_{O_k \cap A}
    \end{align*}
    is a simple function. We claim that $\| \mathbf 1_A (x) h(x) - k(x) \| \leq \varepsilon$ for each $x \in G$. Indeed, for $x \in A^c$ we trivially have that the difference is zero. For $x \in A$, we have $x \in O_k$ for exactly one $k = 1, \dots, N$. Then,
    \begin{align*}
        \| f(x) h(x) - k(x)\| &= \| h(x) - h(x_k)\| \leq \varepsilon,
    \end{align*}
    as $x \in O + x_k$ and using uniform continuity of $h$. Hence, $\mathbf 1_A \cdot h$ can be approximated (even uniformly) by a sequence of simple functions.
\end{proof}
We want to end this discussion on integration theory by mentioning the following facts, even though they are only about scalar-valued integration. It is written in the notations already introduced in Section \ref{sec:repr}, i.e., $(\Xi, m)$ is an lca group endowed with a 2-cocycle as in Assumptions \ref{assum:1} and \ref{assum:2} and $(U_x)_{x \in \Xi}$ the corresponding irreducible projective unitary representation on the Hilbert space $\mathcal H$. 
\begin{lem}\label{lemma1}
    Let $\varphi, \psi \in \mathcal H$. Then, there is a $\sigma$-compact subset $A \subseteq \Xi$ such that $\langle U_x \varphi, \psi\rangle = 0$ whenever $x \not \in A$.
\end{lem}
The lemma follows immediately from the fact that $x \mapsto \langle U_x \varphi, \psi\rangle \in L^2(\Xi)$ and Lemma \ref{lemma:Deitmar}.

Since the Haar measure restricted to $\sigma$-compact sets is $\sigma$-finite, Lemma \ref{lemma1} can be used to justify the application of the theorems of Fubini and Tonelli in certain situations. We will make repeated use of this throughout later sections.

Another issue that needs to be addressed when integrating over groups which are not second countable is the following: With the standard definition of $L^\infty(G)$, this space can no longer be identified with the dual of $L^1(G)$. This can be addressed by making the following modified definition of $L^\infty(G)$. A good reference for this is \cite[Section 2.3]{Folland2016}.

We say that a subset $A \subset G$ is \emph{locally Borel} if $A \cap F$ is a Borel set for every $\sigma$-finite Borel set $F \subset G$. $A$ is called \emph{locally null} if it is locally Borel and $A \cap F$ is of measure zero for every $\sigma$-finite Borel set $F \subset G$.

A function $f: G \to \mathbb R$ is called \emph{locally Borel-measurable} if $f^{-1}(B)$ is locally Borel for every Borel set $B \subset \mathbb R$. Such a function is in $\mathcal L^\infty(G)$ if there exists a constant $c > 0$ such that $\{ g \in G: ~|f(x)| > c\}$ is locally null. The infimum over all such constants is then the seminorm $\| f\|_{\mathcal L^\infty}$. $L^\infty(G)$ is now obtained by identifying functions which differ only on locally null sets. Of course, when $G$ is second-countable, this coincides with the standard definition of $L^\infty(G)$.

A more manageable description of local Borel sets is obtained when taking into account some structure theory: $G$ always has a subgroup $G_0$ which is open (hence automatically closed) and $\sigma$-compact, cf.\ \cite[Proposition 2.4]{Folland2016}. Now, a subset $A \subset G$ is locally Borel if and only if $A \cap yG_0$ is Borel for every $yG_0 \in G/G_0$. $A$ is locally null if and only if $A \cap yG_0$ is of measure zero for every $yG_0 \in G/G_0$. If $\{ yG_0: ~A \cap yG_0 \neq \emptyset\}$ is countable, then $A$ is a Borel set and satisfies:
\begin{align*}
    \lambda(A) = \sum_{yG_0 \in G/G_0} \lambda(A \cap yG_0),
\end{align*}
where $\lambda$ denotes the Haar measure. 

A rather explanatory example regarding the differences of both conventions of $L^\infty(G)$ is given immediately before and after Proposition \ref{prop:convprop} below.

\section{Convolutions in Quantum Harmonic Analysis}\label{sec:conv}
In this section, we will introduce the convolution operators which lie at the heart of quantum harmonic analysis. We want to note that several of our results can already be found in \cite{Halvdansson2022}, applied to the ordinary representations of the Heisenberg group affiliated to the phase space. Nevertheless, we decided to give a full presentation of the matter from the beginning: Besides the fact that the mathematics is a little more elegant in the abelian case, we will later of course make crucial use of the fact that the convolution of operators turns out to be commutative, if defined appropriately. For this, we will make use of the operator $R$ (as defined in Eq.~\eqref{eq:defR}). In the non-commutative (or even non-unimodular) setting, which is covered by the work of Halvdansson, the analogue of the operator $R$ (if it even exists) turns in general out to be an unbounded operator, causing obvious problems. The approach of Halvdansson is to simply ignore $R$ in the convolution of operators. Including $R$ is crucial in obtaining a convolution of functions and operators which is both associative and commutative, two properties that we want to have for our work. Hence, it is not feasible to follow the same road as Halvdansson in our setting. Besides this, it seems that \cite{Halvdansson2022} was written with second-countable groups in mind, without spelling this assumption out. As we will see, certain measure-theoretic issues arise when working on non-second countable groups (and, as a consequence, representations on non-separable Hilbert spaces), non of which has been addressed in \cite{Halvdansson2022}. Hence, we decided to give a full presentation of the theory without adhering to \cite{Halvdansson2022}.

We will frequently use the following definitions: For functions $f: \Xi \to \mathbb C$ we will write
\begin{align*}
    \alpha_{x}(f)(y) &= f(y-x), \quad x, y \in \Xi,\\
    \beta_{-}(f)(y) &= f(-y).
\end{align*}
For an operator $A \in \mathcal L(\mathcal H)$ we will denote
\begin{align*}
    \alpha_{x}(A) &= U_{x} A U_{x}^\ast, \quad x \in \Xi,\\
    \beta_{-}(A) &= RAR^\ast.
\end{align*}
$\alpha_{x}(f)$ and $\alpha_{x}(A)$ will be referred to as the \emph{shifts} of $f$ and $A$ respectively. Note that:
\begin{align*}
\alpha_{x}(\alpha_{y}(f)) &= \alpha_{x+y}(f), ~\alpha_{x}(\beta_{-}(f)) = \beta_{-}(\alpha_{-x}(f)),\\
\alpha_{x}(\alpha_{y}(A)) &= \alpha_{x+y}(A) = \alpha_y(\alpha_x(A)), ~\alpha_{x}(\beta_{-}(A)) = \beta_{-}(\alpha_{-x}(A)).
\end{align*}
Further, if $a: \Xi \to S^1$ is continuous with $a(x) = a(-x)$, then
\begin{align*}
    U_x^a A (U_x^a)^\ast = U_x A U_x^\ast,
\end{align*}
hence $\alpha_x$ does not depend on the choice of the factor $a$. If $R' = \frac{1}{c^{1/2}}R$ such that $R'$ is self-adjoint, then
\begin{align*}
    RAR^\ast = R' A (R')^\ast = R' A R',
\end{align*}
so whenever we consider $\beta_-(A)$, we may and will assume that $R$ is self-adjoint.
In the following lemma, $\mathcal T^1(\mathcal H)$ denotes the trace class and $\| A\|_{\mathcal T^1}$ the trace norm.
\begin{lem}\begin{enumerate}[(1)]
    \item Let $A \in \mathcal T^1(\mathcal H)$. Then, $\alpha_{x}(A) \in \mathcal T^1(\mathcal H)$, $\| \alpha_{x}(A)\|_{\mathcal T^1} = \| A\|_{\mathcal T^1}$ for every $x \in \Xi$ and $\Xi \ni x \mapsto \alpha_{x}(A)$ is continuous in trace norm.
    \item Let $A \in \mathcal L(\mathcal H)$. Then, $\| \alpha_{x}(A)\|_{op} = \| A\|_{op}$ for every $x \in \Xi$ and $\Xi \ni x \mapsto \alpha_{x}(A)$ is weak$^\ast$ continuous (with respect to the predual $\mathcal T^1(\mathcal H)$ of $\mathcal L(\mathcal H)$).
\end{enumerate}
\end{lem}

\begin{proof}
    The proof is analogous to that in \cite{Werner1984}: For rank one operators, use that $x \mapsto U_{x}$ is continuous in strong operator topology. Approximate a generic trace class operator by finite rank operators. Further, the norm is preserved since every $U_{x}$ is unitary. The statement of (2) follows now easily from (1) by duality.    
\end{proof}
We let $\d x$ be the Haar measure of $\Xi$, normalized such that $\lambda = 1$ in Theorem \ref{thm:moyal}. Given $f \in L^1(\Xi)$ and $A \in \mathcal L(\mathcal H)$, we can consider the function $h(x) = f(x) \alpha_{x}(A)$. Since $f$ is measurable and $x \mapsto \alpha_{x}(A)$ is weak$^\ast$ continuous, $h$ is weak$^\ast$ measurable. Thus, the integral
\begin{align*}
    f \ast A := A \ast f := \int_{\Xi} f(x) ~\alpha_{x}(A)~\d x
\end{align*}
exists as a Gelfand integral, cf.~Theorem \ref{thm:Gelfand}. $f \ast A$ defines again a bounded operator on $\mathcal H$ with $\| f \ast A\|_{op} \leq \| f\|_{L^1(\Xi)} \| A\|_{op}$. If $A \in \mathcal T^1(\mathcal H)$, then the continuity properties in trace class guarantee that $f \ast A \in \mathcal T^1(\mathcal H)$ (since the integral is defined as a Bochner integral in $\mathcal T^1(\mathcal H)$ by Theorem \ref{thm:Bochner}) with $\| f \ast A\|_{\mathcal T^1} \leq \| f\|_{L^1(\Xi)} \| A\|_{\mathcal T^1}$. Note that this ``convolution'' shares some similarities with the classical convolution of functions: If $f, g \in L^1(\Xi)$, then
\begin{align*}
    f \ast g = \int_{\Xi} f(x) g(\cdot - x)~\d x = \int_{\Xi} f(x) \alpha_{x}(g)~\d x,
\end{align*}
which can again be understood as a Bochner integral in $L^1(\Xi)$. Similarly to the well-known formula
\begin{align}
    \alpha_{x}(f \ast g) = (\alpha_{x}(f)) \ast g = f \ast (\alpha_{x}(g)),
\end{align}
it is also true that
\begin{align}
    \alpha_{x}(f \ast A) = f \ast (\alpha_{x}(A)).
\end{align}
Further, we have:
\begin{align}
    \alpha_x(f \ast A) &= \int_\Xi f(y) \alpha_{x+y}(A)~\d y\\
    &= \int_\Xi f(y-x) \alpha_y(A)~\d y = \alpha_x(f) \ast A.
\end{align}
These two convolutions are accompanied by a third convolution, which is taken between two operators. To motivate this, note that for $f, g \in L^1(\Xi)$:
\begin{align*}
    f \ast g(y) = \int_{\Xi} f(x) \alpha_{y}(\beta_{-}(g))(x)~\d x.
\end{align*}
Since the operator trace plays the analogous role for $\mathcal T^1(\mathcal H)$ as the integral plays for $L^1(\Xi)$, we define the following: Given $A \in \mathcal T^1(\mathcal H)$ and $B \in \mathcal L(\mathcal H)$, we set:
\begin{align*}
    A \ast B(x) := \tr(A \alpha_{x}(\beta_{-}(B))) = \tr(\alpha_{x}(\beta_{-}(A)) B).
\end{align*}
Using the continuity properties of $x \mapsto \alpha_{x}(A)$ in $\mathcal T^1(\mathcal H)$, it is not hard to see that $A \ast B$ is a bounded and uniformly continuous function on $\Xi$ satisfying $\| A \ast B\|_{\infty} \leq \| A\|_{\mathcal T^1} \| B\|_{op}$. Furthermore, we again have $\alpha_x(A\ast B) = \alpha_x(A)\ast B = A\ast\alpha_x(B)$. The key property to this third convolution is the following:
\begin{prop}\label{prop:thirdconv}
Let $A, B \in \mathcal T^1(\mathcal H)$. Then, it is $A \ast B \in L^1(\Xi)$ with $\| A \ast B\|_{L^1(\Xi)} \leq \| A\|_{\mathcal T^1} \| B\|_{\mathcal T^1}$ and $\int_{\Xi} A \ast B(x)~\d x = \tr(A) \tr(B)$.
\end{prop}
Before attempting the proof, let us note that we will write rank one operators on $\mathcal H$ in the tensor product notation $A = \varphi \otimes \psi$, which means $A(\phi) = (\varphi \otimes \psi)(\phi) = \langle \phi, \psi\rangle \varphi$. In particular, the tensor products are linear in the first and anti-linear in the second factor.
\begin{proof}[Proof of Proposition \ref{prop:thirdconv}]
First, we consider rank one operators $A = \varphi_1 \otimes \psi_1$ and $B = \varphi_2 \otimes \psi_2$. Then,
\begin{align}
    A \ast B(x) = \langle U_x R \varphi_2, \psi_1\rangle \overline{\langle U_x R \psi_2, \varphi_1\rangle}
\end{align}
such that Theorem \ref{thm:moyal} gives:
\begin{align*}
    \int_\Xi A \ast B(x) ~dx = \langle \varphi_1, \psi_1\rangle \langle \varphi_2, \psi_2\rangle = \tr(A)\tr(B).
\end{align*}
The Cauchy-Schwarz inequality yields
\begin{align*}
    \int_\Xi |A \ast B(x)|~dx &= \int_\Xi | \langle U_x R \varphi_2, \psi_1\rangle \langle U_x R \psi_2, \varphi_1\rangle|~dx\\
    &\leq \left( \int_\Xi |\langle U_x R \varphi_2, \psi_1\rangle|^2~dx \int_\Xi |\langle U_x R \psi_2, \varphi_1\rangle|^2~dx \right)^{1/2}\\
    &= \| \varphi_1\| \| \varphi_2\| \| \psi_1\| \| \psi_2\|\\
    &= \| A\|_{\mathcal T^1} \| B\|_{\mathcal T^1},
\end{align*}
where we again used Theorem \ref{thm:moyal}. Hence, the result follows for rank one operators. For arbitrary trace class operators, write them as a convergent series of rank one operators and deduce the result from the rank one case.
\end{proof}
\begin{cor}\label{lemma2}
For $A \in \mathcal T^1(\mathcal H)$ and $\varphi, \psi \in \mathcal H$, the map $x \mapsto \langle A U_{x}^\ast \varphi, U_{x}^\ast \psi\rangle$ is in $L^1(\Xi)$ with $\int_\Xi |\langle A U_x^\ast \varphi, U_x^\ast \psi\rangle|~\d x \leq \| A\|_{\mathcal T^1} \| \varphi\| \| \psi\|$.
\end{cor}
\begin{proof}
    One verifies easily that $A \ast (R\varphi \otimes R\psi)(x) = \langle AU_{x} \varphi, U_{x}\psi\rangle$, which is in $L^1(\Xi)$ by the previous result and satisfies
    \begin{align*}
        \int_\Xi |A \ast (R \varphi \otimes R \psi) (x)|~\d x \leq \| A\|_{\mathcal T^1} \| R \varphi \otimes R \psi\|_{\mathcal T^1} = \| A\|_{\mathcal T^1} \| \varphi\| \| \psi\|.
    \end{align*}
    Now, using that the Haar measure is unimodular, we see that
    \begin{align*}
        \int_\Xi |\langle A U_x^\ast \varphi, U_x^\ast \psi\rangle| ~\d x &= \int_\Xi |\langle A U_{-x} \varphi, U_{-x}\psi\rangle|~\d x\\
        &= \int_\Xi |\langle AU_x \varphi, U_x \psi\rangle|~\d x,
    \end{align*}
    which finishes the proof.
\end{proof}
Having defined all the necessary convolutions between elements of $L^1(\Xi)$ and $\mathcal T^1(\mathcal H)$, we can now extend it to a bilinear operation on $L^1(\Xi) \oplus \mathcal T^1(\mathcal H)$, denoted by the same symbol:
\begin{align}
    (f, A) \ast (g, B) := (f \ast g + A \ast B, f \ast B + g \ast A).
\end{align}
\begin{prop}
 The operation ``$\ast$'' on $L^1(\Xi) \oplus \mathcal T^1(\mathcal H)$ is commutative, associative and distributive. 
\end{prop}
\begin{proof}
    The proofs are very similar to that of the case with $\Xi = \mathbb R^{2n}$, using Fubini/Tonelli at the same places as justified by Lemma \ref{lemma1}. We will not repeat the details here and refer instead to \cite{Werner1984, fulsche2020correspondence, luef_skrettingland2018}. Nevertheless, we want to emphasize that commutativity of the operator convolution, $A \ast B = B \ast A$, is satisfied due to the identities
    \begin{align}
        U_{-x}^\ast A U_{-x} = \overline{m(-x,x)}U_x A U_{-x} = U_x A U_x^\ast, \quad x \in \Xi,
    \end{align}
    and the possibility of choosing $R$ self-adjoint.
\end{proof}
\begin{rem}
    We recall that, in the following, we will consider the space $L^\infty(\Xi)$ as it was defined at the end of Section \ref{sec:Gelfand_Pettis}, which keeps the duality $L^1(\Xi)' \cong L^\infty(\Xi)$ intact even if $\Xi$ is not second countable.
\end{rem}
Clearly, $f \ast g$ is also well-defined if $f \in L^1(\Xi)$ and $g \in L^\infty(\Xi)$. Similarly, $A \ast B$ makes sense by the same formula for $A \in \mathcal T^1(\mathcal H)$ and $B \in \mathcal L(\mathcal H)$. We have also seen that $f \ast A$ is well-defined for $f \in L^1(\Xi)$ and $A \in \mathcal L(\mathcal H)$. There is one gap to fill in this scheme: $f \ast A$ also makes sense when $f \in L^\infty(\Xi)$ and $A \in \mathcal T^1(\mathcal H)$. Then, the integral is defined in the weak$^\ast$ sense: For $N \in \mathcal T^1(\mathcal H)$, it is
    \begin{align*}
        \langle \int_{\Xi} f(x) \alpha_{x}(A) ~\d x , N\rangle_{tr} &:= \int_{\Xi} \langle f(x) \alpha_x(A), N\rangle_{tr} ~\d x \\
        &= \int_{\Xi} f(x) N \ast \beta_-(A)(x)~\d x.
    \end{align*}
    Here, $\langle \cdot, \cdot\rangle_{tr}$ denoted the bilinear trace duality pairing. By Proposition \ref{prop:thirdconv}, this expression is well-defined with
    \begin{align*}
        \left | \langle \int_\Xi f(x) \alpha_x(A)~dx, N\rangle_{tr} \right | \leq \| f\|_{L^\infty} \| A\|_{\mathcal T^1} \| N\|_{\mathcal T^1}.
    \end{align*}    
    Hence, by Theorem \ref{thm:Gelfand}, the integral $f \ast A := \int_\Xi f(x) \alpha_x(A)$ exists as a Gelfand integral in $\mathcal L(\mathcal H)$ with $\| f \ast A\| \leq \| f\|_\infty \| A\|_{\mathcal T^1}$. More or less by definition, this convolution is now also compatible with the duality pairing between $\mathcal T^1(\mathcal H)$ and $\mathcal L(\mathcal H)$, e.g., $\langle f \ast A, N\rangle_{tr} = \langle f, \beta_-(A) \ast N\rangle_{tr}$, cf.~\cite{Werner1984,fulsche2020correspondence,luef_skrettingland2018}. 
\begin{ex}
    Let $\Xi = G \times \widehat{G}$ with $G$ discrete and $U_{(x, \xi)}$ as in Example \ref{main_example}, acting on $L^2(G) = \ell^2(G)$. Then, $\delta_0 \in \ell^2(G)$. Let $f \in \ell^\infty(G)$ and extend $f$ to a function on $\Xi$ by $f_0(x, \xi) = f(x)$. Then, since $\widehat{G}$ is compact, we have for $\varphi \in \ell^2(G)$ and $y \in G$:
    \begin{align*}
        f_0 \ast (\delta_0 \otimes \delta_0)(\varphi)(y) &= \sum_{x \in G} \int_{\widehat{G}} f_0(x,\xi) [U_{(x,\xi)}\delta_0](y) \langle \varphi, U_{(x,\xi)}\delta_0\rangle~d\xi\\
        &= \sum_{x \in G} f(x)\varphi(x) \delta_0(y-x) \int_{\widehat{G}} \xi(y) \overline{\xi(x)}~d\xi\\
        &= f(y)\varphi(y).
    \end{align*}
    Hence, $f_0 \ast (\delta_0 \otimes \delta_0) = M_f$, the operator of multiplication by $f$.
\end{ex}
    Here is one special instance of this convolution:
\begin{lem}\label{lemma:average}
    Let $A \in \mathcal T^1(\mathcal H)$. Then, it is:
    \begin{align*}
        1 \ast A = \int_{\Xi} \alpha_{x}(A)~\d x = \tr(A) I.
    \end{align*}
\end{lem}
\begin{proof}
    Observe that
    \begin{align*}
        \alpha_{y}(1 \ast A) = \int_{\Xi} \alpha_{x+y}(A)~\d x = \int_{\Xi} \alpha_{x}(A)~\d x  = 1 \ast A.
    \end{align*}
    By irreducibility of the representation, Schur's Lemma implies $1 \ast A = c I$ for some $c\in \mathbb C$. Now, for any $B \in \mathcal T^1(\mathcal H)$ it is:
    \begin{align*}
        c \tr(B) = \tr((1 \ast A)B) = \int_{\Xi} 1 \cdot (\beta_-(A) \ast B)(x) ~\d x = \tr(A) \tr(B),
    \end{align*}
    showing that $c =\tr(A)$.
\end{proof}

We want to end this discussion by mentioning that, using the mapping properties we have obtained so far for convolutions together with the complex interpolation method \cite{Bergh_Lofstrom1976}, one obtains the following version of Young's convolution inequality:
\begin{thm}
    Let $r, p, q \in [1, \infty]$ such that $1 + \frac{1}{r} = \frac{1}{p} + \frac{1}{q}$. Further, let $f, g: \Xi \to \mathbb C$ be measurable functions and $A, B \in \mathcal L(\mathcal H)$. Then, the following estimates hold true:
    \begin{align}
        \| f \ast g\|_{L^r} &\leq \| f\|_{L^p} \| g\|_{L^q},\\
        \| f \ast B\|_{\mathcal T^r} &\leq \| f\|_{L^p} \| A\|_{\mathcal T^q},\\
        \| A \ast g\|_{\mathcal T^r} &\leq \| A\|_{\mathcal T^p} \| g\|_{L^q},\\
        \| A \ast B\|_{L^r} &\leq \| A\|_{T^p} \| B\|_{T^q}.
    \end{align}
\end{thm}

\section{Convolutions and positivity}\label{sec:positivity}

An operator $A \in \mathcal L(\mathcal H)$ is said to be positive, $A \geq 0$, if it is positive in the sense of quadratic forms: $\langle A\varphi, \varphi\rangle \geq 0$ for all $\varphi \in \mathcal H$. Functions $f \in L^\infty(\Xi)$ are called positive when they satisfy $f \geq 0$ almost everywhere. An element $(f, A)$ of $L^\infty(\Xi) \oplus \mathcal L(\mathcal H)$ is positive, $(f, A) \geq 0$, if both entries are positive.
\begin{lem}
    Let $A \in \mathcal T^1(\mathcal H), B \in \mathcal L(\mathcal H)$, $A, B \geq 0$ and $f \in L^\infty(\Xi)$, $g \in L^1(\Xi)$, $f, g \geq 0$. Then, it is $f \ast A \geq 0$, $A \ast B \geq 0$ and $g \ast B \geq 0$.
\end{lem}
\begin{proof}
    Let $\varphi \in \mathcal H$. Then, we have
    \begin{align*}
        \langle f \ast A \varphi, \varphi\rangle &= \int_\Xi f(x) \langle AU_{x}^\ast \varphi, U_{x}^\ast \varphi\rangle~\d x \geq 0.
    \end{align*}
    Similarly, one obtains that $g \ast B \geq 0$. 
    For the convolution of $A$ and $B$, it suffices to verify the claim for $A$ being a positive rank one operator, say $A = \varphi\otimes \varphi$. Then, one readily verifies that 
    \begin{align*}
        A \ast B(x) = \langle BU_{x}^\ast \varphi, U_{x}^\ast \varphi\rangle \geq 0.
    \end{align*}
    This finishes the proof.
\end{proof}
Let us briefly discuss a fact from classical harmonic analysis. One might be tempted to state the following as a true statement, say for $G$ a locally compact abelian group:
\begin{align}\label{eq:convprop}
\begin{split}
    0 \leq f \in L^\infty(G) \text{ s.th. } & f \ast g \in L^1(G) \text{ for every } 0 \leq g \in L^1(G) \\
    & \Rightarrow f \in L^1(G).
\end{split}
\end{align}
When the Haar measure on $G$ is $\sigma$-finite, the above statement is easily seen to be true by Tonelli's theorem:
\begin{align*}
    \| f\|_{L^1} \| g\|_{L^1} = \| f\ast g\|_{L^1} < \infty,
\end{align*}
hence $f \in L^1(G)$. Nevertheless, for general locally compact abelian groups the statement needs some caution.

To give an example, we consider the group $G = \mathbb R \times \mathbb R_d$, where $\mathbb R_d$ denotes the additive group of real numbers, endowed with the discrete topology. Then, if we denote by $D = \{ (x, x): ~x \in \mathbb R\}$ the diagonal in $\mathbb R^2$, it is not hard to see that $D$ is a member of the Borel-$\sigma$-algebra of $G$, hence $0 \leq \mathbf 1_D \in L^\infty(G)$. Further, using the description of the Haar measure of $G$ given in \cite[Section 2.3]{Folland2016}, it is not hard to see that the Haar measure of $D$ equals $\infty$, i.e.\ $\mathbf 1_D \not \in L^1(G)$. Further, one can prove that $\mathbf 1_D \ast g = 0$ for every $g \in L^1(G)$. Hence, this might look as if one has obtained a counterexample to \eqref{eq:convprop}. Nevertheless, when dealing with the right formulation of the result, which is the following, it still holds true:
\begin{prop}\label{prop:convprop}
    Let $0 \leq f \in L^\infty(G)$ such that $f \ast g \in L^1(G)$ for every $0 \leq g \in L^1(G)$. Then, there exists a representative $f'$ of $f \in L^\infty(G)$ such that $f' \in L^1(G)$.

    If further $f \in C_b(G)$, i.e.\ $f$ is continuous, then $f \in L^1(G)$.
\end{prop}
In the above example, the representative would simply be $f' = 0$, as the diagonal $D$ is a locally null set. At this stage, we want to refer the reader back to the discussion regarding locally null sets at the end of Section \ref{sec:Gelfand_Pettis}. Indeed, using the description of the measure of locally Borel sets, it is not hard to verify that in the example $G = \mathbb R \times \mathbb R_d$, where $G_0 = \mathbb R \times \{ 0\}$, the diagonal $D$ is locally zero such that $\mathbf 1_D = 0$ in $L^\infty(G)$. 

We now come back to Proposition \ref{prop:convprop}. While the statement is certainly well-known, we could not locate a decent reference in the literature. Therefore, we provide a proof for completeness.
\begin{proof}[Proof of Proposition \ref{prop:convprop}]
    For convenience, we choose to fix a particular representative in $\mathcal L^\infty(G)$ of $f \in L^\infty(G)$, which we will nevertheless still denote by $f$. Of course, we can choose this representative such that it is non-negative everywhere. 
    
    Let $A \subset G$ be $\sigma$-finite. Then, we have for every $0 \leq g \in L^1(G)$:
    \begin{align*}
        (\mathbf 1_A f) \ast g(x) = \int_A f(y) g(x-y)~dy \leq f \ast g(x).
    \end{align*}
    Hence, the assumption shows that $(\mathbf 1_A f) \ast g \in L^1(G)$ for every $0 \leq g \in L^1(G)$. Fubini's theorem yields that $\mathbf 1_A f \in L^1(G)$ for every $\sigma$-finite subset $A \subset G$.

    We now set $B = \{ x \in G: ~f(x) > 0\}$. Since $f$ is locally Borel measurable, this set is locally Borel. Further, fix $G_0$ an open and $\sigma$-compact subgroup of $G$. We distinguish several cases:
    \begin{enumerate}[(i)]
        \item If $B$ is locally null, then $0 = f \in L^\infty(G)$ such that the statement holds true with $f' = 0$. If $f$ is continuous, then $f = 0$, because continuous functions which are locally null have to be the zero function.
        \item Assume that 
        \begin{align*}
           C := \{ y G_0 \in G/G_0: ~\lambda(B \cap yG_0) \neq 0\}
        \end{align*}
        is countable. Since $G_0$ is $\sigma$-compact, it is $\sigma$-finite. Hence $B \cap yG_0$ is $\sigma$-finite for every $yG_0 \in C$. Consequently, 
        \begin{align*}
            f' = f \mathbf 1_{\cup(B \cap yG_0)},
        \end{align*}
        where the union is taken over all $yG_0 \in C$, is supported on a $\sigma$-finite subset of $G$ and $f - f'$ is locally zero. Hence, Fubini's theorem can now be employed to prove that $f' \in L^1(G)$. 

        When $f$ is continuous and $f(x) \neq 0$ for some $x \in G$, then $f(x) \neq 0$ in a neighborhood of $x$. Hence, $B$ is open. Thus, $B \cap yG_0 \neq \emptyset$ implies $\lambda(B \cap yG_0) > 0$. In particular, $f = f'$, i.e.\ $f \in L^1(G)$.
        \item Assume that
        \begin{align*}
            \{ y G_0 \in G/G_0: ~\lambda(B \cap yG_0) \neq 0\}
        \end{align*}
        is uncountable. In this case, it is not hard to see that there is $\varepsilon > 0$ such that $\{ x \in G: ~f(x)> \varepsilon\}$ has a Borel subset $A$ which is $\sigma$-finite and of infinite Haar measure. Then, $\mathbf 1_A f \not \in L^1(G)$, a contradiction. Therefore, this case cannot occur.\qedhere
    \end{enumerate}
\end{proof}
From the perspective of quantum harmonic analysis, it is now natural to ask when implications of the following kind hold true:
\begin{align*}
    0 \leq A \in \mathcal L(\mathcal H) \text{ s.th. } g \ast A \in \mathcal T^1(\mathcal H) \text{ for all } 0 \leq g \in L^1(\Xi) \Rightarrow A \in \mathcal T^1(\mathcal H)\\
    0 \leq A \in \mathcal L(\mathcal H) \text{ s.th. } B \ast A \in L^1(\Xi) \text{ for all } 0 \leq B \in \mathcal T^1(\mathcal H) \Rightarrow A \in \mathcal T^1(\mathcal H)\\
    0 \leq f \in L^\infty(\Xi) \text{ s.th. } B \ast f \in \mathcal T^1(\mathcal H) \text{ for all } 0 \leq B \in \mathcal T^1(H) \Rightarrow f \in L^1(\Xi).
\end{align*}
We now discuss versions of these implications.
\begin{lem}
    Let $0 \leq A \in \mathcal L(\mathcal H)$ such that $g \ast A \in \mathcal T^1(\mathcal H)$ for every $0 \leq g \in L^1(\Xi)$. Then, it is $A \in \mathcal T^1(\mathcal H)$.
\end{lem}
\begin{proof}
      Assume $0 \leq g \in C_c(\Xi)$, $g \neq 0$. Let $\mathcal F$ be any orthonormal basis of $\mathcal H$ and note that $\{ U_x^\ast \phi: ~\phi \in \mathcal F\}$ is, for every $x \in \Xi$, again an orthonormal basis of $\mathcal H$. We see that
        \begin{align*}
            \infty > \sum_{\phi \in \mathcal F} \langle (g \ast A) \phi, \phi\rangle &= \sum_{\phi \in \mathcal F} \int_\Xi g(x) \langle U_x A U_x^\ast \phi, \phi\rangle~dx\\
            &= \sum_{\phi \in \mathcal F} \int_\Xi g(x) \langle AU_x^\ast \phi, U_x^\ast \phi\rangle~dx.
        \end{align*}
        Here, the sum over the possibly uncountable set $\mathcal F$ is defined as
        \begin{align*}
            \sum_{\phi \in \mathcal F} \langle (g \ast A)\phi, \phi\rangle := \sup_{\mathcal F' \subset \mathcal F: \mathcal F' \text{finite}} \sum_{\phi \in \mathcal F'} \langle (g \ast A)\phi, \phi\rangle,
        \end{align*}
        and similarly for the other sums. An appropriate version of the monotone convergence theorem for nets, see for example \cite[Theorem IV.15]{Reed_Simon1}\footnote{Note that Reed and Simon didn't spell out the assumption that the net consists of non-negative functions, which is necessary for the theorem to hold.}, gives:
        \begin{align*}
            \sum_{\phi \in \mathcal F} \int_\Xi g(x) \langle AU_x^\ast \phi, U_x^\ast \phi\rangle~dx &= \int_\Xi g(x) \sum_{\phi \in \mathcal F} \langle AU_x^\ast \phi, U_x^\ast \phi\rangle ~dx\\
            &= \tr(A) \| g\|_{L^1}.
        \end{align*}
        This proves that $A$ is trace class.
\end{proof}

\begin{lem}
    Let $0 \leq A \in \mathcal L(\mathcal H)$ such that $A \ast B \in L^1(\Xi)$ for every $0 \leq B \in \mathcal T^1(\mathcal H)$. Then, $A \in \mathcal T^1(\mathcal H)$.
\end{lem}
\begin{proof}
    We distinguish two cases.
    \begin{enumerate}[(i)]
        \item $A \in \mathcal K(\mathcal H)$. In this case, we write $A = \sum_{j=1}^\infty \lambda_j \varphi_j \otimes \varphi_j$, where $\lambda_j$ is a non-negative and decreasing sequence converging to $0$ and $(\varphi_j)_{j \in \mathbb N}$ is an orthonormal system. Let $0 \leq B = \sum_{k=1}^\infty \mu_k \psi_k \otimes \psi_k \in \mathcal T^1(\mathcal H)$ such that $\mu_k$ is a non-negative decreasing sequence in $\ell^1(\mathbb N)$ and $(\psi_k)_k$ is an orthonormal system in $\mathcal H$. Then,
        \begin{align*}
            A \ast B(x) &= \sum_{j, k=1}^\infty \lambda_j \mu_k (\varphi_j \otimes \varphi_j) \ast (\psi_k \otimes \psi_k)(x)\\
            &=  \sum_{j, k=1}^\infty \lambda_j \mu_k  \tr((\varphi_j \otimes \varphi_j) U_x R (\psi_k \otimes \psi_k) R U_x^\ast)\\
            &=  \sum_{j, k=1}^\infty \lambda_j \mu_k |\langle U_x R \psi_k, \varphi_j\rangle|^2.
        \end{align*}
        By the monotone convergence theorem and Theorem \ref{thm:moyal}, we obtain:
        \begin{align*}
            \infty &> \int_\Xi A \ast B(x)~dx = \sum_{j, k=1}^\infty \lambda_j \mu_k \int_\Xi \int_\Xi  |\langle U_x R \psi_k, \varphi_j\rangle|^2 ~dx\\
            &=  \sum_{j, k=1}^\infty \lambda_j \mu_k \| \varphi_j\|^2 \| \psi_k\|^2 = \tr(A) \tr(B).
        \end{align*}
        \item Assume that $A$ is not compact. We distinguish two cases:
        \begin{enumerate}[(I)]
            \item $0$ is the only accumulation point of $\sigma(A)$. Since $A$ is not compact, we know that there has to be an eigenvalue $\lambda \neq 0$ of infinite multiplicity. Let $(\psi_n)_{n \in \mathbb N}$ be an orthonormal system within the eigenspace of $\lambda$. Then, consider the projection $T = \sum_{n =1}^\infty \psi_n \otimes \psi_n$. Given any rank one operator $B = \varphi \otimes \varphi$, we have
            \begin{align*}
                A \ast B(x) = \langle A U_x R \varphi, U_x R \varphi\rangle \geq \langle TAU_x R \varphi, U_x R \varphi\rangle
            \end{align*}
            by the spectral theorem, as $\mathbf 1_{\lambda}(t) \cdot t \leq t$ for $t \geq 0$, and because $T$ commutes with $A$. By the monotone convergence theorem and Theorem \ref{thm:moyal}:
            \begin{align*}
                \int_\Xi \langle TAU_x R \varphi, U_x R \varphi\rangle~dx &= \sum_{n=1}^\infty \int_\Xi \langle (\psi_n \otimes \psi_n)A U_x R \varphi, U_x R \varphi\rangle~dx\\
                &= \sum_{n=1}^\infty \int_\Xi \langle U_x R \varphi, A \psi_n\rangle \langle \psi_n, U_x R \varphi\rangle~dx\\
                &= \sum_{n=1}^\infty \| \varphi\|^2 \langle \psi_n, A\psi_n\rangle\\
                &= \sum_{n=1}^\infty \| \varphi\|^2 \lambda = \infty.
            \end{align*}
            Hence, we also have $\int_\Xi A \ast B(x)~dx = \infty$, which contradicts the assumption.
            \item There exists an accumulation point $\lambda_0 > 0$ of $\sigma(A)$. In this case, let $0 < \varepsilon < \lambda_0$. Consider, in the sense of the spectral theorem, $\mathbf 1_{B(\lambda_0, \varepsilon)}(A)$. By the spectral mapping theorem, this operator has infinite-dimensional range. Let $(\psi_n)_{n \in \mathbb N}$ be an orthonormal system within the range. Similarly to the case (I), one now has (with $B = \varphi \otimes \varphi$) that $A \ast B(x) \geq (\mathbf 1_{B(\lambda_0, \varepsilon)}(A)A) \ast B$ and:
            \begin{align*}
                \int_\Xi (\mathbf 1_{B(\lambda_0, \varepsilon)}(A)A) \ast B(x)~dx &= \sum_{n=1}^\infty \int \langle U_x R \varphi, A\psi_n\rangle \langle \psi_n, U_x R \varphi\rangle~dx\\
                &= \sum_{n=1}^\infty \| \varphi\|^2 \langle \psi_n, A\psi_n\rangle\\
                &\geq \sum_{n=1}^\infty \| \varphi\|^2 (\lambda_0 - \varepsilon)\\
                &= \infty.
            \end{align*}
            Hence, we also have $\int_\Xi A \ast B(x)~dx = \infty$.\qedhere
        \end{enumerate}
    \end{enumerate}
\end{proof}
We now come to the last of the four implications. While we think the statement should be true (at least in some form) for general $0 \leq f \in L^\infty(\Xi)$, some measure-theoretic issues with the monotone convergence theorem for nets force us to formulate and prove the statement only for bounded continuous functions.
\begin{lem}\label{lem:charl1}
    Let $0 \leq f \in C_b(\Xi)$ such that $B \ast f \in \mathcal T^1(\mathcal H)$ for every $0 \leq B \in \mathcal T^1(\mathcal H)$. Then, $f \in L^1(\Xi)$. 
\end{lem}
\begin{proof}
    We fix an orthonormal basis $(\varphi_{\gamma})_{\gamma \in \Gamma}$ of $\mathcal H$. For any compact set $K \subset \Xi$ we have:
    \begin{align*}
        B \ast (\mathbf 1_K f) \leq B \ast f.
    \end{align*}
    Hence, $B \ast (\mathbf 1_K f) \in \mathcal T^1(\mathcal H)$. Thus,
    \begin{align*}
        \tr(B \ast (\mathbf 1_K f)) &= \sum_{\gamma \in \Gamma} \langle (B \ast (\mathbf 1_K f) \varphi_\gamma, \varphi_\gamma\rangle\\
        &= \sum_{\gamma \in \Gamma} \int_K f(x) \langle B U_{-x}\varphi_\gamma, U_{-x}\varphi_\gamma\rangle~dx
    \end{align*}
    Using again the version of the monotone convergence theorem for nets from \cite[Theorem IV.15]{Reed_Simon1} yields:
    \begin{align*}
        \tr(B \ast (\mathbf 1_K f)) = \int_K f(x) \sum_{\gamma \in \Gamma} \langle B U_{-x}\varphi_\gamma, U_{-x}\varphi_\gamma\rangle~dx &= \tr(B) \int_K f(x)~dx.
    \end{align*}
    Hence, we have seen that for every $K \subset \Xi$ compact it is:
    \begin{align*}
        \tr(B \ast f) \geq \tr(B) \int_K f(x)~dx.
    \end{align*}
    If $f$ is not supported on a $\sigma$-finite subset, it is an easy exercise in measure theory to prove that $\int_\Xi f(x)~dx = \infty$, contradicting $\tr(B \ast f) < \infty$. If the support of $f$ is $\sigma$-finite, then we conclude that $\int_\Xi f(x)~dx \leq \frac{\tr(B \ast f)}{\tr(B)} < \infty$, which finishes the proof.
\end{proof}

In the following, we will occasionally identify $L^\infty(\Xi)$ and $\mathcal L(\mathcal H)$ with the subspaces $L^\infty(\Xi) \oplus \{ 0\}$ and $\{ 0 \} \oplus \mathcal L(\mathcal H)$ of $L^\infty(\Xi) \oplus \mathcal L(\mathcal H)$ without notationally distinguishing between them. Similarly, we will sometimes write ``$f$'' instead of ``$(f, 0)$'' and ``$A$'' instead of ``$(0, A)$''. 

We recall that $L^\infty(\Xi) \oplus \mathcal L(\mathcal H)$ is the dual space of $L^1(\Xi) \oplus \mathcal T^1(\mathcal H)$ under the isometric identification $\varphi_{(g, B)}(f, A) = \int f(x) g(x)~\d x + \tr(AB)$, $(g, B) \in L^\infty(\Xi) \oplus \mathcal L(\mathcal H)$. With respect to this predual, we can consider the weak$^\ast$ topology on $L^\infty(\Xi) \oplus \mathcal L(\mathcal H)$. Further, $L^\infty(\Xi) \oplus \mathcal L(\mathcal H)$ can be considered as a unital $C^\ast$ algebra with product $(f, A) (g, B) = (fg, AB)$, unit $(1, I)$ and norm $\| (f, A)\| = \max \{ \| f\|_\infty, \| A\|_{op}\}$.

One of the main reasons for Werner to investigate operator convolutions in \cite{Werner1984} was their connection to positive correspondence rules:
\begin{defin}
    A \emph{positive correspondence rule} is a map $\Gamma: L^\infty(\Xi) \oplus \mathcal L(\mathcal H) \to L^\infty(\Xi) \oplus \mathcal L(\mathcal H)$ with $\Gamma(L^\infty(\Xi)) \subset \mathcal L(\mathcal H)$ and $\Gamma(\mathcal L(\mathcal H)) \subset L^\infty(\Xi)$ satisfying the following properties:
    \begin{enumerate}
        \item $\Gamma$ is weak$^\ast$ continuous;
        \item $\Gamma(f, A) \geq 0$ whenever $(f, A) \geq 0$;
        \item $\Gamma(1,0) = (0, I)$ and $\Gamma(0, I) = (1, 0)$;
        \item $\alpha_{x}(\Gamma(f,A)) = \Gamma(\alpha_{x}(f,A))$.
    \end{enumerate}
\end{defin}
\begin{thm}\label{thm:corrrule}
    \begin{enumerate}[(1)]
    \item Let $\Gamma: L^\infty(\Xi) \oplus \mathcal L(\mathcal H) \to L^\infty(\Xi) \oplus \mathcal L(\mathcal H)$. Then, $\Gamma$ is a positive correspondence rule if and only if there are $B_1, B_2 \in \mathcal T^1(\mathcal H)$ with $B_j \geq 0$, $\tr(B_j) = 1$, such that $\Gamma(f, A) = (A \ast B_1, f \ast B_2)$. In this case, the $B_j$ are uniquely determined.
    \item If $\Gamma$ is a positive correspondence rule, then it is completely positive. In particular, $\Gamma((f, A)^\ast)  \Gamma ((f, A)) \leq \Gamma( (f, A)^\ast (f, A))$.
    \item $\Gamma$ satisfies the Berezin-Lieb inequalities: Assume $\Gamma$ is a positive correspondence rule. If $A = A^\ast \in \mathcal L(\mathcal H)$ and $\varphi: \sigma(A) \to [0, \infty)$ is convex and continuous ($\sigma(A)$ denoting the spectrum), then:
    \begin{align*}
        \int_\Xi \varphi(\Gamma(A)) \leq \tr(\varphi(A)).
    \end{align*}
    Similarly, if $f = \overline{f} \in L^\infty(\Xi)$ and $\varphi: \operatorname{essran}(f) \to [0, \infty)$ is continuous and convex ($\operatorname{essran}(f)$ denoting the essential range), then:
    \begin{align*}
        \varphi(\tr(\Gamma(f))) \leq \int_\Xi \varphi(f).
    \end{align*}
    \end{enumerate}
\end{thm}
We first want to emphasize that the above result, initially formulated by Werner for the case $\Xi = \mathbb R^{2n}$, is of independent interest, as it characterizes candidates for positivity preserving quantization maps. Parts of the result have been generalized to a significant class of locally compact groups, see e.g.\ \cite{Cassinelli_DeVito_Toigo2003, Kiukas_Lahti_Ylinen2006, Kiukas2006}. While those results are usually formulated for second countable groups, this restriction is not a serious one. The proofs of these generalizations boil down either to an application of Mackey's imprimitivity theorem or the Radon-Nikod\'{y}m property of the trace class, the latter having also being used in Werner's initial proof. Both approaches can be used in the non-second countable case. We will follow the proof close to Werner's original ideas. Further, Dammeier and Werner \cite{Dammeier_Werner2023} recently described how this theorem can be more easily be proven using the quantum Bochner theorem, Theorem \ref{thm:Bochner} below. Since our version of the quantum Bochner theorem needs a little extra structure (we need to assume that $m$ is similar to a bicharacter), which is not necessary for the present theorem, we chose to give the longer proof, which is independent of the quantum Bochner theorem. 
\begin{proof}[Proof of Theorem \ref{thm:corrrule}]
    \begin{enumerate}[(1)]
        \item Given such $B$, it is not hard to verify that $\Gamma$ is indeed a positive correspondence rule. The main part consists of proving the converse direction. Before attempting the proof, let us start by recalling some properties about the Radon-Nikod\'{y}m property. We refer to \cite{Diestel_Uhl1977, Ryan2002} for details. Since the Hilbert space $\mathcal H$ satisfies the approximation property, we can naturally identify the trace class $\mathcal T^1(\mathcal H)$ with the projective tensor product $\mathcal T^1(\mathcal H)  \cong \mathcal H \widehat{\otimes}_\pi \mathcal H^\ast$. Since Hilbert spaces also have the Radon-Nikod\'{y}m property, this implies that $\mathcal T^1(\mathcal H)$ also has the Radon-Nikod\'{y}m property. We recall that having the Radon-Nikod\'{y}m property means that whenever $(\Omega, \mathcal A, \lambda)$ is a finite measure space and $X: \Omega \to \mathcal T^1(\mathcal H)$ is a vector measure of bounded variation which is absolutely continuous with respect to $\lambda$ (i.e.\ $\lambda(S) = 0$ implies $X(S) = 0$ for $S \in \mathcal A$), then there exists some $f: \Omega \to \mathcal T^1(\mathcal H)$ which is Bochner integrable with respect to $\lambda$ such that $X(S) = \int_S f~d\lambda$. Further, $f$ is unique up to a.e.-equality and also essentially bounded. Clearly, this property then extends to $\sigma$-finite measure spaces $(\Omega, \mathcal A, \lambda)$ and vector measures which are of bounded variation on sets $S \in \mathcal A$ with $\lambda(S) < \infty$ as well (where, now, $f$ only needs to be bounded on sets with finite measure).

        We will split the proof into two parts: As the first step, part (i) below, we show that a map $\Gamma: L^\infty(\Xi) \to \mathcal L(\mathcal H)$, which is positivity preserving, $\alpha$-covariant, weak$^\ast$ continuous and satisfies $\Gamma(1) = \Id$, is of the form $\Gamma(f) = B \ast f$ for some $0 \leq B \in \mathcal T^1(\mathcal H)$ with $\tr(B) = 1$. As a second step, part (ii), we show that every map $\Gamma: \mathcal L(\mathcal H) \to L^\infty(\Xi)$ which is positivity preserving, $\alpha$-covariant, weak$^\ast$ continuous and satisfies $\Gamma(\Id) = 1$ is of the form $\Gamma(A) = B \ast A$, where again $0 \leq B \in \mathcal T^1(\mathcal H)$ with $\tr(B) = 1$. Putting both pieces together, we obtain the general statement of the theorem, where the two operators $B_1, B_2$ may be different.
        
        (i) We will show that there is an operator $B \in \mathcal T^1(\mathcal H)$, $B \geq 0$, such that $\Gamma(f)= B \ast f$. In doing so, we closely follow an argument taken from \cite{Kiukas2006}. We give Kiukas' argument again to both verify that second countability of the group (which Kiukas assumed) plays no role, and also to give the interested reader the proof right away. For proving that $\Gamma(f) = B \ast f$, we actually only work with the restriction of $\Gamma$ to $L^\infty(\Xi)$, a fact that we will use later again.
        
        Since $\Gamma$ is weak$^\ast$ continuous, there exists a continuous pre-adjoint map $\Gamma_\ast: L^1(\Xi) \oplus \mathcal T^1(\mathcal H) \to L^1(\Xi) \oplus \mathcal T^1(\mathcal H)$ with $\Gamma_\ast(L^1(\Xi)) \subset \mathcal T^1(\mathcal H)$ and $\Gamma_\ast(\mathcal T^1(\mathcal H)) \subset L^1(\Xi)$, preserving positivity, satisfying $\alpha_{x}(\Gamma_\ast(f,A)) = \Gamma_\ast(\alpha_{x}(f,A))$ and $\int_\Xi \Gamma_\ast(A) + \tr(\Gamma_\ast(f)) = \int_\Xi f + \tr(A)$. Assume that $f \in L^1(\Xi) \cap L^\infty(\Xi)$ and $A \in \mathcal T^1(\mathcal H)$ with $f, A \geq 0$. Then, it is:
        \begin{align*}
            \tr(\Gamma(f)\alpha_{x}(A)) &= \tr(\Gamma(\alpha_{-x}(f)) A) \\
            &= \int_\Xi \alpha_{-x}(f)(y) \Gamma_\ast (A)(y)~\d y.
        \end{align*}
        Since the supports occurring in the following integrals are $\sigma$-finite and the integrands non-negative, we may apply Tonelli's theorem to get:
        \begin{align*}
            \int_\Xi &\tr(\Gamma(f) \alpha_{x}(A))~\d x \\
            &= \int_\Xi \int_\Xi \alpha_{-x}(f)(y) \Gamma_\ast (A)(y)~\d y ~\d x\\
            &= \int_\Xi \Gamma_\ast(A)(y) \int_\Xi f(y-x)~\d x ~\d y\\
            &= \| \Gamma_\ast(A)\|_{L^1(\Xi)} \| f\|_{L^1(\Xi)}.
        \end{align*}
        Since
        \begin{align*}
            \int_\Xi &\tr(\Gamma(f) \alpha_{x}(A))~\d x = \int_\Xi \Gamma(f) \ast (\beta_-(A))(x)~\d x \\
            &= \tr(\Gamma(f)) \tr(\beta_-(A)) = \tr(\Gamma(f)) \tr(A),
        \end{align*}
        it follows that $\tr(\Gamma(f)) \tr(A) = \| f\|_{L^1(\Xi)} \| \Gamma_\ast(A)\|_{L^1(\Xi)}$. In particular, $\Gamma|_{L^1(\Xi) \cap L^\infty(\Xi)}$ extends to a continuous map $\widetilde{\Gamma}: L^1(\Xi) \to \mathcal T^1(\mathcal H)$. 

    Now, let $H \subset \Xi$ be a $\sigma$-compact open subgroup (e.g., the subgroup generated by an open and relatively compact neighborhood of the zero element). Then, the map
    \begin{align*}
        H \supset S \mapsto \widetilde{\Gamma}(\mathbf 1_{S}) \in \mathcal T^1(\mathcal H),
    \end{align*}
    where $S$ is a Borel subset of $H$, is a vector measure for the $\sigma$-finite measure space $(H, \mathcal B(H), \mu|_H)$ (where $\mathcal B(H)$ denotes the Borel-$\sigma$-algebra of $H$ and $\mu|_H$ the restricted Haar measure). Since 
    \begin{align*}
        \| \widetilde{\Gamma}(\mathbf 1_S)\|_{\mathcal T^1} = c\| \mathbf 1_S\|_{L^1(\Xi)} = c \mu(S) < \infty
    \end{align*}
    for some constant $c$, it is not hard to verify that this vector measure is of bounded variation on subsets of $H$ with finite measure. Further, it is easily seen to be absolutely continuous with respect to $\mu|_H$. By the Radon-Nikod\'{y}m property, there exists a unique $h_H: H \to \mathcal T^1(\mathcal H)$ such that
    \begin{align*}
        \widetilde{\Gamma}(S) = \int_{S} h_H(y)~\d y.
    \end{align*}
    Since $\widetilde{\Gamma}$ is covariant, it follows for $x \in H$ and $S \subset H$ of finite measure:
    \begin{align*}
        \int_S h_H(y+x)~\d y &= \int_{S+x} h_H(y)~\d y\\ &= \widetilde{\Gamma}(\alpha_{x}(\mathbf 1_S)) = \alpha_{x}(\widetilde{\Gamma}(f)) \\
        &= U_{x} \int_S f(y) h_H(y)~\d y U_{x}^\ast\\
        &= \int_S f(y) U_{x}h_H(y) U_{x}^\ast ~\d y.
    \end{align*}
    Uniqueness of $h_H$ now implies that $h_H(y+x) = U_{x} h_H(y) U_{x}^\ast$ $y$-almost everywhere. Consider $g_H(y) = \alpha_{-y}(h_H(y))$. Then, $g_H$ is again Bochner integrable and satisfies (in an a.e.-sense) $g_H(x+y) = g_H(y)$. By \cite[Lemma 4(b)]{Kiukas_Lahti_Ylinen2006} (which was formulated for second-countable groups, but the proof works for arbitrary $\sigma$-compact groups), there exists a $B_H \in \mathcal T^1(\mathcal H)$ such that $g_H(y) = B_H$ and hence $h_H(y) = \alpha_y(B_H)$ almost everywhere. In particular, $h_H$ is essentially bounded on all of $H$. Further, uniqueness of $h_H$ enforces that $B_H$ is also uniquely determined. Now, by the dominated convergence theorem (for Bochner integrals), we obtain for every $f \in L^1(H)$ that
    \begin{align*}
        \widetilde{\Gamma}(f) = \int_H f(y) \alpha_y(B_H)~dy = f \ast B_H.
    \end{align*}
    For arbitrary $f \in L^1(\Xi)$, one now uses that $f$ is always supported within some $\sigma$-compact open subgroup $H$ of $\Xi$, hence we obtain $\widetilde{\Gamma}(f) = f \ast B_H$ for every $f \in L^1(\Xi)$. The proof is concluded by the fact that $B_H$ does not depend on the choice of the subgroup $H$. Indeed, if $H_1, H_2$ are two open and $\sigma$-compact subgroups of $\Xi$, then the subgroup $H$ generated by $H_1$ and $H_2$ is again open and $\sigma$-compact. By uniqueness of $B_H$, $B_{H_1}$ and $B_{H_2}$, respectively, we see that $B_{H_1} = B_H = B_{H_2}$, which finishes this part of the proof.

    (ii) Let $0 \leq A \in \mathcal T^1(\mathcal H)$ and $0 \leq g \in L^1(\Xi)$. Then, we have (where $\Gamma_\ast$ is again the pre-adjoint, $\Gamma_\ast: L^1(\Xi) \to \mathcal T^1(\mathcal H)$, of $\Gamma$)
    \begin{align*}
        \int_\Xi \Gamma(A) \ast g(x) ~dx &= \int_\Xi \langle \Gamma(A), \alpha_x(\beta_-(g))\rangle~dx\\
        &= \int_\Xi \langle A, \alpha_x \Gamma_\ast (\beta_-(g))\rangle~dx \\
        &= \int_\Xi A \ast \beta_-(\Gamma_\ast(\beta_-(g)))(x)~dx\\
        &= \tr(A) \tr(\Gamma_\ast(\beta_-(g)))
    \end{align*}
    by Proposition \ref{prop:thirdconv}. Since
    \begin{align*}
        \tr(\Gamma_\ast(\beta_-(g))) = \langle \Gamma_\ast(\beta_-(g)), \Id\rangle = \langle \beta_-(g), \Gamma(\Id)\rangle = \langle \beta_-(g), 1\rangle = \| g\|_{L^1},
    \end{align*}
    we arrive at $\int_\Xi \Gamma(A) \ast g(x)~dx = \| A\|_{\mathcal T^1}\| g\|_{L^1}$. Since $\Gamma(A)$ is continuous for $A \in \mathcal T^1(\mathcal H)$, which follows from the covariance and 
    \begin{align*}
        \| \Gamma(A) - \alpha_x(\Gamma(A))\|_\infty &= \| \Gamma( A - \alpha_x(A))\|_\infty \leq \| A - \alpha_x(A)\|_{op}\\
        &\leq \| A - \alpha_x(A)\|_{\mathcal T^1}\to 0, \quad x \to 0,
    \end{align*}
we obtain from Proposition \ref{prop:convprop} that $\Gamma(A) \in L^1(\Xi)$ (meaning that $\Gamma(A)$ has a continuous representative, which is then integrable). Therefore, $\widetilde{\Gamma} := \Gamma|_{\mathcal T^1(\mathcal H)}$ maps $\mathcal T^1(\mathcal H) \to L^1(\Xi)$ continuously, with appropriate covariance, and is positivity and norm preserving. Hence  $\widetilde{\Gamma}^\ast: L^\infty(\Xi) \to \mathcal L(\mathcal H)$ is covariant, weak$^\ast$ continuous, positivity preserving, and we have, for $0 \leq A \in \mathcal T^1(\mathcal H)$:
\begin{align*}
    \langle (\widetilde{\Gamma})^\ast (1), A\rangle = \langle 1, \widetilde{\Gamma}(A)\rangle = \int_\Xi \widetilde{\Gamma}(A) = \tr(A).
\end{align*}
Thus, $\widetilde{\Gamma}^\ast(1) = \Id$. Therefore, $\widetilde{\Gamma}^\ast$ is a map of the kind considered in (i). Hence, there exists a unique $0 \leq B \in \mathcal T^1(\mathcal H)$ such that $\tr(B) = 1$ and $\widetilde{\Gamma}^\ast(f) = B \ast f$. Now, the (unique) predual to that map is
\begin{align*}
    \widetilde{\Gamma}(A) = \beta_-(B) \ast A.
\end{align*}
By density and weak$^\ast$ continuity of $\Gamma$, we arrive at $\Gamma(A) = \beta_-(B) \ast A$ for every $A \in \mathcal L(\mathcal H)$.
    \item These statements follow from the results in \cite{Stormer1974}. Note that the map $\Gamma: L^\infty(\Xi) \oplus \mathcal L(\mathcal H) \to L^\infty(\Xi) \oplus \mathcal L(\mathcal H)$ is pseudo-multiplicative in the sense of \cite{Stormer1974},
    \begin{align*}
        \Gamma((f, A)) = \Gamma((1, I) (f, A)) = (1, I) \Gamma(f, A),\quad (f, A) \in L^\infty(\Xi) \oplus \mathcal L(\mathcal H).
    \end{align*}
    Since $\Gamma$ is assumed positive by definition, it is completely positive by \cite[Theorem 2.2]{Stormer1974}. The estimate is \cite[Theorem 3.1]{Stormer1974}.
    \item This proof works exactly as in \cite{Werner1984}.
    \end{enumerate}
\end{proof}

\section{The Fourier transform of Quantum Harmonic Analysis}\label{sec:fouriertrafo}

In Quantum Harmonic Analysis, there are at least two different Fourier transforms around: The first Fourier transform needed is the Fourier transform mapping functions on $\Xi$ to functions on $\widehat{\Xi}$. The second Fourier transform will be a transform mapping operators on $\mathcal H$ to functions on $\Xi$. If $\Xi = G \times \widehat{G}$ for some lca group $G$, then there is even the standard Fourier transform $\mathcal F: L^2(G) \to L^2(\widehat{G})$. This last Fourier transform will not be particularly important to us, we will focus on the first two Fourier transforms. 

Let us first come to the Fourier transform for functions on $\Xi$. Of course, we could simply treat $\Xi$ as a generic locally compact abelian group and use the standard Fourier theory. It will nevertheless be more suitable to instead use the symplectic Fourier transform, which is possible since we always identify $\widehat{\Xi}$ with $\Xi$ through the bicharacter $\sigma$. The main difference between the standard Fourier transform and the symplectic Fourier transform is the fact that the symplectic Fourier transform is self-inverse. Unfortunately, we couldn't locate a decent reference for the theory of the symplectic Fourier transform on lca groups of the form with general Heisenberg multipliers. Since essentially all the proofs work analogous to the standard group Fourier transform, we only give definitions and results.
For $f \in L^1(\Xi)$, the symplectic Fourier transform of $f$ is given by
\begin{align*}
    \mathcal F_\sigma(f)(\xi) = \int_{\Xi} \sigma(x, \xi) f(x)~\d x, \quad \xi \in \Xi.
\end{align*}
The symplectic Fourier transform satisfies a collection of well-known properties:
\begin{thm}[Riemann-Lebesgue Lemma for functions]
Let $f \in L^1(\Xi)$. Then, it is $\mathcal F_\sigma(f) \in C_0(\Xi)$ with $\| \mathcal F_\sigma(f)\|_\infty \leq \| f\|_{L^1}$.
\end{thm}
\begin{thm}[Injectivity of symplectic Fourier transform]
    Let $f \in L^1(\Xi)$. Then, it is $f = 0$ if and only if $\mathcal F_\sigma(f) = 0$.
\end{thm}
\begin{thm}
Let $f, g \in L^1(\Xi)$. Then, it is:
\begin{enumerate}
    \item $\mathcal F_\sigma(\alpha_{x}(f)) = \sigma(x, (\cdot))\mathcal F_\sigma(f)$, for $x \in \Xi$.
    \item $\mathcal F_\sigma(\sigma(x, \cdot) f) = \alpha_{x}(\mathcal F_\sigma(f))$, for $x \in \Xi$.
    \item \emph{The convolution theorem:} $\mathcal F_\sigma(f \ast g) = \mathcal F_\sigma(f) \cdot \mathcal F_\sigma(g)$.
    \item $\mathcal F_\sigma(L^1(\Xi))$ is dense in $C_0(\Xi)$.
\end{enumerate}
\end{thm}
Equally well, one can define the symplectic Fourier transform of a finite Radon measure $\mu$ on $\Xi$, essentially by the same formula:
\begin{align*}
    \mathcal F_\sigma(\mu)(\xi) = \int_{\Xi} \sigma(x, \xi) \d \mu(x), \quad \xi \in \Xi.
\end{align*}
\begin{prop}
    Let $\mu$ be a finite Radon measure on $\Xi$. Then, $\mathcal F_\sigma(\mu) \in L^\infty(\Xi)$ with $\|\mathcal F_\sigma(\mu)\|_\infty \leq \| \mu\|$, the total variation norm of $\mu$.
\end{prop}
\begin{thm}
The map $\mu \mapsto \mathcal F_\sigma(\mu)$ is injective.
\end{thm}
\begin{thm}[Fourier inversion]
    Let $f \in L^1(\Xi)$ such that $\mathcal F_\sigma(f) \in L^1(\Xi)$. Then, it is $\mathcal F_\sigma(\mathcal F_\sigma(f)) = cf$, where $c > 0$ is independent of $f$. 
\end{thm}
Since it is the previous theorem which distinguishes the symplectic Fourier transform from the standard Fourier transform, we wish to very briefly sketch why this holds true. Recall that we identify $\Xi$ and $\widehat{\Xi}$ by $\Xi \ni x \mapsto \sigma(x, (\cdot))\in \widehat{\Xi}$. Therefore, we have (in an appropriate weak sense, up to the appropriate normalization of the dual Haar measure $d\xi$):
\begin{align*}
    \int_\Xi \sigma(\xi,y) \sigma(x,\xi)~\d \xi = \int_\Xi \frac{\sigma(\xi,y)}{\sigma(\xi,x)}~\d \xi = \int_\Xi \sigma(\xi,y-x)~\d \xi = \delta_{y-x}.
\end{align*}
We therefore obtain:
\begin{align*}
    \mathcal F_\sigma (\mathcal F_\sigma(f))(y) &= \int_\Xi \sigma(\xi,y) \int_\Xi \sigma(x,\xi) f(x)~\d x ~\d \xi\\
    &= \int_\Xi f(x) \int_\Xi \sigma(\xi,y) \sigma(x,\xi)~\d \xi ~\d x\\
    &= \int_\Xi f(x) \delta_{y-x}~\d x\\
    &= f(y).
\end{align*}
While the above reasoning is obviously not rigorous, there is no problem in turning this into a good argument (say, for $f$ a Schwartz-Bruhat or a Feichtinger function). The only thing which is unclear from the above reasoning is the normalization of the Haar measure $\d \xi$, which is done by the constant $c$ in the inversion theorem. 

So far, we have always identified $\Xi$ and $\widehat{\Xi}$ through the symplectic bicharacter $\sigma$. We will continue to identify both groups in this way. Nevertheless, we will now use the symbol $\widehat{\Xi}$ to indicate that the Haar measure is getting normalized as $\d\xi = \frac{1}{c}~\d x$, with $c$ from the previous result. From now on, $\mathcal F_\sigma$ will map functions on $\Xi$ to function on $\widehat{\Xi}$, simply indicating that we renormalize the Haar measure. Further, we will denote 
\begin{align*}
    \mathcal F_\sigma'(f)(y) = \int_{\widehat{\Xi}} \sigma(\xi,y) f(\xi)~\d\xi = \frac{1}{c} \int_\Xi \sigma(x,y) f(x)~\d x.
\end{align*}
Thus, the Fourier inversion formula now reads $\mathcal F_\sigma' \mathcal F_\sigma f = f$ whenever $f \in L^1(\Xi)$, $\mathcal F_\sigma(f) \in L^1(\widehat{\Xi})$.
\begin{thm}[Plancherel's theorem]
With the above renormalization of the Haar measure on  $\widehat{\Xi}$ we have that $\mathcal F_\sigma$ extends to a unitary operator $\mathcal F_\sigma: L^2(\Xi) \to L^2(\widehat{\Xi})$ with inverse $\mathcal F_\sigma': L^2(\widehat{\Xi}) \to L^2(\Xi)$.
\end{thm}
\begin{thm}[Hausdorff-Young inequality]
    For $p \in [1, 2]$ and $q$ the conjugate exponent, $\frac{1}{p} + \frac{1}{q} = 1$, it is:
    \begin{align*}
        \| \mathcal F_\sigma(f)\|_{L^q(\widehat{\Xi})}\leq \| f\|_{L^p(\Xi)}, \quad f \in L^p(\Xi).
    \end{align*}
\end{thm}
\begin{thm}[Bochner's theorem]
    Let $f$ be a bounded continuous function on $\widehat{\Xi}$. Then, $f = \mathcal F_\sigma(\mu)$ for some (positive) finite Radon measure $\mu$ on $\Xi$ if and only if $f$ is positive definite.
\end{thm}
\begin{thm}[Wiener's approximation theorem]\label{thm:wienerfunctions}
Let $S \subset L^1(\Xi)$ be any subset. Then, $\{ \alpha_{x}(f): x \in \Xi, ~f\in S\}$ spans a dense subspace of $L^1(\Xi)$ if and only if
\begin{align*}
    \bigcap_{f \in S} \{ \xi \in \widehat{\Xi}: ~\mathcal F_\sigma(f)(\xi) = 0\} = \emptyset.
\end{align*}
\end{thm}
A set of functions satisfying the conditions in the above theorem is termed \emph{regular}.
\begin{rem}
Note that any function in $C_0(\widehat{\Xi})$ is always supported in some $\sigma$-compact subset of $\widehat{\Xi}$ (for example by Lemma \ref{lemma:Deitmar}, as compactly supported continuous functions are of course contained in $L^1(\widehat{\Xi})$). Hence, when $\Xi$ is no longer $\sigma$-compact, the translates of a single function from $L^1(\Xi)$ can no longer span a dense subspace of $L^1(\Xi)$. Thus, it is necessary at this level of generality to work with the family-version of Wiener's theorem, and when $\Xi$ is not $\sigma$-compact, any family satisfying the assumptions of Wiener's theorem has to be uncountable.
\end{rem}

Having now summarized the properties and definitions of the symplectic Fourier transform, it is our goal to define the Fourier transform for operators and give analogous properties to those of the symplectic Fourier transform.
\begin{defin}
    Let $A \in \mathcal T^1(\mathcal H)$. Then, we define its Fourier transform $\mathcal F_U(A)$ by
    \begin{align*}
        \mathcal F_U(A)(\xi) = \tr(AU_\xi^\ast), \quad \xi \in \widehat{\Xi}.
    \end{align*}
\end{defin}
There is a clear motivation for having this as the definition of a Fourier transform: The operators $U_{\xi}$ should be thought of as the analogues of the characters $\sigma(\xi, (\cdot))$. The (symplectic) Fourier transform $\mathcal F_\sigma(f)$ consists of pairing the function $f$ with the character $\overline{\sigma(\xi, \cdot)}$ and then integrating. Replacing the character by the operator $U_{\xi}$, complex conjugation by taking the adjoint and integration by taking the trace yields immediately the definition at hand.
\begin{rem}
    \begin{enumerate}[(1)]
    \item The index ``$U$'' in the notation $\mathcal F_U(A)$ indicates that the Fourier transform is taken with respect to the representation $U$. Of course, different projective representations of $\Xi$ over $\mathcal H$ give rise to different Fourier transforms on $\mathcal T^1(\mathcal H)$. In previous works on Quantum Harmonic Analysis, the notation was usually $\mathcal F_W(A)$, with ``W'' standing for either ``Weyl'' or ``Wigner'' (depending in the preference of the author). While it would still be appropriate to refer to $\mathcal F_U$ as the Fourier-Weyl-Wigner transform, we find it necessary on this level of generality to notationally fix the precise representation one is dealing with.
    \item We want to emphasize that the Fourier-Weyl transform $\mathcal F_U(A)$ has also appeared under other names in the literature. Just to give one example, in \cite{Feichtinger_Kozek1998} (and further follow up literature in the setting of time-frequency analysis) it was named the \emph{spreading function}. If the operator $A$ is simply a rank one operator $A = f \otimes g$ in the standard setting $\Xi = G \times \widehat{G}$, then $\mathcal F_U(A)$ is essentially the short-time Fourier transform $V_g(f)(z) = \langle f, U_z g\rangle$.
    \end{enumerate}
\end{rem}
\begin{prop}\label{prop:convolutions}
    Let $f \in L^1(\Xi)$ and $A, B \in \mathcal T^1(\mathcal H)$. Then, the following properties hold true:
    \begin{enumerate}[(1)]
        \item $\mathcal F_U(f \ast A) = \mathcal F_\sigma(f) \cdot \mathcal F_U(A)$.
        \item $\mathcal F_\sigma(A \ast B)(\xi) = m(\xi, -\xi) \mathcal F_U(A)(\xi) \mathcal F_U(B)(\xi)$.
        \item $\mathcal F_U(\alpha_{\eta}(A))(\xi) = \sigma(\eta, \xi) \mathcal F_U(A)(\xi)$.
        \item $\alpha_{\eta}(\mathcal F_U(A))(\xi) = m(-\eta, \xi) \mathcal F_U(U_{-\eta}^\ast A)(\xi) = m(\xi,-\eta) \mathcal F_U(AU_{-\eta}^\ast)(\xi)$.
    \end{enumerate}
\end{prop}
\begin{proof}
    \begin{enumerate}[(1)]
        \item This follows from immediate computations:
        \begin{align*}
            \mathcal F_U(f \ast A)(\xi) &= \tr( \int_{\Xi} f(y) \alpha_{y}(A) ~\d y~ U_{\xi}^\ast)\\
            &= \int_{\Xi} f(y) \tr( U_{y} A U_{y}^\ast U_{\xi}^\ast) ~\d y\\
            &= \int_\Xi f(y) \sigma(y, \xi) \tr(AU_{\xi}^\ast)~\d y\\
            &= \mathcal F_\sigma(f)(\xi) \cdot \mathcal F_U(A)(\xi).
        \end{align*}
        \item The following computations show the result:
        \begin{align*}
            \mathcal F_\sigma(A \ast B)(\xi) &= \int_\Xi \sigma(y,\xi) A \ast B(y) ~\d y\\
            &= \int_\Xi \sigma(\xi, y) \tr(A U_{y} \beta_-(B)U_{y}^\ast)~\d y\\
            &= \int_\Xi \sigma(\xi, y) \tr(U_{\xi}^\ast A U_{y} \beta_-(B)U_{y}^\ast U_{\xi})~\d y\\
            &= \int_\Xi \tr(U_{\xi}^\ast A U_{y} \beta_-(BU_{-\xi}) U_{y}^\ast)~\d y
            \intertext{Using now Proposition \ref{prop:thirdconv} it is:}
            &= \tr(U_{\xi}^\ast A) \tr(\beta_-(BU_{-\xi})) \\
            &= \mathcal F_U(A)(\xi) \tr(BU_{-\xi})\\
            &=m(\xi,-\xi) \mathcal F_U(A)(\xi) \mathcal F_U(B)(\xi).
        \end{align*}
        \item and (4) are immediate consequences of the CCR relations.
    \end{enumerate}
\end{proof}
The moralic outcome of point (4) is the following: Shifting the Fourier transform should be the same as taking the Fourier transform of some modulation. With respect to these formulas, it seems that there is no natural analogue of modulation for operators in general. Nevertheless, with some additional structure there is a better formula. We want to emphasize that the character $m_a$ from Example \ref{ex:2regularweyl} satisfies the hypothesis of the following lemma.
\begin{lem}\label{lem:5.15}
    Let $\Xi$ be 2-regular and $m$ be a bicharacter such that $m(x,y) = \overline{m(y,x)}$ for all $x, y \in \Xi$. Then, we have for $A \in \mathcal T^1(\mathcal H)$:
    \begin{align*}
        \mathcal F_U(U_{\frac{x}{2}} A U_{\frac{x}{2}})(\xi) = \mathcal F_U(A)(\xi - x).
    \end{align*}
\end{lem}
\begin{proof}
    We first note that the assumptions on $m$ imply
    \begin{align*}
        m(x,-x) = \frac{1}{m(x,x)} = \frac{1}{\overline{m(x,x)}} = \overline{m(x,-x)}
    \end{align*}
    such that $m(x,-x) \in \{ -1, 1\}$. Since $m(x, -x) = m(\frac{x}{2}, -\frac{x}{2})^4$, and we know that $m(\frac{x}{2}, -\frac{x}{2}) \in \{ -1, 1\}$, we necessarily have $m(x, -x) = 1$. In particular, we see that $U_{-x} = U_x^\ast$. Hence,
    \begin{align*}
        \mathcal F_U(U_{\frac{x}{2}}A U_{\frac{x}{2}})(\xi) &=\tr(A U_{\frac{x}{2}} U_{-\xi} U_{\frac{x}{2}})\\
        &= m(\frac{x}{2}, -\xi) m(\frac{x}{2}-\xi, \frac{x}{2}) \tr(A U_{\xi - x})\\
        &= \frac{m(\xi, \frac{x}{2})m(\frac{x}{2}, \frac{x}{2})}{m(\xi, \frac{x}{2})} \mathcal F_U(A)(\xi - x).
    \end{align*}
    Since $m(x,x) = \frac{1}{m(x,-x)} = 1$ for every $x \in \Xi$, we see that the coefficient equals $1$, which finishes the proof.
\end{proof}
The previous lemma shows that, under the assumptions of the lemma, $U_{\frac{x}{2}}A U_{\frac{x}{2}}$ plays the role of the \emph{modulation} of the operator $A$ by $x$. This structure has already been observed in a special case in \cite{Berge_Berge_Fulsche}. Here, we will not pursue this modulation of operators any further, leaving it as a topic for further studies.

\begin{prop}[Riemann-Lebesgue Lemma for $\mathcal F_U$]
    Let $A \in \mathcal T^1(\mathcal H)$. Then, it is $\mathcal F_U(A) \in C_0(\widehat{\Xi})$ with $\| \mathcal F_U(A)\|_\infty \leq \| A\|_{\mathcal T^1}$.
\end{prop}
\begin{proof}
The norm estimate is clear. Since $\mathcal F_U$ is continuous with respect to the trace norm, it suffices to verify that $\mathcal F_U(A)$ is continuous for $A$ a rank one operator. But if $A = \varphi \otimes \psi$, then $\mathcal F_U(A)(\xi) = \langle \varphi, U_{\xi}\psi\rangle$ and continuity of the Fourier transform follows from the strong continuity of the representation.

For seeing that $\mathcal F_U(A)$ vanishes at infinity, observe that we have $A \ast A \in L^1(\Xi)$, hence $\mathcal F_\sigma(A \ast A)(\xi) = m(\xi, -\xi) \mathcal F_U(A)(\xi)^2 \in C_0(\widehat{\Xi})$. Since $m(\xi, -\xi)$ is of modulus one, we obtain that $\mathcal F_U(A)^2 \in C_0(\widehat{\Xi})$, hence $\mathcal F_U(A) \in C_0(\widehat{\Xi})$.
\end{proof}
\begin{prop}
The Fourier transform $\mathcal F_U: \mathcal T^1(\mathcal H) \to C_0(\widehat{\Xi})$ is injective.
\end{prop}
\begin{proof}
    Since the projective representation is irreducible, Schur's lemma yields that $U' = \{ U_{\xi}: ~\xi \in \Xi\}' = \mathbb C I$ (where the commutant is taken in $\mathcal L(\mathcal H)$). Since the span of $U$ is a self-adjoint unital algebra in $\mathcal L(\mathcal H)$, von Neumann's bicommutant theorem shows that it is dense in $\mathcal L(\mathcal H)$ in weak$^\ast$ topology (w.r.t.\ the predual $\mathcal T^1(\mathcal H)$). Therefore, $\mathcal F_U(A)(\xi) = \overline{m(\xi, -\xi)}\tr(AU_{-\xi}) = 0$ for all $\xi \in \Xi$ implies that $A$ is annihilated by a dense subspace of the dual of $\mathcal T^1(\mathcal H)$, i.e.\ $A = 0$.
\end{proof}

Before discussing the following properties, we introduce the transform that will in the end turn out to be the inverse of $\mathcal F_U$, hence we already denote it in this way. More practically, it is simply the integrated form of the representation $U$.
\begin{defin}
    For $f \in L^1(\widehat{\Xi})$ we set
    \begin{align*}
        \mathcal F_U^{-1}(f) = \int_{\widehat{\Xi}} f(\xi) U_{\xi}~\d \xi.
    \end{align*}
\end{defin}
We recall again that by $\widehat{\Xi}$ we denote $\Xi$ endowed with the dual measure in the sense of Plancherel's theorem. By the strong continuity of the representation, $\mathcal F_U^{-1}(f) \in \mathcal L(\mathcal H)$ is well-defined, in strong operator topology with $\| \mathcal F_U^{-1}(f)\|_{op}\leq \| f\|_{L^1(\widehat{\Xi})}$ (e.g., it could be understood pointwise as a Bochner integral by Theorem \ref{thm:Bochner}). Further, we want to mention that
\begin{align*}
    \mathcal F_U^{-1}(f)^\ast = \mathcal F_U^{-1}(f^{\ast_m}), \quad \text{where} \quad f^{\ast_m}(x) = \overline{f(-x)m(-x,x)}.
\end{align*}
\begin{prop}\label{prop:twistedconv}
    Let $f, g \in L^1(\widehat{\Xi})$. Then, it is $\mathcal F_U^{-1}(f) \mathcal F_U^{-1}(g) = \mathcal F_U^{-1}(f \ast_m g)$. Here, $f \ast_m g$ denotes the twisted convolution
    \begin{align*}
        f \ast_m g(\xi) = \int_{\widehat{\Xi}} f(\xi-\eta) g(\eta) m(\xi-\eta,\eta)~\d \eta.
    \end{align*}
\end{prop}
\begin{proof}
    It is:
    \begin{align*}
        \mathcal F_U^{-1}(f) \mathcal F_U^{-1}(g) &= \int_{\widehat{\Xi}} f(\xi) U_{\xi}~\d \xi \int_{\widehat{\Xi}} g(\eta) U_{\eta}~\d \eta\\
        &= \int_{\widehat{\Xi}} \int_{\widehat{\Xi}} f(\xi) g(\eta) U_{\xi} U_{\eta}~\d \xi~\d \eta\\
        &= \int_{\widehat{\Xi}} \int_{\widehat{\Xi}} f(\xi) g(\eta) m(\xi,\eta)U_{\xi +\eta} ~\d \xi~\d \eta\\
        &= \int_{\widehat{\Xi}} \int_{\widehat{\Xi}} f(\xi-\eta) g(\eta) m(\xi-\eta,\eta)U_{\xi} ~\d \xi~\d \eta
        \intertext{Exchanging the order of integration (which may be justified by pairing with an arbitrary trace class operator and applying Fubini's theorem, which is legal by Lemma \ref{lemma:Deitmar}), it is:}
        &= \int_{\widehat{\Xi}} f \ast_m g(\xi) U_{\xi}~\d \xi.\qedhere
    \end{align*}
\end{proof}
\begin{prop}\label{prop:fourierinversion}
    Let $A \in \mathcal T^1(\mathcal H)$ such that $\mathcal F_U(A) \in L^1(\widehat{\Xi})$. Then, it is $\mathcal F_U^{-1}(\mathcal F_U(A)) = A$. 
\end{prop}
\begin{proof}
    It is, using Lemma \ref{lemma:average}:
    \begin{align*}
        \mathcal F_U^{-1}(\mathcal F_U(A)) &= \int_{\widehat{\Xi}} \mathcal F_U(A)(\xi) U_{\xi}~\d \xi\\
        &= \int_{\widehat{\Xi}} \tr(AU_{\xi}^\ast) U_{\xi}~\d \xi\\
        &= \int_{\widehat{\Xi}} \int_{\Xi} U_{x} A U_{\xi}^\ast U_{x}^\ast ~\d x \ U_{\xi}~\d \xi.
    \end{align*}
    This double integral is, a priori, defined weakly as an element of $\mathcal L(\mathcal H)$. We will prove that the sesquilinear form $\langle \mathcal F_U^{-1}(\mathcal F_U(A)) \varphi, \psi\rangle$ agrees with $\langle A \varphi, \psi\rangle$ for every $\varphi, \psi \in \mathcal H$, which then proves the equality. For this, we have:
    \begin{align*}
       \langle \mathcal F_U^{-1}(\mathcal F_U(A)) \varphi, \psi\rangle &= \int_{\widehat{\Xi}} \int_{\Xi} \langle U_{x} A U_{\xi}^\ast U_{x}^\ast  U_{\xi} \varphi, \psi\rangle~\d x~\d \xi\\
       &= \int_{\widehat{\Xi}} \int_{\Xi} \sigma(\xi, x) \langle AU_{x}^\ast \varphi, U_{x}^\ast \psi\rangle~\d x~\d\xi\\\
       &= \int_{\widehat{\Xi}} \int_\Xi \sigma(\xi, x) (A \ast (R\varphi \otimes R\psi))(x)~\d x~\d \xi\\
       &= \int_{\widehat{\Xi}} \mathcal F_\sigma(A \ast (R\varphi \otimes R\psi))(-\xi)~\d\xi
\end{align*}
Since 
\begin{align*}
\mathcal F_\sigma(A \ast (R\varphi \otimes R\psi))(-\xi) = m(-\xi, \xi)\mathcal F_U(A)(\xi) \mathcal F_U(R\varphi \otimes R\psi)(\xi)
\end{align*}
and since, furthermore, $|m(\xi, -\xi)| = 1$, $\mathcal F_U(A) \in C_0(\widehat{\Xi})$ and 
\begin{align*}
    F_U(R\varphi \otimes R\psi)(\xi) = \tr((R\varphi \otimes R\psi)U_\xi^\ast) = \langle R\varphi, U_\xi R\psi\rangle \in L^2(\widehat{\Xi})
\end{align*}
as a function of $\xi$ by Theorem \ref{thm:moyal}, we obtain $\mathcal F_\sigma(A \ast (R\varphi \otimes R\psi)) \in L^2(\widehat{\Xi})$, which justifies the use of the Fourier inversion in the following. We therefore obtain, using the identity $\mathcal F_\sigma(\beta_-(g)) = \beta_-(\mathcal F_\sigma(g))$:
\begin{align*}
      \langle \mathcal F_U^{-1}(\mathcal F_U(A))\varphi, \psi\rangle &= \mathcal F_\sigma[\mathcal F_\sigma(A \ast (R \varphi \otimes R\psi))(-(\cdot))](0)\\
      &= \mathcal F_\sigma(\mathcal F_\sigma(A \ast (R \varphi \otimes R\psi)))(0)\\
       &= A \ast (R\varphi \otimes R\psi))(0)\\
       &= \tr(A (\varphi \otimes \psi)) = \langle A \varphi, \psi\rangle.
    \end{align*}
Since $\mathcal F_U^{-1}(\mathcal F_U(A))$ and $A$ are given by the same quadratic form, they agree.
\end{proof}
We obtain the following corollary, which is - at least in the setting of elementary LCA groups - known in a similar form \cite[Lemma 7.6.5]{Feichtinger_Kozek1998}.
\begin{cor}\label{lemma:fouriertrafoproduct}
Let $A, B \in \mathcal T^1(\mathcal H)$ with $\mathcal F_U(A), \mathcal F_U(B) \in L^1(\widehat{\Xi})$. Then, the following holds true:
\begin{align*}
    \mathcal F_U(AB) = \mathcal F_U(A) \ast_m \mathcal F_U(B).
\end{align*}
\end{cor}
\begin{proof}
Propositions \ref{prop:twistedconv} and \ref{prop:fourierinversion} yield:
    \begin{align*}
        \mathcal F_U(AB) &= \mathcal F_U(\mathcal F_U^{-1}(\mathcal F_U(A)) \mathcal F_U^{-1}(\mathcal F_U(B))) \\
        &= \mathcal F_U( \mathcal F_U^{-1}(\mathcal F_U(A) \ast_m \mathcal F_U(B))) \\
        &= \mathcal F_U(A) \ast_m \mathcal F_U(B).
    \end{align*}
    This finishes the proof.
\end{proof}
\begin{cor}
 For $A \in \mathcal T^1(\mathcal H)$ such that $\mathcal F_U(A) \in L^1(\widehat{\Xi})$ it holds true that
    \begin{align*}
    \| A\|_{\mathcal T^2} = \| \mathcal F_U(A)\|_{L^2(\widehat{\Xi})}.
    \end{align*}
\end{cor}
\begin{proof}
   This follows immediately from the previous corollary: Since we have $\mathcal F_U(A^\ast)(\xi) = \overline{m(-\xi, \xi) \mathcal F_U(A)(-\xi)}$, which is readily verified, we have
    \begin{align*}
        \tr(AA^\ast) &= \mathcal F_U(AA^\ast)(0) = \mathcal F_U(A) \ast_m \mathcal F_U(A^\ast)(0)\\
        &= \int_{\widehat{\Xi}} \mathcal F_U(A)(-\eta) \overline{\mathcal F_U(A)(-\eta)} m(-\eta, \eta)\overline{m(-\eta, \eta)}~\d\eta\\
        &= \| \mathcal F_U(A)\|_{L^2(\widehat{\Xi})}^2.\qedhere
    \end{align*}
\end{proof}
We have now reached the point where we need to assume integrability of the representation. Indeed, for $\varphi, \psi \in \mathcal H$ we have
\begin{align*}
    \mathcal F_U(\varphi \otimes \psi)(\xi) = \langle \varphi, U_\xi \psi\rangle.
\end{align*}
Therefore, integrability of the representation ensures that there exists a dense subspace of $\mathcal T^1(\mathcal H)$ such that $\mathcal F_U(A) \in L^1(\widehat{\Xi})$ for each $A$ from this dense subspace. In particular, this ensures that the identity $\| A\|_{\mathcal T^2} = \| \mathcal F_U(A)\|_{L^2(\widehat{\Xi})}$ is satisfied by $A$ from a dense subspace of $\mathcal T^1(\mathcal H)$. Combining all the previous results, we have obtained:
\begin{thm}[Plancherel's theorem for operators]
Assume the representation is integrable. Then, the Fourier transform extends to a unitary operator $\mathcal F_U: \mathcal T^2(\mathcal H) \to L^2(\widehat{\Xi})$ with adjoint $\mathcal F_U^{-1}: L^2(\widehat{\Xi}) \to \mathcal T^2(\mathcal H)$.
\end{thm}
The following consequence is in sharp contrast to the behaviour of the (symplectic) Fourier transform of functions:
\begin{cor}
   Assume the representation is integrable.  Then, for every $A \in \mathcal T^1(\mathcal H)$ it is $\mathcal F_U(A)\in C_0(\widehat{\Xi}) \cap L^2(\widehat{\Xi})$. In particular, the formula
    \begin{align*}
        \mathcal F_U(AB) = \mathcal F_U(A) \ast_m \mathcal F_U(B)
    \end{align*}
    holds for all operators $A,B\in \mathcal T^1(\mathcal H)$.
\end{cor}
\begin{proof}
    The first fact simply follows from $\mathcal T^1(\mathcal H) \subset \mathcal T^2(\mathcal H)$. Therefore, the twisted convolution $\mathcal F_U(A) \ast_m \mathcal F_U(B)$ is defined for all trace-class operators, being the twisted convolution of two  $L^2$ functions. The convolution identity follows by standard density reasoning.
\end{proof}
The following fact is now the standard consequence from the Riemann-Lebesgue lemma, Plancherel's theorem and the complex interpolation method (see, e.g., \cite{Bergh_Lofstrom1976}):
\begin{prop}[Hausdorff-Young inequality for $\mathcal F_U$]
   Assume the representation is integrable.  For $p \in [1, 2]$ and $q$ the conjugate exponent, $\frac{1}{p} + \frac{1}{q} = 1$, it holds:
    \begin{align*}
        \| \mathcal F_U(A)\|_{L^q(\widehat{\Xi})} \leq \| A\|_{\mathcal T^p}, \quad A \in \mathcal T^p(\mathcal H).
    \end{align*}
\end{prop}
We want to note that there are also versions of the Riemann-Lebesgue lemma and the Hausdorff-Young inequality for $\mathcal F_U^{-1}$.
\begin{prop}[Riemann-Lebesgue lemma for $\mathcal F_U^{-1}$]
Assume the representation is integrable. Let $f \in L^1(\widehat{\Xi})$. Then, it is $\mathcal F_U^{-1}(f) \in \mathcal K(\mathcal H)$ with $\| \mathcal F_U^{-1}(f)\|_{op} \leq \| f\|_{L^1(\widehat{\Xi})}$.
\end{prop}
\begin{proof}
    For $f \in C_c(\widehat{\Xi}) \subset L^1(\widehat{\Xi})$ we have $\mathcal F_U^{-1}(f) \in \mathcal T^2(\mathcal H)$. Hence, for arbitrary $f \in L^1(\widehat{\Xi})$ we can approximate $\mathcal F_U^{-1}(f)$ by compact operators. The estimate $\| \mathcal F_U^{-1}(f)\|_{op} \leq \| f\|_{L^1(\widehat{\Xi})}$ follows easily.
\end{proof}
\begin{prop}[Hausdorff-Young inequality for $\mathcal F_U^{-1}$]
Assume the representation is integrable. For $p \in [1, 2]$ and $q$ the conjugate exponent, $\frac{1}{p} + \frac{1}{q} = 1$, it holds:
\begin{align*}
    \| \mathcal F_U^{-1}(f)\|_{\mathcal T^q} \leq \| f\|_{L^p(\widehat{\Xi})}, \quad f \in L^p(\widehat{\Xi}).
\end{align*}
\end{prop}
\begin{rem}
    We want to add that the map $\mathcal F_U^{-1} \circ \mathcal F_\sigma$ plays the role of a pseudodifferential quantization rule. Indeed, for $\Xi = G \times \widehat{G}$ and $m((x, \xi), (y, \eta)) = \overline{\langle x, \eta\rangle}$, this map gives rise to the Kohn-Nirenberg correspondence, as it was already discussed in, e.g., \cite{Feichtinger_Kozek1998}. Here, we see that the flexibility of dealing with other multipliers allows for a broad class of pseudodifferential quantizations to be studied from the present point of view. 
\end{rem}
We now come to the operator version of Bochner's theorem. In this result, we say that a function $f: \widehat{\Xi} \to \mathbb C$ is \emph{$m$-twisted positive semidefinite} if it satisfies, for all $\zeta_j \in \mathbb C$ and $x_j \in \widehat{\Xi}$, $j = 1, \dots, n$:
\begin{align*}
    \sum_{j,k=1}^n \zeta_k \overline{\zeta_j} m(-x_k, x_k) \overline{m(x_j, -x_k)} f(x_j - x_k) \geq 0.
\end{align*}
If $m$ is a bicharacter, which we assume in part of the theorem below, the above condition is of course equivalent to:
\begin{align*}
    \sum_{j,k=1}^n \zeta_k \overline{\zeta_j} m(x_j-x_k, x_k) f(x_j - x_k) \geq 0.
\end{align*}
\begin{thm}[Bochner's theorem for operators]\label{thm:bochner_op}
    Let $f: \Xi \to \mathbb C$. 
    \begin{enumerate}[(1)]
        \item If $f = \mathcal F_U(A)$ for some $A \in \mathcal T^1(\mathcal H)$, $A \geq 0$, then $f$ is continuous and $m$-twisted positive semidefinite.
        \item Assume that the multiplier $m$ is a bicharacter. Then, if $f: \Xi \to \mathbb C$ is continuous and $m$-twisted positive-semidefinite, there exists some $A \in \mathcal T^1(\mathcal H)$, $A \geq 0$, such that $f = \mathcal F_U(A)$.
        \item Let $m$ be a bicharacter and $a: \Xi \to \mathbb T$ continuous with $a(e_\Xi) = 1$ and $a(-x) = a(x)$ for all $x \in \Xi$. Then, if $f: \Xi \to \mathbb C$ and $m_a$-twisted positive-semidefinite, there exists some $A \in \mathcal T^1(\mathcal H)$, $A \geq 0$, such that $f = \mathcal F_{U^a}(A)$.
    \end{enumerate}
\end{thm}
Let us note that most interesting examples are either of the form required in (2) or (3) of the result. We want to emphasize that our proof is an adaptation of Kastler's proof \cite[Theorem 7a]{Kastler1965}, which seems to be the first result in this direction (on the phase space $\Xi = \mathbb R^{2d}$), cf.\ also the independent works of Loupias and Miracle-Sole \cite{Loupias_MiracleSole1966, Loupias_MiracleSole1967}.
\begin{proof}[Proof of Theorem \ref{thm:bochner_op}]
\begin{enumerate}[(1)]
    \item This is easily shown similarly to the classical Bochner theorem: For $A \in \mathcal T^1(\mathcal H)$, $A \geq 0$, we have 
    \begin{align*}
\tr(A(\sum_{k=1}^n \zeta_k U_{x_k}) (\sum_{j=1}^n \zeta_j U_{x_j}))^\ast \geq 0.
    \end{align*}
    On the other hand, the trace can be computed as:
    \begin{align*}
\tr(A(\sum_{k=1}^n &\zeta_k U_{x_k}) (\sum_{j=1}^n \zeta_j U_{x_j})^\ast) \\
&= \sum_{j,k=1}^n \zeta_k \overline{\zeta_j} \tr(A U_{x_k} U_{x_j}^\ast)\\
&= \sum_{j,k=1}^n \zeta_k \overline{\zeta_j} ~ m(-x_k, x_k)\tr(AU_{-x_k}^\ast U_{x_j}^\ast)\\
&= \sum_{j,k=1}^n \zeta_k \overline{\zeta_j} ~ m(-x_k, x_k) \overline{m(x_j, -x_k)} \tr(AU_{x_j - x_k}^\ast)\\
&= \sum_{j,k=1}^n \zeta_k \overline{\zeta_j} ~ m(-x_k, x_k) \overline{m(x_j, -x_k)} f(x_j - x_k).
    \end{align*}
    \item  We first want to mention that the (very short) proof given by Werner \cite{Werner1984} in the setting of $G = \mathbb R^n$ is flawed. Indeed, he made use of the (wrong) claim that continuous positive definite functions are Fourier transforms of positive $L^1$ functions. By the correct formulation of Bochner's theorem, these are instead Fourier transforms of positive (regular) measures. Even the class of continuous and positive definite functions that vanish at infinity is strictly larger than the Fourier transforms of $L^1$. We therefore do not know if Werner's approach to the theorem can be repaired. Therefore have to give a longer proof, which closely follows the proof for the case $G = \mathbb R^n$ as it is presented in \cite{Kastler1965}.

Denote by $\mathcal V_0$ the space
\begin{align*}
    \mathcal V_0 = \{ \varphi: \Xi \to \mathbb C; ~\varphi(x) \neq 0 \text{ for only finitely many } x \in \Xi\}.
\end{align*}
For $\varphi, \psi \in \mathcal V_0$ we set
\begin{align*}
    \langle \varphi, \psi\rangle_{\mathcal V_0} := \sum_{\zeta, \eta \in \Xi} \varphi(\zeta) \overline{\psi(\eta)} m(\zeta - \eta, \eta) f(\zeta - \eta).
\end{align*}
By the assumptions on $f$, we see that $\langle \varphi, \varphi\rangle_{\mathcal V_0} \geq 0$ for every $\varphi \in \mathcal V_0$. Further, $\langle \cdot, \cdot \rangle_{\mathcal V_0}$ is clearly sesquilinear. On $\mathcal V_0$ we now set for $z \in \Xi$:
\begin{align*}
    \mathcal U_z\varphi(\zeta) := m(\zeta, z) \varphi(\zeta + z).
\end{align*}
Then, it is
\begin{align*}
    \langle \mathcal U_z & \varphi, \mathcal U_z \psi\rangle_{\mathcal V_0} \\
    &= \sum_{\zeta, \eta \in \Xi} m(\zeta, z) \varphi(\zeta + z) \overline{m(\eta, z)\psi(\eta + z)} m(\zeta - \eta, \eta) f(\zeta - \eta)\\
    &= \sum_{\zeta, \eta \in \Xi} m(\zeta - z, z) \varphi(\zeta) \overline{m(\eta - z, z)\psi(\eta)}m(\zeta - \eta, \eta - z)f(\zeta - \eta)\\
    &= \sum_{\zeta, \eta} \varphi(\zeta) \overline{\psi(\eta)} m(\zeta - \eta, \eta)f(\zeta - \eta)\\
    &= \langle \varphi, \psi\rangle_{\mathcal V_0}.
\end{align*}
Hence, the $\mathcal U_z$ are unitary. Further, for $z, w \in \Xi$ it is
\begin{align*}
    \mathcal U_z \mathcal U_w(\varphi)(\zeta) &= m(\zeta, z) U_w \varphi(\zeta + z) \\
    &= m(\zeta, z) m(\zeta + z, w) \varphi(\zeta + z + w)\\
    &= m(z,w) \mathcal U_{w+z}\varphi(\zeta),
\end{align*}
i.e.\ $\mathcal U_z \mathcal U_w = m(z,w) \mathcal U_{z+w}$. This implies that $\mathcal U_z^\ast = \overline{m(-z, z)}\mathcal U_{-z}$. Formally, we therefore have that $z \mapsto \mathcal U_z$ is a projective unitary representation of $\Xi$ with multiplier $m$. Strictly speaking, this is not the case, as $(\mathcal V_0, \langle \cdot, \cdot\rangle_{\mathcal V_0})$ is not a Hilbert space: Neither is the sesquilinear form positive definite (only positive semidefinite), nor is the space complete. This can be enforced by first passing to the quotient of $\mathcal V_0$ by $\{ \varphi \in \mathcal V_0: ~\langle \varphi, \varphi\rangle_{\mathcal V_0} = 0\}$ and then completing the quotient. Denote the resulting Hilbert space by $\mathcal V$. Then, the operators $\mathcal U_z$ naturally extend to $\mathcal V$, yielding an actual projective representation. We will from now on only consider the space $\mathcal V$ and denote the newly obtained unitary operators also by $\mathcal U_z$. Further, let $\Phi$ be the element of $\mathcal V$ obtained from $\Phi_0\in\mathcal V_0$ where
\begin{align*}
    \Phi_0(\zeta) = \begin{cases}
        1, \quad \zeta = e \text{, the neutral element of } \Xi,\\
        0, \quad \text{else}.
    \end{cases}
\end{align*}
Then, one easily checks that $\Phi \neq 0$ and $\langle \mathcal \Phi, \mathcal U_z \Phi\rangle = f(z)$. Continuity of $f$ shows that this is continuous in $z$. From here, one concludes that $\mathcal U_z$ acts continuously (say, in weak operator topology) on the cyclic subspace generated by $\Phi$. Denote this subspace by $\mathcal V'$.

As a consequence of Theorem \ref{Mackey-Stone-vNeumann}, there is a unitary map
\begin{align*}
    V: \mathcal V' \to \bigoplus_{\iota \in I} \mathcal H,
\end{align*}
where $\mathcal H$ is the representing space of the representation $U$ and $I$ is some index set. Further, we have $V (U_z)_{\iota \in I} V^\ast = \mathcal U_z$. By basic Hilbert space theory, the set $\{ \iota \in I: ~(V\Phi)_{\iota} \neq 0\}$ is at most countable. Let $(\iota_k)_{k \in \mathbb N}$ be an enumeration of this set (with obvious modifications of the following if the set is finite). Write $\varphi_k = (V\Phi)_{\iota_k} \in \mathcal H$ and set $A = \sum_{k=1}^\infty \varphi_k \otimes \varphi_k$. Then, $A\geq 0$ is trace class:
\begin{align*}
    \tr(A) = \sum_{k=1}^\infty \langle \varphi_k, \varphi_k\rangle = \sum_{k=1}^\infty \langle (V\Phi)_k, (V\Phi)_k\rangle = \| V\Phi\|^2 = \| \Phi\|_{\mathcal V}^2 < \infty.
\end{align*}
Further, we have
\begin{align*}
    \mathcal F_U(A)(z) &= \tr(AU_z) = \sum_{k=1}^\infty \tr((\varphi_k \otimes \varphi_k) U_z^\ast) = \sum_{k=1}^\infty \langle \varphi_k, U_z \varphi_k\rangle\\
    &= \langle \Phi, \mathcal U_z  \Phi\rangle = f(z).
\end{align*}
This finishes the proof.
\item This is a simple consequence of (2), as $f$ is twisted positive-definite with respect to $m_a$ if and only if $\overline{a}f$ is twisted positive-definite with respect to $m$. Further, $\mathcal F_U(A) = \overline{a}\mathcal F_{U^a}(A)$.\qedhere
    \end{enumerate}
\end{proof}
The last result in this section is of crucial importance.
\begin{thm}[Wiener's approximation theorem for operators]\label{thm:wienerappr_ops}
    Let $\mathcal R \subset \mathcal T^1(\mathcal H)$. Then, the following are equivalent:
    \begin{enumerate}[(i)]
        \item $\operatorname{span}\{ \alpha_{x}(A): ~x \in \Xi, ~A \in \mathcal R\}$ is dense in $\mathcal T^1(\mathcal H)$. 
        \item $L^1(\Xi) \ast \mathcal R$ is dense in $\mathcal T^1(\mathcal H)$.
        \item $\mathcal T^1(\mathcal H) \ast \mathcal R$ is dense in $L^1(\Xi)$.
        \item $\{ A \ast A: ~A \in \mathcal R\}$ is a regular family in $L^1(\Xi)$ (in the sense of Theorem \ref{thm:wienerfunctions}).
        \item $\cap_{A \in \mathcal R} \{ \xi \in \widehat{\Xi}: ~\mathcal F_U(A)(\xi) = 0\} = \emptyset$.
        \item For $B \in \mathcal L(\mathcal H)$, $A \ast B = 0$ for every $A \in \mathcal R$ implies $B = 0$.
        \item For $f \in L^\infty(\Xi)$, $A \ast f = 0$ for every $A \in \mathcal R$ implies $f = 0$.
    \end{enumerate}
\end{thm}
\begin{proof}
    Throughout the proof, we assume that $\| A\|_{\mathcal T^1} = 1$ for every $A \in \mathcal R$. Of course, no generality is lost upon doing so. Further, we write $\mathcal R = \{ A_\lambda\}_{\lambda \in \Lambda}$ for some index set $\Lambda$.

    $(i) \Rightarrow (ii)$: Given $B \in \mathcal T^1(\mathcal H)$ and $\varepsilon > 0$, there are by assumption finitely many $a_j \in \mathbb C$, $x_j \in \Xi$ and $A_j \in \mathcal R$ such that:
    \begin{align*}
        \| B - \sum_j a_j \alpha_{x_j}(A_j)\|_{\mathcal T^1} < \varepsilon.
    \end{align*}
    Let now $(h_\gamma)_{\gamma \in \Gamma} \subset L^1(\Xi)$ be an approximate identity. Since $\alpha$ acts strongly continuous on $\mathcal T^1(\mathcal H)$, there is $\gamma_0 \in \Gamma$ such that:
    \begin{align*}
        \| \sum_j a_j \alpha_{x_j}(A_j) - h_{\gamma_0} \ast \sum_{j}a_j \alpha_{x_j}(A_j)\|_{\mathcal T^1} < \varepsilon.
    \end{align*}
    Together, we obtain:
    \begin{align*}
        \| B - h_{\gamma_0} \ast \sum_j a_j \alpha_{x_j}(A_j)\|_{\mathcal T^1} < 2\varepsilon.
    \end{align*}
    Since
    \begin{align*}
        h_{\gamma_0} \ast \sum_j a_j \alpha_{x_j}(A_j) = \sum_j a_j \alpha_{x_j}(h_{\gamma_0}) \ast A_j \in L^1(\Xi) \ast \mathcal R,
    \end{align*}
    the implication follows.
    
    $(iv) \Leftrightarrow (v)$: This is clear from the convolution formula, Proposition \ref{prop:convolutions}(2).

    $(iii) \Leftrightarrow (iv)$: Since $\mathcal T^1(\mathcal H) \ast \mathcal R$ is an $\alpha$-invariant subspace of $L^1(\Xi)$ containing all $A \ast A$ for $A \in \mathcal R$, this equivalence follows from Theorem \ref{thm:wienerfunctions}.

    $(iii) \Leftrightarrow (vii)$: Let us denote by $\bigoplus_{\lambda \in \Lambda, \ell^1} \mathcal T^1(\mathcal H)$ the $\ell^1$ direct sum. Then, the Banach space adjoint of the map
    \begin{align*}
        \bigoplus_{\lambda \in \Lambda, \ell^1} \mathcal T^1(\mathcal H) \ni \{ B_\lambda\}_{\lambda \in \Lambda} \mapsto \sum_{\lambda \in \Lambda} B_\lambda \ast A_\lambda \in L^1(\Xi)
    \end{align*}
    is the map
    \begin{align*}
        L^\infty(\Xi) \ni f \mapsto \{ f \ast \beta_{-}(A_\lambda)\}_{\lambda \in \Lambda} \in \bigoplus_{\lambda \in \Lambda, \ell^\infty} \mathcal L(\mathcal H),
    \end{align*}
    where the latter means the $\ell^\infty$ direct sum of the involved Banach spaces. Now, a map between two Banach spaces has dense range if and only if its dual map is injective. Since the first map has dense range if and only if $\mathcal T^1(\mathcal H) \ast \mathcal R$ is dense in $L^1(\Xi)$, the result follows.

    $(ii) \Leftrightarrow (vi)$: This works analogously to the previous point.

    $(iii) \Rightarrow (ii)$: Let $B \in \mathcal T^1(\mathcal H)$ and $\varepsilon > 0$. Since $L^1(\Xi)$ contains an approximate identity and $\alpha$ acts strongly continuous on $\mathcal T^1(\mathcal H)$, there is $f \in L^1(\Xi)$ such that $\| B - f\ast B\|_{\mathcal T^1} < \varepsilon$. Now, there are finitely many $C_\lambda \in \mathcal T^1(\mathcal H)$ such that $\| f - \sum_\lambda C_\lambda \ast A_\lambda\|_{L^1(\Xi)} < \varepsilon$. This yields:
    \begin{align*}
        \| B - \sum_\lambda (B \ast C_\lambda) \ast A_\lambda\|_{\mathcal T^1} \leq 2 \varepsilon.
    \end{align*}

    $(vi) \Rightarrow (v)$: For every $A \in \mathcal T^1(\mathcal H)$ it is $A \ast U_{x}(y) = g_{x}(y)\mathcal F_U(A)(x)$, where $g_x$ is some function with $|g_x(y)| = 1$. Since $U_{x} \neq 0$, there exists some $\lambda \in \Lambda$ such that $A_\lambda \ast U_{x} \neq 0$, hence there is $y \in \Xi$ such that $A_\lambda \ast U_{x}(y) =  g_x(y) \mathcal F_U(A_\lambda)(x) \neq 0$ for every $x \in \Xi$.

    $(vi) \Rightarrow (i)$: Assume that $\tr(B \alpha_{x}(A_\lambda)) = 0$ for every $x \in \Xi$ and $\lambda \in \Lambda$. This means that $\beta_{-}(B) \ast A_\lambda (-x) = 0$ for every $x \in \Xi$ and $\lambda \in \Lambda$. By $(vi)$, this gives $\beta_{-}(B) = 0$, hence $B=0$. Therefore, $(vi)$ implies $(i)$.
\end{proof}

Note that a more extensive characterization of regular sets of functions in $L^1(\Xi)$ can be made:
\begin{thm}
    Let $S \subset L^1(\Xi)$. Then, additionally to the condition from Theorem \ref{thm:wienerfunctions}, $S$ being regular is equivalent to each of the following statements:
    \begin{enumerate}[(i)]
        \item $S \ast L^1(\Xi)$ is dense in $L^1(\Xi)$.
        \item $S \ast \mathcal T^1(\mathcal H)$ is dense in $\mathcal T^1(\mathcal H)$.
        \item $g \in L^\infty(\Xi)$ and $f \ast g = 0$ for every $f \in S$ implies $g = 0$.
        \item $B \in \mathcal L(\mathcal H)$ and $f \ast B = 0$ for every $f \in S$ implies $B = 0$.
    \end{enumerate}
\end{thm}
The proof is now very similar to that of the previous result and omitted.

There are analogous results regarding $p$-regular families of operators for $p = 2, \infty$. For $G = \mathbb R^n$, they have been proven in \cite{Kiukas_Lahti_Schultz_Werner2012}, see also \cite[Section 7]{luef_skrettingland2018}. The analogous statements can be made for arbitrary $\Xi, m$ with the same modifications of the statements and proofs as discussed for the case $p = 1$ above. Since $p = 1$ is by far the most important case, we defer from presenting the other cases here. Let us end this section by discussing the known results on the existence of regular operators.

\begin{rem}
\begin{enumerate}
\item As already noted above, $\mathcal F_U(A)$ is supported on a $\sigma$-compact subset of $\widehat{\Xi}$. Hence, the search for a regular family consisting only of one element can only be successful provided $\Xi$ itself is $\sigma$-compact, which is for example given if $\Xi = G \times \widehat{G}$, where $G$ is second countable. 
\item The formula $\mathcal F_U(\varphi \otimes \psi)(\xi) = \langle \varphi, U_\xi \psi\rangle$, which follows right from the definition, can be useful when searching for regular operators. To give an example, when $G = \mathbb R$, $\Xi = G \times \widehat{G}$ and $U_{(x,\xi)} = M_\xi T_x$ this proves that $\varphi \otimes \varphi$ is regular for $\varphi(x) = e^{-x^2}$. 
\item When $\Xi = \Xi_1 \times \dots \times \Xi_n$ and $U$ being the tensor product of the projective representations $U_1, \dots, U_n$, we have for $A = A_1 \otimes \dots \otimes A_n \in \mathcal T^1(\mathcal H)$, $\mathcal H = \mathcal H_1 \otimes \dots \otimes \mathcal H_n$:
\begin{align*}
    \mathcal F_U(A)(\xi) = \mathcal F_{U_1}(A_1)(\xi_1) \cdot \dots \cdot \mathcal F_{U_n}(A_n)(\xi_n)
\end{align*}
in suggestive notation. This can be used to build regular operators on groups with a product structure.
\item If $A \in \mathcal T^1(L^2(G))$ is regular, then $\mathcal F A \mathcal F^{-1} \in \mathcal T^1(L^2(\widehat{G}))$ is regular with respect to the dual representation, since
\begin{align*}
    \mathcal F_U(A)(x, \xi) &= \tr(AU_{(x, \xi)}^\ast) = \tr(\mathcal F AU_{(x, \xi)}^\ast\mathcal F^{-1}) = \tr( \mathcal F A \mathcal F^{-1} \widehat{U}_{(\xi, x)}^\ast) \\
    &= \mathcal F_{\widehat{U}}(\mathcal F A \mathcal F^{-1})(\xi, x).
\end{align*}
More generally, regular families on $G$ can be transformed into regular families on $\widehat{G}$ in this way.
\item In \cite{Kiukas_Lahti_Schultz_Werner2012} a regular operator on $\ell^2(\mathbb Z)$ was given: For a constant $c \in \mathbb C$, $0 < |c| < 1$, set
\begin{align*}
    \psi(n) = \begin{cases}
            c^n, \quad &n \geq 0,\\
            0, \quad &n < 0.
    \end{cases}
\end{align*}  
Then, it is for $(k, e^{i\theta}) \in \mathbb Z \times \mathbb T$ with $k > 0$:
\begin{align*}
    \langle \psi, U_{(k, e^{i\theta})}\psi\rangle_{L^2(\mathbb Z)} &= \frac{1}{c^k} \sum_{n=k}^\infty |c|^ne^{-i\theta n} = \frac{1}{c^k |c|^k}e^{i\theta k} \frac{1}{1-|c|e^{-i\theta}} \neq 0.
\end{align*}
For $k = 0$ it is:
\begin{align*}
    \langle \psi, U_{(0, e^{i\theta})} \psi\rangle_{L^2(G)} = \frac{1}{1-|c|e^{-i\theta}} \neq 0.
\end{align*}
If $k < 0$, then finally:
\begin{align*}
    \langle \psi, U_{(k, e^{i\theta})}) \psi \rangle_{L^2(G)} = \frac{1}{c^k} \frac{1}{1-|c|e^{-i\theta}} \neq 0.
\end{align*}
Hence $\psi \otimes \psi$ is regular.
\item The span of the positive rank one operators is dense in $\mathcal T^1(\mathcal H)$, hence
\begin{align*}
    \{ \varphi \otimes \varphi: ~\varphi \in \mathcal H, \| \varphi\|_{\mathcal H} = 1\}
\end{align*}
is a regular family. Upon dividing $\{ \varphi \in \mathcal H: ~\| \varphi\| = 1\}$ by the equivalence relation $\varphi \sim \psi \Leftrightarrow \varphi = cU_{x}\psi$, some $c \in \mathbb C$, $|c| = 1$, and some $x \in \Xi$, and picking one representative $\varphi_\mu \in \mu \in \{ \varphi \in \mathcal H: \| \varphi\| = 1\}/\sim$ for each equivalence class $\mu$, the family $\{ \varphi_\mu \otimes \varphi_\mu\}_\mu$ is still a regular family. While not particularly useful in practice, this is probably the best one could hope for on this level of generality.
\item In particular, $\mathcal T^1(\mathcal H) \ast \mathcal T^1(\mathcal H)$ is dense in $L^1(\Xi)$.
\end{enumerate}
\end{rem}

\section{Extending the theory to coorbit spaces}\label{sec:coorbit}
So far, we only considered operators acting on the Hilbert space $\mathcal H$. Indeed, it has turned out very fruitful to consider the methods of QHA for operators acting on other (Banach) spaces, which arise naturally. While some parts of the theory should only be expected to work out in the Hilbert space case (such as Bochner's theorem or the result on positive correspondence rules), it is Wiener's approximation theorem for operators which has found some important applications in these more general settings (cf.\ \cite{fulsche2020correspondence, Fulsche2022}). We will consider the extension of the theory towards operators acting on the coorbit spaces $\operatorname{Co}_p(U)$ (assuming that the projective representation $U$ of $\Xi$ is integrable). For a general perspective on coorbit spaces, we refer to the nice recent survey \cite{Berge2022} and references therein, as well as the original literature on the matter \cite{Feichtinger_Grochenig1989}. Nevertheless, we want to mention that most works on coorbit spaces work only with ordinary unitary representations (i.e., where the multiplier equals $1$), with some notable exceptions such as \cite{Christensen1996}. Essentially all basic results concerning this theory extend to projective unitary representations with straightforward modifications.

It is our main goal to set up the convolutions of functions and operators, as well as their basic mapping properties, and further discuss Wiener's approximation theorem for operators. We will not always give full proofs, as many are very similar to the Hilbert space case, and also omit some details on the constructions.

We assume that $\Xi$ and $m$ are such that the projective representation $U$, acting on the Hilbert space $\mathcal H$, is integrable. Fix an integrable element $0 \neq \varphi_0\in \mathcal H$, i.e., $\int_\Xi |\langle \varphi_0, U_x \varphi_0\rangle|~dx < \infty$ (time-frequency analysts refer to this $\varphi_0$ as the \emph{window function}). Then, the subspace
\begin{align*} 
\operatorname{Co}_1(U) := \{ f \in \mathcal H: ~[ x \mapsto \langle f, U_x \varphi_0 \rangle ] \in L^1(\Xi)\}
\end{align*}
is dense in $\mathcal H$. The expression
\begin{align*}
    \| f\|_{1, \varphi_0} := \int_\Xi |\langle f, U_x \varphi_0\rangle|~dx
\end{align*}
defines a complete norm on $\operatorname{Co}_1(U)$, and different choices of $\varphi_0$ give rise to equivalent norms. Since
\begin{align*}
    \| f\|^2_{\mathcal H} \| \varphi_0\|_{\mathcal H}^2 &= \int_{\Xi} |\langle f, U_x \varphi_0\rangle|^2~dx \leq \sup_{y \in \Xi} |\langle f, U_y \varphi_0\rangle| \int_\Xi |\langle f, U_x \varphi_0\rangle|~dx\\
    &\leq \| f\|_{\mathcal H} \| \varphi_0\|_{\mathcal H} \| f\|_{1, \varphi_0},
\end{align*}
we see that the dense embedding $\operatorname{Co}_1(U) \hookrightarrow \mathcal H$ is also continuous. In particular, if we let $\operatorname{Co}_\infty(U) := \operatorname{Co}_1(U)'$, then $\mathcal H \hookrightarrow \operatorname{Co}_\infty(U)$ continuously. For any $f \in \operatorname{Co}_\infty(U)$ we let $\mathcal W_{\varphi_0}(f)(x) = \langle f, U_x \varphi_0\rangle$ in the sense of the dual pairing denote the wavelet transform of $f$ (with window $\varphi_0$). Then, for $p \in [1, \infty]$:
\begin{align*}
    \operatorname{Co}_p(U) := \{ f \in \operatorname{Co}_\infty(U): \mathcal W_{\varphi_0}(f) \in L^p(\Xi)\}. 
\end{align*}
Upon endowing $\operatorname{Co}_p(U)$ with the norm $\| f\|_{p, \varphi_0} = \| \mathcal W_{\varphi_0}(f)\|_{L^p}$, it turns into a Banach space, and different windows give rise to equivalent norms. The spaces $\operatorname{Co}_p(U)$, the \emph{coorbit spaces}, are independent of the choice of the window function $\varphi_0$. They satisfy some basic properties:
\begin{lem}\label{lemma:prop_coorbit}
    Let $p, p', q\in [1, \infty]$ with  $p \leq q$ and $\frac{1}{p} + \frac{1}{p'} = 1$. Then, the following statements hold true:
    \begin{enumerate}
        \item $\operatorname{Co}_p(U) \hookrightarrow \operatorname{Co}_q(U)$ continuously. If $q < \infty$, the embedding is also densely.
        \item $\operatorname{Co}_2(U) = \mathcal H$.
        \item $\operatorname{Co}_p(U)' \cong \operatorname{Co}_{p'}(U)$ if $p < \infty$, where the duality is induced by the pairing
        \begin{align*}
            \langle f, g\rangle_{\varphi_0} = \int_\Xi \mathcal W_{\varphi_0}(f)(x) \overline{\mathcal W_{\varphi_0}(g)(x)}~dx, \quad f \in \operatorname{Co}_p(U), g \in \operatorname{Co}_{p'}(U). 
        \end{align*}
    \end{enumerate}
\end{lem}
Since $\mathcal W_{\varphi_0}$ maps $\mathcal H$ injectively onto a closed subspace of $L^2(\Xi)$, $\mathcal W_{\varphi_0} \mathcal W_{\varphi_0}^\ast$ is the projection of $L^2(\Xi)$ onto the range of $\mathcal W_{\varphi_0}$. For $g \in L^2(\Xi)$ we have:
\begin{align}
   \mathcal W_{\varphi_0} \mathcal W_{\varphi_0}^\ast g(x) &= \langle \mathcal W_{\varphi_0}^\ast g, U_x \varphi_0\rangle_{L^2(\Xi)} = \langle g, \mathcal W_{\varphi_0} U_x \varphi_0\rangle_{L^2(\Xi)}\notag\\
   &= \int_\Xi g(y) \overline{\langle U_x \varphi_0, U_y \varphi_0\rangle_{\mathcal H}}~dy\notag \\
   &= \int_\Xi g(y) m(y, -y) \overline{\langle U_{-y}U_x \varphi_0, \varphi_0\rangle_{\mathcal H}}~dy\notag\\
   &= \int_\Xi g(y) m(y,-y) \overline{m(-y,x)} \mathcal W_{\varphi_0}(\varphi_0)(x-y)~dy\notag\\
   &= \int_\Xi g(y) \mathcal W_{\varphi_0}(\varphi_0)(x-y) \frac{m(y,-y)}{m(-y,x)}~dy.
\end{align}
This is clearly some form of twisted convolution, hence extends to a continuous projection $\mathcal W_{\varphi_0} \mathcal W_{\varphi_0}^\ast: L^p(\Xi) \to \mathcal W_{\varphi_0}(\operatorname{Co}_p(U))$. In particular, each space $\mathcal W_{\varphi_0} (\operatorname{Co}_p(U))$ is a closed subspace of $L^p(\Xi)$. 

Let us recall some facts from general Banach space theory. For a Banach space $X$ we denote by $\mathcal T^1(X)$ the class of \emph{nuclear operators} (see, e.g., \cite[Chapter 5]{Gohberg_Goldberg_Krupnik2000}. This is defined as the set of all bounded linear operators $A \in \mathcal L(X)$ such that there exists sequences $\varphi_n \in X, \psi_n \in X'$ with $\sum_{n=1}^\infty \| \varphi_n\|_X \| \psi_n\|_{X'} < \infty$ such that
\begin{align*}
    A = \sum_{n=1}^\infty \varphi_n \otimes \psi_n.
\end{align*}
This spaces $\mathcal T^1(X)$ is normed by taking the infimum over all admissible series $\sum_{n=1}^\infty \| \varphi_n\| \|\psi_n\|$. Further, this norm is complete. When $X$ is a Hilbert space, this class agrees with the class of all trace class operators. 

Recall that $X$ has the \emph{approximation property} when for every compact set $K \subset X$ and every $\varepsilon > 0$ there exists a finite rank operator $F$ with $\| x - Fx\| < \varepsilon$ for every $x \in K$. Provided $X$ has the property, the trace functional 
\begin{align*}
    \tr(A) := \sum_{n=1}^\infty \psi_n(\varphi_n)
\end{align*}
is well-defined and bounded on $A \in \mathcal T^1(\mathcal H)$. When the Banach space $X$ satisfies the approximation property and is reflexive, then the dual of the space of nuclear operators can be identified with $\mathcal L(X)$: Indeed, for every bounded linear functional $\Phi \in \mathcal T^1(X)'$, there exists a unique $A \in \mathcal L(X)$ such that $\Phi(N) = \tr(AN)$.

Certain classical spaces are well-known to satisfy the approximation property: Besides Hilbert spaces, all the $L^p$ belong to this class. Further, if $X$ is a Banach space with the approximation property and $Y \subset X$ a closed and complemented subspace (i.e., there exists a bounded projection from $X$ onto $Y$), then $Y$ has the approximation property as well. Since $L^p(\Xi)$ has the approximation property and the operator $\mathcal W_{\varphi_0} \mathcal W_{\varphi_0}^\ast$ extends to a continuous projection from $L^p(\Xi)$ to $\mathcal W_{\varphi_0}(\operatorname{Co}_p(U))$, each space $\mathcal W_{\varphi_0} (\operatorname{Co}_p(U))$ has the approximation property. Since $\mathcal W_{\varphi_0} (\operatorname{Co}_p(U))$ and $\operatorname{Co}_p(U)$ are isometrically equivalent, we obtain that each of the coorbit spaces $\operatorname{Co}_p(U)$ has the approximation property. Hence, we can consider nuclear operators and their traces on each of the coorbit spaces $\operatorname{Co}_p(U)$. 

Let us also recall some facts regarding tensor products of Banach spaces. As a general reference on the matter, see for example \cite{Ryan2002}. When $X$ and $Y$ are Banach spaces, one can consider the \emph{algebraic tensor product} $X \otimes Y$, i.e., the space of linear combinations of formal expressions $x \otimes y$, where $x \in X$ and $y \in Y$. This means, that the elements of $X \otimes Y$ are of the form $A = \sum_{j=1}^n a_j x_j \otimes y_j$, with $a_j \in \mathbb C$. Of course, there is no loss of generality in assuming that $a_j = 1$. On this space, one can consider the norm
\begin{align*}
    \| A \| = \inf \sum_{j=1}^n \| x_j\|_X \| y_j\|_Y,
\end{align*}
where the infimum ranges over all $x_j, y_j$ with $A = \sum_{j=1}^n x_j \otimes y_j$. Completing the algebraic tensor product with respect to this norm leads to the \emph{projective tensor product} $X \widehat{\otimes}_\pi Y$. When $Y = X'$, then every algebraic tensor $A = \sum_{j=1}^n x_j \otimes y_j$ yields a finite rank operator on $X$ acting as $A(\varphi) = \sum_{j=1}^n y_j(\varphi) \cdot x_j$. Hence, we can also embed the projective tensor product $X \widehat{\otimes}_\pi X'$ into $\mathcal T^1(X)$ by extending the above embedding. As a matter of fact, this embedding may fail to be injective. But when $X$ satisfies the approximation property, the natural map from $X \widehat{\otimes}_\pi X'$ to $\mathcal T^1(X)$ is an isomorphism (compare, e.g., \cite[Corollary 4.8]{Ryan2002}). In this case, when $Y_1 \subset X$ and $Y_2 \subset X'$ are subspaces, possibly endowed with a different complete norm but continuously embedded, we can naturally embed $Y_1 \widehat{\otimes}_\pi Y_2$ into $\mathcal T^1(X)$, where the embedding is injective and continuous.

We now come back to coorbit spaces. The operators $U_x^\ast$ leave the space $\operatorname{Co}_1(U)$ invariant and further $x \mapsto U_x^\ast$ is strongly continuous on $\operatorname{Co}_1(U)$. Hence, by duality, $x \mapsto U_x$ acts weakly continuous on $\operatorname{Co}_\infty(U)$. It then leaves each of the spaces $\operatorname{Co}_p(U)$ invariant and acts isometrically on each of them. If $p < \infty$, it also acts strongly continuous on the space. Similarly, one extends the operator $R$ to each of the coorbit spaces. $R$ acts isometrically on each of the spaces $\operatorname{Co}_p(U)$ with respect to the norm $\| \cdot\|_{p, \varphi_0}$ if and only if $R \varphi_0 = \varphi_0$, which we assume in the following.

We now want to imitate the construction of operator convolutions. For this, we define the shift of an operator $A \in \mathcal L(\operatorname{Co}_p(U))$ as
\begin{align*}
    \alpha_x(A) = U_x A U_x^\ast, \quad x \in \Xi.
\end{align*}
For obtaining suitable continuity properties of the map $x \mapsto \alpha_x(A)$, one needs the strong continuity of $x \mapsto U_x$ both on $\operatorname{Co}_p(U)$ and on the dual space. Since this is not given on $\operatorname{Co}_\infty(U)$, one runs into problems when wanting to develop the theory on $\mathcal L(\operatorname{Co}_1(U))$. This can be overcome by different means. One of them would be the following: Let us denote
\begin{align*}
    \operatorname{Co}_{\infty, 0}(U) = \{ f \in \operatorname{Co}_\infty(U): ~\mathcal W_{\varphi_0}(f) \in C_0(\Xi)\}.
\end{align*}
Then, this is a closed subspace of $\operatorname{Co}_\infty(U)$ on which $x \mapsto U_x$ acts strongly continuous. The crucial fact is now that $\operatorname{Co}_{\infty, 0}(U)' \cong \operatorname{Co}_1(U)$. One can now either work with operators $A \in \mathcal L(\operatorname{Co}_{\infty, 0}(U))$ or with operators $A \in \operatorname{Co}_1(U) \widehat{\otimes}_{\pi} \operatorname{Co}_{\infty, 0}(U) \subset \mathcal T^1(\operatorname{Co}_1(U))$. In the special case $\Xi = \mathbb C^n$, such ideas have been used in \cite{Fulsche2022}. For simplicity, we will assume in the following that $p \in (1, \infty)$, even though similar results can be shown with a little more effort for the boundary cases. Nevertheless, we want to emphasize that this extra effort is not only about making proper definitions of the convolutions, which is actually only a minor inconvenience. Indeed, the main theorem of this section, Theorem \ref{thm:wienerappr_ops_pdependent}, is simply wrong when stated like this for $p = \infty$. An idea of the necessary modifications can be obtained from the recent paper \cite{Fulsche2022}, where such issues have been discussed for $\Xi = \mathbb R^{2n}$.

Therefore, we now fix $p \in (1, \infty)$. We set, as usual, for $A \in \mathcal T^1(\operatorname{Co}_p(U))$ and $f \in L^1(\Xi)$:
\begin{align*}
    f \ast A := \int_\Xi f(x) \alpha_x(A).
\end{align*}
As in the Hilbert space case, this defines an element of $\mathcal T^1(\operatorname{Co}_p(U))$ with $\| f\ast A\|_{\mathcal T^1} \leq \| f\|_{L^1} \| A\|_{\mathcal T^1}$ as well as $\tr(f \ast A) = \int f~dx \tr(A)$ (where the trace denotes the nuclear trace). Similarly, we set again for $A, B \in \mathcal T^1(\operatorname{Co}_p(U))$:
\begin{align*}
    A \ast B (x) = \tr(A \alpha_x(\beta_-(B))).
\end{align*}
Here, we of course have $\beta_-(B) := R B R$. The convolution of two operators, $A \ast B$, also clearly extends to $A \in \mathcal T^1(\operatorname{Co}_p(U))$ and $B \in \mathcal L(\operatorname{Co}_p(U))$, yielding a bounded and uniformly continuous function on $\Xi$. Most properties of these convolutions are analogous to the Hilbert space case. The only exception to this is the convolution of two trace class operators. Certainly, the mapping properties of the convolution operator between two operators is of interest. We will not pursue a general discussion here, but restrict to the simplest case. For doing so, we will consider certain tensor product spaces. Here, we will identify the tensor $f\otimes g$ ($f \in \operatorname{Co}_p(U)$ and $g \in \operatorname{Co}_q(U))$) always with the rank one operator $(f \otimes g)(h) = \langle h, g\rangle_{\varphi_0} f$ ($h \in \operatorname{Co}_p(U)$), where $\varphi_0$ is a fixed window function and $\langle \cdot, \cdot \rangle_{\varphi_0}$ is the extension of the Hilbert space inner product, inducing the duality between $\operatorname{Co}_p(U)$ and $\operatorname{Co}_q(U)$, as in \ref{lemma:prop_coorbit}. In particular, the tensor $f \otimes g$ is linear in the first and anti-linear in the second entry.

We will identify the set of nuclear operators $\mathcal T^1(\operatorname{Co}_p(U))$ with the projective tensor product $\operatorname{Co}_p(U) \widehat{\otimes}_\pi \operatorname{Co}_q(U)$ (which is possible since the coorbit spaces have the approximation property). Further, we will consider the projective tensor product $\operatorname{Co}_1(U) \widehat{\otimes}_\pi \operatorname{Co}_1(U)$ as a subspace of $\mathcal T^1(\operatorname{Co}_p(U))$, which is possible since $\operatorname{Co}_1(U)$ continuously embeds into $\operatorname{Co}_p(U)$ and $\operatorname{Co}_q(U)$.

\begin{lem}\label{lem:conv_cont_co}
    Let $A \in \operatorname{Co}_1(U) \widehat{\otimes}_\pi \operatorname{Co}_1(U)$ and $B\in\operatorname{Co}_p(U) \widehat{\otimes}_\pi \operatorname{Co}_q(U)$. Then $A \ast B \in L^1(\Xi)$. Further, 
    \begin{align*}
        \| A \ast B\|_{L^1} \lesssim \| A\|_{\operatorname{Co}_1(U) \widehat{\otimes}_\pi \operatorname{Co}_1(U)} \| B\|_{\operatorname{Co}_p(U) \widehat{\otimes}_\pi \operatorname{Co}_q(U)}
    \end{align*}
    and
    \begin{align*}
        \int_\Xi A \ast B (x) ~dx = \tr(A) \tr(B).
    \end{align*}
\end{lem}
We want to mention that the norms on the right-hand side of the estimate in the above lemma are subject to the choice of a suitable window function, explaining the use of the notation ``$\lesssim$''. 
\begin{proof}[Proof of Lemma \ref{lem:conv_cont_co}]
    We first assume that both $A$ and $B$ are of rank one, i.e.\ $A = \varphi_1 \otimes \varphi_2$ with $\varphi_j \in \operatorname{Co}_1(U)$ and $B = f \otimes g$. By density, it will be sufficient to prove the result for $f \in \operatorname{Co}_p(U) \cap \mathcal H$ and $g \in \operatorname{Co}_q(U) \cap \mathcal H$.

    First, we note that $A \ast B(x) = \langle U_x Rf, \varphi_2\rangle \langle \varphi_1, U_x Rg\rangle$ such that
    \begin{align*}
        \int_\Xi |A \ast B(x)|~dx &= \int_\Xi | \langle U_x Rf, \varphi_2\rangle  \ \langle\varphi_1, U_x R g\rangle|~dx\\
        &= \int_\Xi | \langle U_{-x} f, R \varphi_2\rangle \ \langle R \varphi_1, U_{-x} g\rangle|~dx\\
        &= \int_\Xi |\langle U_x f, R \varphi_2\rangle \ \langle R \varphi_1, U_x g\rangle|~dx\\
        &\leq \left [ \int_\Xi |\langle U_x f, R \varphi_2 \rangle|^p ~dx \right ]^{1/p} \left [ \int_\Xi | \langle R \varphi_1 ,U_x g\rangle|^q ~dx \right ]^{1/q}\\
        &= \| \mathcal W_{R\varphi_2} f\|_{L^p} \| \mathcal W_{R\varphi_1} g\|_{L^q}.
        \end{align*}
        Estimating this is now, in essence, a twisted version of the proof of \cite[Theorem 3.5]{Berge2022}. We give some details for convenience. Fix a window function $0 \neq \varphi_0 \in \operatorname{Co}_1(U)$ such that $R\varphi_0 = \varphi_0$ and $\| \varphi_0\|_{\mathcal H} = 1$. We have, using the relations from \ref{thm:moyal}:
        
        \begin{align*}
        \mathcal W_{R\varphi_2}f(x) &= \langle f, U_x \varphi_2\rangle \langle \varphi_0, \varphi_0\rangle\\
        &=\int_\Xi \langle U_y f, \varphi_0\rangle \overline{\langle U_y U_x R\varphi_2, \varphi_0\rangle}~dx\\
        &= \int_\Xi \langle f, U_y^\ast \varphi_0\rangle \overline{\langle R \varphi_2, U_{y+x}^\ast \varphi_0\rangle} \ \overline{m(y,x)}~dy\\
        &= \int_\Xi \frac{m(y+x,-y-x)}{m(y,-y)m(y,x)} \langle f, U_{-y} \varphi_0\rangle \overline{\langle R \varphi_2, U_{-y-x} \varphi_0\rangle}~dy\\
        &= \int_\Xi \frac{m(x-y, y-x)}{m(-y,y)m(-y,x)} \langle f, U_y \varphi_0\rangle \overline{\langle R \varphi_2, U_{y-x} \varphi_0\rangle}~dy\\
        &= \int_\Xi \frac{m(x-y, y-x)}{m(-y,y)m(-y,x)} \mathcal W_{\varphi_0}(f)(y) \overline{\mathcal W_{\varphi_0}(\varphi_2)(x-y)}~dy\\
        &= [\mathcal W_{\varphi_0}(f) \ast_m' \overline{\mathcal W_{\varphi_0}(\varphi_2)} ](x).
        \end{align*}
        Here, $\psi_1 \ast_m' \psi_2(x)$ denotes the twisted convolution
        \begin{align*}
            \psi_1 \ast_m' \psi_2(x) &= \int_\Xi \frac{m(x-y, y-x)}{m(-y,y)m(-y,x)} \psi_1(y) \psi_2(x-y)~dy.
        \end{align*}
        This twisted convolution clearly satisfies estimates of the form $\| \psi_1 \ast_m' \psi_2\|_{L^p} \leq \| \psi_1\|_{L^p} \| \psi_2\|_{L^1}$ such that we arrive at:
        \begin{align*}
            \| \mathcal W_{R\varphi_2} f\|_{L^p} &= \| \mathcal W_{\varphi_0}(f) \ast_m' \overline{\mathcal W_{\varphi_0}(\varphi_2)}\|_{L^p}\\
            &\leq \| \mathcal W_{\varphi_0}(f)\|_{L^p} \| \mathcal W_{\varphi_0}(\varphi_2)\|_{L^1}\\
            &= \| f\|_{p, \varphi_0} \| \varphi_2\|_{1, \varphi_1}.
            \end{align*}
Analogously, one estimates $\| \mathcal W_{R\varphi_1}(g)\|_{L^q}$. Hence, we have seen that:
\begin{align*}
    \int_\Xi |(\varphi_1 \otimes \varphi_2) \ast (f \otimes g)(x)|~dx \leq \| \varphi_1\|_{1, \varphi_0} \| \varphi_2\|_{1, \varphi_0} \| f\|_{p, \varphi_0} \| g\|_{q, \varphi_0}
\end{align*}
From here, it is standard to obtain the estimate
\begin{align*}
\int_\Xi |A \ast B(x)|~dx \lesssim \| A\|_{\operatorname{Co}_1(U) \widehat{\otimes}_\pi \operatorname{Co}_1(U)} \| B\|_{\operatorname{Co}_p(U) \widehat{\otimes}_\pi \operatorname{Co}_q(U)}
\end{align*}
for arbitrary $A$, $B$ as in the hypothesis of the lemma. The equality $\int_\Xi A \ast B(x)~dx = \tr(A) \tr(B)$ is then also obtained by density.
\end{proof}
Having established the previous result, it now also possible to define the convolution $A \ast f$ with $A \in \operatorname{Co}_1(U) \widehat{\otimes}_\pi \operatorname{Co}_1(U)$ and $f \in L^\infty(\Xi)$ weakly: Given any $N \in \mathcal T^1(\operatorname{Co}_p(U))$, we have:
\begin{align*}
    \langle f \ast A, N\rangle_{tr} &= \int_\Xi f(z)~ \beta_-(A) \ast N(z)~dz,\\
    |\langle f \ast A, N\rangle_{tr}| \leq \| f\|_\infty \| \beta_-(A) \ast N\|_{L^1} &\lesssim \| f\|_\infty \| A\|_{\operatorname{Co}_1(U) \widehat{\otimes}_\pi \operatorname{Co}_1(U)} \| B\|_{\mathcal T^1(\operatorname{Co}_p(U))}.
\end{align*}

We remark that we know from the Hilbert space case that the convolution operation maps $\operatorname{Co}_2(U) \widehat{\otimes}_\pi \operatorname{Co}_2(U) \times \operatorname{Co}_2(U) \widehat{\otimes}_\pi \operatorname{Co}_2(U)$ continuously to $L^1(\Xi)$. By applying complex interpolation methods, one obtains a larger class of tensor products which give rise to $L^1$ convolutions.

As already indicated earlier, one can now extend the convolution to allow one of the factors to be either in $L^\infty(\Xi)$ or in $\mathcal L(\operatorname{Co}_p(U))$. Then, we obtain:
\begin{align*}
    \| f\ast A\|_{op} &\leq \| f\|_{L^1} \| A\|_{op},\\
    \| f \ast A\|_{op} &\leq \| f\|_\infty \| A\|_{\mathcal T^1},\\
    \| A \ast B\|_{\infty} &\leq \| A\|_{\mathcal T^1} \| B\|_{op}.
\end{align*}
For this notion of convolution, we only want to emphasize that the convolution of three operators, $A \ast B \ast C$ with $A \in \operatorname{Co}_1(U) \widehat{\otimes}_\pi \operatorname{Co}_1(U), B \in \operatorname{Co}_p(U) \widehat{\otimes}_\pi \operatorname{Co}_q(U)$ and $C \in \mathcal L(\operatorname{Co}_p(U))$, is still associative: $(A \ast B) \ast C = A \ast (B \ast C)$. This can be proven as in the Hilbert space case and has important consequences concerning the correspondence theory described in \cite{Werner1984}. Discussing this is beyond the scope of this paper and will be picked up again in future work.

While certain natural results on the convolutions of operators on $\operatorname{Co}_p(U)$ fail (e.g., one should not expect to obtain a version of Bochner's theorem), it turns out that Wiener's theorem for operators carries over to this setting. This has turned out to be particularly useful in the first named authors' previous works \cite{fulsche2020correspondence, Fulsche2022}, hence we will give the result here in this general setting.  Note that $U_x^\ast \in \mathcal L(\operatorname{Co}_p(U))$ for any $p \in (1, \infty)$, hence $\mathcal F_U(A)(\xi) = \tr(AU_\xi^\ast)$ is well-defined for any $A \in \mathcal T^1(\operatorname{Co}_p(U))$ and $\xi \in \widehat{\Xi}$.
\begin{thm}\label{thm:wienerappr_ops_pdependent}
    Let $\mathcal R \subset \operatorname{Co}_1(U) \widehat{\otimes}_\pi \operatorname{Co}_1(U)$. Then, the following are equivalent:
    \begin{enumerate}[(i)]
        \item $\operatorname{span} \{ \alpha_x(A): ~x \in \Xi, A \in \mathcal R\}$ is dense in $\mathcal T^1(\operatorname{Co}_p(U))$.
        \item $L^1(\Xi) \ast \mathcal R$ is dense in $\mathcal T^1(\operatorname{Co}_p(U))$.
        \item $\mathcal T^1(\operatorname{Co}_p(U)) \ast \mathcal R$ is dense in $L^1(\Xi)$.
        \item $\{ A \ast A: ~A \in \mathcal R\}$ is a regular family in $L^1(\Xi)$.
        \item $\cap_{A \in \mathcal R} \{ \xi \in \widehat{\Xi}: ~\mathcal F_U(A)(\xi) = 0\} = \emptyset$.
        \item For $B \in \mathcal L(\operatorname{Co}_p(U))$, $A \ast B = 0$ for every $A \in \mathcal R$ implies $B = 0$.
        \item For $f \in L^\infty(\Xi)$, $A \ast f = 0$ for every $A \in \mathcal R$ implies $f = 0$.
    \end{enumerate}
\end{thm}
We will not go through the proof, as it is essentially the same as in the Hilbert space case from Theorem \ref{thm:wienerappr_ops}. Nevertheless, we want to note that the facts (iv) and (v) are clearly independent of the precise value of $p \in (1, \infty)$. Hence, we conclude that:
\begin{cor}\label{thm:wienerappr_pindep}
    Let $\mathcal R \subset \operatorname{Co}_1(U) \widehat{\otimes}_\pi \operatorname{Co}_1(U)$. Then, the properties (i)-(vii) of Theorem \ref{thm:wienerappr_ops_pdependent} are independent of the value $p\in (1, \infty)$. In particular, they hold true if and only if they hold true for $p = 2$.
\end{cor}
\begin{ex}
    \begin{enumerate}
        \item The considerations in this section cover the operator theory on the Fock spaces $F_t^p$ described in \cite{fulsche2020correspondence}: Indeed, the Fock spaces $F_t^p$ are essentially (up to multiplication by a certain Gaussian function) Wavelet transforms of the Coorbit spaces $\operatorname{Co}_p(U)$ with a Gaussian window, where $U$ is the representation on $\Xi = \mathbb R^{2n}$ with respect to the multiplier $m((x, \xi), (y, \eta)) = e^{-ix\eta + i\sigma((x, \xi), (y, \eta))}$ and $\sigma$ the standard symplectic form. Hence, the operator theory on $p$-Fock spaces developed in \cite{fulsche2020correspondence} (and continued in \cite{Fulsche2022}) falls within the framework of this section.
        \item When $G$ is a discrete group, $m$ the multiplier $m((x, \xi), (y, \eta)) = \overline{\langle x, \eta\rangle}$ on $\Xi = G \times \widehat{G}$ and $U$ the affiliated representation, then $\operatorname{Co}_p(U) = \ell^p(G)$, as is not hard to verify. In this case, the approach described here is another way towards working with band-dominated operators on sequence spaces (see e.g.~ \cite{Rabinovich_Roch_Silbermann2004} for results concerning such operators on $G = \mathbb Z^d$). This approach will be elaborated further in upcoming works.
    \end{enumerate}
\end{ex}

\section{Discussion}\label{sec:discussion}
We end this work by a short list of open problems:
\begin{question}
    In Section \ref{sec:repr}, we mentioned that square integrability of the $m$-representation obtained from Theorem \ref{Mackey-Stone-vNeumann} is equivalent to the property that the $m$-regular representation possesses and irreducible subrepresentation. Is this property always satisfied for a phase space $\Xi$ endowed with a Heisenberg multiplier $m$?
\end{question}

\begin{question}
    Parts of our theory hinged on the assumption of integrability of the projective representation. Nevertheless, we do not know if this assumption is ever violated. If there are instances of phase spaces $(\Xi, m)$ where the projective representation is not integrable, then the phase space cannot be of the standard form $\Xi = G \times \widehat{G}$ with $m$ similar to $m'((x, \xi), (y, \eta)) = \overline{\langle x, \eta\rangle}$. So the question is: Are there examples of phase spaces $(\Xi, m)$ for which the projective representation is not integrable?
\end{question}
\begin{question}
The following statement was proven in \cite[Corollary 4.5]{Vemuri2008}. For $\Xi = \mathbb R^d \times \widehat{\mathbb R^d}$ and $m((x, \xi), (y, \eta)) = e^{-ix\eta}$ the following holds true: If $A \in \mathcal T^1(\mathcal H)$ such that $\mathcal F_U(A)$ is compactly supported, then $\mathcal F_\sigma \mathcal F_U(A) \in L^1(\Xi)$. The proof of this statement hinged on rather explicit constructions. Is such a fact also true for more general phase spaces?
\end{question}
\begin{question}
    Is a statement converse to the previous result true, i.e.: If $f \in L^1(\Xi)$ such that $\mathcal F_\sigma(f) \in C_c(\widehat{\Xi})$, then $\mathcal F_U^{-1}(\mathcal F_\sigma(f)) \in \mathcal T^1(\mathcal H)$? For the specific setting $\Xi = G\times \widehat{G}$ and $m((x, \xi), (y, \eta)) = \langle x, \eta\rangle$, this fact holds true: If $f \in L^1(\Xi)$ with $\mathcal F_\sigma(f) \in C_c(\widehat{\Xi})$, then it is well-known that $f$ is contained in the Feichtinger algebra $\mathcal S_0(\Xi)$ (which follows, for example, from \cite[(vii) p. 26]{reiter71} and the fact that $\mathcal S_0(\Xi)$ is a Segal algebra). Now, it is not hard to see that an operator $A$ has its symbol in $\mathcal S_0(G \times \widehat{G})$ if and only if its integral kernel is in $\mathcal S_0(G \times G)$. Operators with such kernels are known to be of trace class, e.g.\ by \cite[Section 3.4]{Feichtinger_Jakobsen2022}.
\end{question}
\begin{question}
    Does the following generalization of Lemma \ref{lem:charl1} hold true? If $0 \leq f \in L^\infty(\Xi)$ such that $B \ast f \in \mathcal T^1(\mathcal H)$ for every $0 \leq B \in \mathcal T^1(\mathcal H)$, then there exists $f' \in L^1(\Xi)$ such that $f = f'$ locally null-a.e.?
\end{question}

\subsection*{Acknowledgements}
We would like to thank Reinhard F.\ Werner, Lauritz van Luijk, Jukka Kiukas and Franz Luef for interesting and helpful discussions. We also acknowledge the valuable suggestions made by the reviewers, which helped a lot in improving the manuscript. 

N.G. acknowledges financial support by the MICIIN with funding from European Union NextGenerationEU(PRTR-C17.I1) and by the Generalitat de Catalunya.

\bibliographystyle{abbrv}
\bibliography{refs}

\begin{thebibliography}{10}

\bibitem{Baggett_Kleppner1973}
L.~Baggett and A.~Kleppner.
\newblock Multiplier representations of abelian groups.
\newblock {\em J. Funct. Anal.}, 14:299--324, 1973.

\bibitem{Bauer_etAl2023}
W.~Bauer, L.~van Luijk, A.~Stottmeister, and R.~F. Werner.
\newblock {Self-adjointness of Toeplitz operators on the Segal-Bargmann space}.
\newblock {\em J. Funct. Anal.}, 284:109778, 2023.

\bibitem{Berge2022}
E.~Berge.
\newblock A primer on coorbit theory.
\newblock {\em J. Fourier Anal. Appl.}, 28:2, 2022.

\bibitem{Berge_Berge_Fulsche}
E.~Berge, S.~M. Berge, and R.~Fulsche.
\newblock A quantum harmonic analysis approach to {S}egal algebras.
\newblock {\em Integr. Equ. Oper. Theory}, 96, 2024.

\bibitem{Berge_Berge_Luef2022}
E.~Berge, S.~M. Berge, and F.~Luef.
\newblock The affine {W}igner distribution.
\newblock {\em Appl. Comput. Harmon. Anal.}, 56:150--175, 2022.

\bibitem{Berge_Berge_Luef_Skrettingland2022}
E.~Berge, S.~M. Berge, F.~Luef, and E.~Skrettingland.
\newblock {Affine quantum harmonic analysis}.
\newblock {\em J. Funct. Anal.}, 282:109327, 2022.

\bibitem{Bergh_Lofstrom1976}
J.~Bergh and J.~L\"{o}fstr\"{o}m.
\newblock {\em {Interpolation Spaces: An Introduction}}, volume 223 of {\em Die
  Grundlehren der mathematischen Wissenschaften in Einzeldarstellungen}.
\newblock Springer Verlag, 1976.

\bibitem{Cassinelli_DeVito_Toigo2003}
G.~Cassinelli, E.~De~Vito, and A.~Toigo.
\newblock Positive operator valued measures covariant with respect to an
  irreducible representation.
\newblock {\em J. Math. Phys.}, 44:4768–4775, 2003.

\bibitem{Christensen1996}
O.~Christiensen.
\newblock Atomic decomposition via projective group representations.
\newblock {\em Rocky Mountain J. Math.}, 26:1289–1312, 1996.

\bibitem{Dammeier_Werner2023}
L.~Dammeier and R.~F. Werner.
\newblock Quantum-classical hybrid systems and their quasifree transformations.
\newblock {\em Quantum}, 7, 2023.

\bibitem{Daubechies1980}
I.~Daubechies.
\newblock {On the distributions corresponding to bounded operators in the Weyl
  quantization}.
\newblock {\em Comm. Math. Phys.}, 75:229–238, 1980.

\bibitem{Nittis_Lein_Seri}
G.~De~Nittis, M.~Lein, and M.~Seri.
\newblock A magnetic pseudodifferential calculus for operator-valued and
  equivariant operator-valued symbols.
\newblock preprint available at arXiv:2210.05731.

\bibitem{Deitmar_Echterhoff2014}
A.~Deitmar and S.~Echterhoff.
\newblock {\em {Principles of Harmonic Analysis}}.
\newblock Springer Verlag, 2nd edition, 2014.

\bibitem{Diestel_Uhl1977}
J.~Diestel and J.~J. Uhl.
\newblock {\em {Vector Measures}}.
\newblock American Mathematical Society, 1977.

\bibitem{Digernes_Varadarajan2004}
T.~Digernes and V.~S. Varadarajan.
\newblock Models for the irreducible representation of a {H}eisenberg group.
\newblock {\em Infin. Dimens. Anal. Quantum Probab. Relat. Top.}, 7:527--546,
  2004.

\bibitem{Dixmier_1977}
J.~Dixmier.
\newblock {\em C*-Algebras}.
\newblock North-Holland, 1977.

\bibitem{Fawzi_Oufkir_Salzmann2024}
O.~Fawzi, A.~Oufkir, and R.~Salzmann.
\newblock Optimal fidelity estimation from binary measurements for discrete and
  continuous variable systems.
\newblock preprint available on arXiv:2409.04189, 2024.

\bibitem{Feichtinger1981}
H.~G. Feichtinger.
\newblock On a new {S}egal algebra.
\newblock {\em Monatsh. Math.}, 92:269--289, 1981.

\bibitem{Feichtinger_Grochenig1989}
H.~G. Feichtinger and K.~H. Gr\"{o}chenig.
\newblock Banach spaces related to integrable group representations and their
  atomic decompositions. {I}.
\newblock {\em J. Funct. Anal.}, 86:307–340, 1989.

\bibitem{Feichtinger_Jakobsen2022}
H.~G. Feichtinger and M.~S. Jakobsen.
\newblock The inner kernel theorem for a certain {S}egal algebra.
\newblock {\em Monatsh. Math.}, 198:675–715, 2022.

\bibitem{Feichtinger_Kozek1998}
H.~G. Feichtinger and W.~Kozek.
\newblock {Quantization of TF lattice-invariant operators on elementary LCA
  groups}.
\newblock In {\em {Gabor Analysis and Algorithms. Applied and Numerical
  Harmonic Analysis}}, pages 233--266. Birkh\"{a}user, 1998.

\bibitem{Folland2016}
G.~B. Folland.
\newblock {\em {A Course in Abstract Harmonic Analysis}}.
\newblock CRC Press, 2nd edition, 2016.

\bibitem{fulsche2020correspondence}
R.~Fulsche.
\newblock {Correspondence theory on $p$-Fock spaces with applications to
  Toeplitz algebras}.
\newblock {\em J. Funct. Anal.}, 279(7):108661, 2020.

\bibitem{Fulsche2022}
R.~Fulsche.
\newblock Toeplitz operators on non-reflexive {F}ock spaces.
\newblock {\em Rev. Mat. Iberoam.}, 40:1115–1148, 2024.

\bibitem{Fulsche_Hagger2023}
R.~Fulsche and R.~Hagger.
\newblock Quantum harmonic analysis for polyanalytic {F}ock spaces.
\newblock {\em J. Fourier Anal. Appl.}, 30(6):Paper No. 63, 2024.

\bibitem{Fulsche_Rodriguez2023}
R.~Fulsche and M.~A. Rodriguez~Rodriguez.
\newblock {Commutative G-invariant Toeplitz C$^\ast$ algebras on the Fock space
  and their Gelfand theory through Quantum Harmonic Analysis}.
\newblock to appear in J. Operator Theory, preprint available on
  arXiv:2307.15632.

\bibitem{Gaal1973}
S.~A. Gaal.
\newblock {\em {Linear Analysis and Representation Theory}}, volume 198 of {\em
  Grundlehren der mathematischen Wissenschaften}.
\newblock Springer, 1973.

\bibitem{Gibbons_Hoffman_Wootters_2004}
K.~S. Gibbons, M.~J. Hoffman, and W.~K. Wootters.
\newblock Discrete phase space based on finite fields.
\newblock {\em Phys. Rev. A}, 70:062101, Dec 2004.

\bibitem{Godement1947}
R.~Godement.
\newblock {Sur les relations d'orthogonalit\'{e} de V.\ Bargmann. II.
  Demonstration general\'{e}}.
\newblock {\em C. R. Acad. Sci. Paris}, 225:657–659, 1947.

\bibitem{Gohberg_Goldberg_Krupnik2000}
I.~Gohberg, S.~Goldberg, and N.~Krupnik.
\newblock {\em {Traces and Determinants of Linear Operators}}.
\newblock Birkh\"{a}user Basel, 2000.

\bibitem{Gross_2008}
D.~Gross.
\newblock {\em Computational power of quantum many-body states and some results
  on discrete phase spaces}.
\newblock PhD thesis, Imperial College London, 2008.

\bibitem{Halvdansson2022}
S.~Halvdansson.
\newblock Quantum harmonic analysis on locally compact groups.
\newblock {\em J. Funct. Anal.}, 285:110096, 2023.

\bibitem{Kastler1965}
D.~Kastler.
\newblock {The $C^\ast$-algebras of a free Boson field. I. Discussion of the
  basic facts.}
\newblock {\em Comm. Math. Phys.}, 1:14–48, 1965.

\bibitem{Kiukas2006}
J.~Kiukas.
\newblock Covariant observables on a nonunimodular group.
\newblock {\em J. Math. Anal. Appl.}, 324:491–503, 2006.

\bibitem{Kiukas_Lahti_Schultz_Werner2012}
J.~Kiukas, P.~Lahti, J.~Schultz, and R.~F. Werner.
\newblock Characterization of informational completeness for covariant phase
  space observables.
\newblock {\em J. Math. Phys.}, 53:102103, 2012.

\bibitem{Kiukas_Lahti_Ylinen2006}
J.~Kiukas, P.~Lahti, and K.~Ylinen.
\newblock Normal covariant quantization maps.
\newblock {\em J. Math. Anal. Appl.}, 319:783–801, 2006.

\bibitem{Loupias_MiracleSole1966}
G.~Loupias and S.~Miracle-Sole.
\newblock {$C^\ast$-alg\`{e}bres des syst\`{e}mes canoniques. I.}
\newblock {\em Comm. Math. Phys.}, 2:31–48, 1966.

\bibitem{Loupias_MiracleSole1967}
G.~Loupias and S.~Miracle-Sole.
\newblock {$C^\ast$-alg\`{e}bres des syst\`{e}mes canoniques. II.}
\newblock {\em Ann. Inst. H. Poincaré Sect. A (N.S.)}, 6:39–58, 1967.

\bibitem{Luef_Skrettingland2018a}
F.~Luef and E.~Skrettingland.
\newblock {Convolutions for Berezin quantization and Berezin-Lieb
  inequalities}.
\newblock {\em J. Math. Phys.}, 59:023502, 2018.

\bibitem{luef_skrettingland2018}
F.~Luef and E.~Skrettingland.
\newblock Convolutions for localization operators.
\newblock {\em J. Math. Pures Appl.}, 118:288--316, 2018.

\bibitem{Luef_Skrettingland2019a}
F.~Luef and E.~Skrettingland.
\newblock {Mixed-state localization operators: Cohen's class and trace class
  operators}.
\newblock {\em J. Fourier Anal. Appl.}, 25:2064–2108, 2019.

\bibitem{Luef_Skrettingland2021}
F.~Luef and E.~Skrettingland.
\newblock {A Wiener Tauberian theorem for operators and functions}.
\newblock {\em J. Funct. Anal.}, 280:108883, 2021.

\bibitem{Mackey1958}
G.~Mackey.
\newblock {Unitary representations of group extensions. I}.
\newblock {\em Acta. Math.}, 99:265–311, 1958.

\bibitem{Mantoiu_Purice2004}
M.~M\u{a}ntoiu and R.~Purice.
\newblock The magnetic {W}eyl calculus.
\newblock {\em J. Math. Phys.}, 45:1394--1417, 2004.

\bibitem{Pool1966}
J.~C.~T. Pool.
\newblock {Mathematical Aspects of the Weyl correspondence}.
\newblock {\em J. Mathematical Phys.}, 7:66–76, 1966.

\bibitem{Prasad_Shapiro_Vemuri2010}
A.~Prasad, I.~Shapiro, and M.~K. Vemuri.
\newblock Locally compact abelian groups with symplectic self-duality.
\newblock {\em Adv. Math.}, 225:2429--2454, 2010.

\bibitem{Rabinovich_Roch_Silbermann2004}
V.~Rabinovich, S.~Roch, and B.~Silbermann.
\newblock {\em Limit Operators and Their Applications in Operator Theory}.
\newblock Number 150 in {Operator Theory: Advances and Applications}.
  Birkh\"{a}user Basel, 2004.

\bibitem{Raussendorf_Browne_Delfosse_Okay_BermejoVega_2017}
R.~Raussendorf, D.~E. Browne, N.~Delfosse, C.~Okay, and J.~Bermejo-Vega.
\newblock Contextuality and {W}igner-function negativity in qubit quantum
  computation.
\newblock {\em Phys. Rev. A}, 95:052334, May 2017.

\bibitem{Reed_Simon1}
M.~Reed and B.~Simon.
\newblock {\em {Methods of Modern Mathematical Physics 1: Functional
  Analysis}}.
\newblock Academic Press, 2nd edition, 1981.

\bibitem{reiter71}
H.~Reiter.
\newblock {\em $L^1$-Algebras and Segal Algebras}, volume 231 of {\em Lecture
  Notes in Mathematics}.
\newblock Springer Verlag, 1971.

\bibitem{Ryan2002}
R.~A. Ryan.
\newblock {\em {Introduction to Tensor Products of Banach Spaces}}.
\newblock Springer Verlag, 2002.

\bibitem{Shucker1983}
D.~S. Shucker.
\newblock Square integrable representations of unimodular groups.
\newblock {\em Proc. Amer. Math. Soc.}, 89:169–172, 1983.

\bibitem{Stormer1974}
E.~St\o{}rmer.
\newblock {Positive linear maps in $C^\ast$ algebras}.
\newblock In A.~Hartk\"{a}mper and H.~Neumann, editors, {\em Foundations of
  quantum mechanics and ordered linear spaces}, volume~29 of {\em Lecture Notes
  in Physics}, pages 85--106. Springer Verlag, 1974.

\bibitem{Talagrand1984}
M.~Talagrand.
\newblock {\em Pettis integral and measure theory}.
\newblock Number 307 in Mem. Amer. Math. Soc. Amer. Math. Soc., 1984.

\bibitem{vonNeumann1931}
J.~v.~Neumann.
\newblock {Die Eindeutigkeit der Schr\"{o}dingerschen Operatoren}.
\newblock {\em Math. Ann.}, 104:570–578, 1931.

\bibitem{Vemuri2008}
M.~K. Vemuri.
\newblock Realizations of the canonical representation.
\newblock {\em Proc. Indian Acad. Sci., Math. Sci.}, 118:115--131, 2008.

\bibitem{Werner1984}
R.~F. Werner.
\newblock Quantum harmonic analysis on phase space.
\newblock {\em J. Math. Phys.}, 25(5):1404--1411, 1984.

\end{thebibliography}

\begin{multicols}{2}

\noindent
Robert Fulsche\\
\href{fulsche@math.uni-hannover.de}{\Letter ~fulsche@math.uni-hannover.de}
\\
\noindent
Institut f\"{u}r Analysis\\
Leibniz Universit\"at Hannover\\
Welfengarten 1\\
30167 Hannover\\
Germany\\

\noindent
Niklas Galke\\
\href{niklas.galke@uab.cat}{\Letter ~niklas.galke@uab.cat}
\\
\noindent
Department of Physics\\
Universitat Aut\`{o}noma de Barcelona\\
Edifici C\\
08193 Bellaterra (Cerdanyola del Vall\`{e}s)\\
Spain

\end{multicols}

\end{document}